\documentclass[11pt]{article}

\usepackage[margin=1in]{geometry}
\usepackage{amssymb, color}
\usepackage{hyperref}
\usepackage{graphicx,xcolor}
\usepackage{amsthm,amsfonts,euscript,amsmath,comment,slashed,here,todonotes}
\usepackage[affil-it]{authblk}

\newtheorem{theorem}{Theorem}[section]
\newtheorem{lemma}[theorem]{Lemma}
\newtheorem{assumption}[theorem]{Assumption}
\newtheorem{proposition}{Proposition}[section]
\newtheorem{corollary}{Corollary}[section]

\theoremstyle{definition}

\theoremstyle{remark}
\newtheorem{remark}[theorem]{Remark}

\numberwithin{equation}{section}

\newcommand\be{\begin{equation}}
\newcommand\ee{\end{equation}}
\newcommand\bea{\begin{eqnarray}}
\newcommand\eea{\end{eqnarray}}
\newcommand\bi{\begin{itemize}}
\newcommand\ei{\end{itemize}}
\newcommand\ben{\begin{enumerate}}
\newcommand\bena{\begin{enumerate}[(a)]}
\newcommand\een{\end{enumerate}}

\newcommand\bp{\begin{proof}}
\newcommand\ep{\end{proof}}

\newcommand{\ud}{\,\mathrm{d}}
\newcommand{\p}{\ensuremath{\partial}}
\newcommand{\n}{\ensuremath{\nonumber}}
\newcommand{\eps}{\ensuremath{\varepsilon}}

\newcommand{\sameer}[1]{\textcolor{magenta}{#1}}

\allowdisplaybreaks

\title{Global-in-$x$ Stability of Steady Prandtl Expansions \\ for 2D Navier-Stokes Flows}
\author{Sameer Iyer \footnote{Department of Mathematics, Princeton University,  Fine Hall, 304 Washington Rd, Princeton, NJ 08544, USA. Partially supported by NSF grant DMS-1802940.  \url{ssiyer@math.princeton.edu}} \qquad Nader Masmoudi \footnote{ NYUAD Research Institute, New York University Abu Dhabi, PO Box 129188, Abu Dhabi,   United Arab Emirates.
Courant Institute of Mathematical Sciences, New York University, 251 Mercer Street, New York, NY 10012, USA,  }}
\begin{document}

\maketitle

\begin{abstract}
In this work, we establish the convergence of 2D, stationary Navier-Stokes flows, $(u^\eps, v^\eps)$ to the classical Prandtl boundary layer, $(\bar{u}_p, \bar{v}_p)$, posed on the domain $(0, \infty) \times (0, \infty)$:
\begin{align*}
\| u^\eps - \bar{u}_p \|_{L^\infty_y} \lesssim \sqrt{\eps} \langle x \rangle^{- \frac 1 4 + \delta},   \qquad \| v^\eps - \sqrt{\eps} \bar{v}_p \|_{L^\infty_y} \lesssim \sqrt{\eps} \langle x \rangle^{- \frac 1 2}.
\end{align*}
This validates Prandtl's boundary layer theory \textit{globally} in the $x$-variable for a large class of boundary layers, including the entire one parameter family of the classical Blasius profiles, with sharp decay rates. The result demonstrates asymptotic stability in two senses simultaneously: (1) asymptotic as $\eps \rightarrow 0$ and (2) asymptotic as $x \rightarrow \infty$. In particular, our result provides the first rigorous confirmation for the Navier-Stokes equations that the boundary layer cannot ``separate" in these stable regimes, which is very important for physical and engineering applications. 
\end{abstract}

\section{Introduction}

\hspace{5 mm} A major problem in mathematical fluid mechanics is to describe the inviscid limit of Navier-Stokes flows in the presence of a boundary. This is due to the mismatch of boundary conditions for the Navier-Stokes velocity field (the ``no-slip" or Dirichlet boundary condition), and that of a generic Euler velocity field (the ``no penetration" condition). In order to aptly characterize the inviscid limit (in suitably strong norms), in \cite{Prandtl} Prandtl proposed, in the precise setting of 2D, stationary flows considered here, the existence of a thin ``boundary layer", $(\bar{u}_p, \bar{v}_p)$, which transitions the Dirichlet boundary condition to an outer Euler flow. 

The introduction of the Prandtl ansatz has had a monumental impact in physical and engineering applications, specifically in the 2D, steady setting, which is used to model flows over an airplane wing, design of golf balls, etc... (see \cite{Schlicting}, for instance). However, its mathematical validity has largely been in question since its inception. Validating the Prandtl ansatz is an issue of \textit{asymptotic stability of the profiles $(\bar{u}_p, \bar{v}_p)$}: $u^\eps \rightarrow \bar{u}_p$ and $v^\eps \rightarrow \bar{v}_p$ as the viscosity, $\eps$, tends to zero (again, in an appropriate sense). Establishing this type of stability (or instability) has inspired several works, which we shall detail in Section \ref{existing}. The main purpose of this work is to provide an affirmation of Prandtl's ansatz, in the precise setting of his seminal 1904 work (2D, stationary flows), most notably \textit{globally} in the variable, $x$, (with asymptotics as $x \rightarrow \infty$) which plays the role of a ``time" variable in this setting. 

The role that $x$ plays as a ``time" variable will be discussed from a mathematical standpoint in Section \ref{subsection:asy:x}. Physically, the importance of this ``time" variable dates back to Prandtl's original work, in which he says:  

\begin{quote}
``The most important practical result of these investigations is that, in certain cases, the flow separates from the surface at a point [$x_\ast$] entirely determined by external conditions... As shown by closer consideration, the necessary condition for the separation of the flow is that there should be a pressure increase along the surface in the direction of the flow." (L. Prandtl, 1904, \cite{Prandtl})
\end{quote}

In this work, we provide the first rigorous confirmation, for the Navier-Stokes equations, that in the  conjectured stable regime (in the absence of a pressure increase), the flow does not separate as $x \rightarrow \infty$. Rather, we prove that the flow relaxes back to the classical self-similar Blasius profiles, introduced by H. Blasius in \cite{Blasius}, which we also introduce and discuss in Section \ref{subsection:asy:x}.   

\subsection{The Setting}

\hspace{5 mm} First, we shall introduce the particular setting of our work in more precise terms. We consider the Navier-Stokes (NS) equations posed on the domain $\mathcal{Q} := (0, \infty) \times (0, \infty)$: 
\begin{align} \label{NS:1}
&u^\eps u^\eps_x + v^\eps u^\eps_Y + P^\eps_x = \eps \Delta u^\eps, \\ \label{NS:2}
&u^\eps v^\eps_x + v^\eps v^\eps_Y + P^\eps_Y = \eps \Delta v^\eps, \\ \label{NS:3}
&u^\eps_x + v^\eps_Y = 0,
\end{align}
We are taking the following boundary conditions in the vertical direction
\begin{align} \label{BC:vert:intro}
&[u^\eps, v^\eps]|_{Y = 0} = [0,0], \qquad [u^\eps(x,Y), v^\eps(x,Y)] \xrightarrow{Y \rightarrow \infty} [u_E(x,\infty), v_E(x,\infty)].
\end{align}
which coincide with the classical no-slip boundary condition at $\{Y =0\}$ and the Euler matching condition as $Y \uparrow \infty$. We now fix the vector field 
\begin{align}
[u_E, v_E] := [1, 0], \qquad P_E  = 0
\end{align}
as a solution to the steady, Euler equations ($\eps = 0$ in \eqref{NS:1} - \eqref{NS:3}), upon which the matching condition above reads $[u^\eps, v^\eps] \xrightarrow{Y \rightarrow \infty} [1, 0]$.  

Generically, there is a mismatch at $Y = 0$ between the boundary condition \eqref{BC:vert:intro} and that of an inviscid Eulerian fluid, which typically satisfies the no-penetration condition, $v^E|_{Y = 0} = 0$. Given this, it is not possible to demand convergence of the type $[u^\eps, v^\eps] \rightarrow [1, 0]$ as $\eps \rightarrow \infty$ in suitably strong norms, for instance in the $L^\infty$ sense. To rectify this mismatch, Ludwig Prandtl proposed in his seminal 1904 paper, \cite{Prandtl}, that one needs to \textit{modify} the limit of $[u^\eps, v^\eps]$ by adding a corrector term to $[1, 0]$, which is effectively supported in a thin layer of size $\sqrt{\eps}$ near $\{Y = 0\}$. Mathematically, this amounts to proposing an asymptotic expansion of the type 
\begin{align} \label{ansatz:1:1}
u^\eps(x, Y) = 1 + u^0_p(x, \frac{Y}{\sqrt{\eps}}) + O(\sqrt{\eps}) = \bar{u}^0_p(x, \frac{Y}{\sqrt{\eps}}) + O(\sqrt{\eps}), 
\end{align} 
where the rescaling $\frac{Y}{\sqrt{\eps}}$ ensures that the corrector, $u^0_p$, is supported effectively in a strip of size $\sqrt{\eps}$. The quantity $\bar{u}^0_p$ is classically known as the Prandtl boundary layer, whereas the $O(\sqrt{\eps})$ term will be referred to in our paper as ``the remainder". Motivated by this ansatz, we introduce the Prandtl rescaling 
\begin{align}
y := \frac{Y}{\sqrt{\eps}}. 
\end{align}
We now rescale the solutions via 
\begin{align}
U^\eps(x, y) := u^\eps(x, Y), \qquad V^\eps(x, y) := \frac{v^\eps(x, Y)}{\sqrt{\eps}}
\end{align}
which satisfy the following system 
\begin{align} \label{eq:NS:1}
&U^\eps U^\eps_x + V^\eps U^\eps_y + P^\eps_x = \Delta_\eps U^\eps, \\  \label{eq:NS:2}
&U^\eps V^\eps_x + V^\eps V^\eps_y + \frac{P^\eps_y}{\eps} = \Delta_\eps V^\eps, \\  \label{eq:NS:3}
&U^\eps_x + V^\eps_y = 0. 
\end{align}

\noindent Above, we have denoted the scaled Laplacian operator, $\Delta_\eps := \p_y^2 + \eps \p_x^2$. 

The effectiveness of the ansatz \eqref{ansatz:1:1}, and the crux of Prandtl's revolutionary idea, is that the leading order term $\bar{u}^0_p$ (and its divergence-free counterpart, $\bar{v}^0_p$) satisfy a much simpler equation than the full Navier-Stokes system, known as the Prandtl system, 
\begin{align}  \label{BL:0:intro:intro}
&\bar{u}^0_p \p_x \bar{u}^0_p + \bar{v}^0_p \p_y \bar{u}^0_p - \p_y^{2} \bar{u}^0_p + P^{0}_{px} = 0, \qquad P^0_{py} = 0, \qquad \p_x \bar{u}^0_p + \p_y \bar{v}^0_p = 0,  
\end{align}
which are supplemented with the boundary conditions
\begin{align} \label{BL:1:intro:intro}
&\bar{u}^0_p|_{x = 0} = \bar{U}^0_p(y), \qquad \bar{u}^0_p|_{y = 0} = 0, \qquad \bar{u}^0_p|_{y = \infty} = 1, \qquad \bar{v}^0_p = - \int_0^y \p_x \bar{u}^0_p. 
\end{align}
This system is simpler than \eqref{eq:NS:1} - \eqref{eq:NS:3} in several senses. First, due to the condition $P^0_{py} = 0$, we obtain that the pressure is constant in $y$ (and then in $x$ due to the Bernoulli's equation), and hence \eqref{BL:0:intro:intro} is really a scalar equation. 

In addition to this, by temporarily omitting the transport term $\bar{v}^0_p \p_y \bar{u}^0_p$, one can make the formal identification that $\bar{u}^0_p \p_x \approx \p_{yy}$, which indicates that \eqref{BL:1:intro:intro} is really a \textit{degenerate, parabolic} equation, which is in stark contrast to the elliptic system \eqref{eq:NS:1} - \eqref{eq:NS:3}. From this perspective, $x$ acts as a time-like variable, whereas $y$ acts as a space-like variable. We thus treat \eqref{BL:0:intro:intro} - \eqref{BL:1:intro:intro} as one would a typical Cauchy problem. Indeed, one can ask questions of local (in $x$) wellposedness, global (in $x$) wellposedness, finite-$x$ singularity formation, decay and asymptotics, etc,... This perspective will be emphasized more in Section \ref{subsection:asy:x}.

\subsection{Asymptotics as $\eps \rightarrow 0$ (Expansion \& Datum)} \label{section:viscosity}
 We expand the rescaled solution as 
\begin{align}
\begin{aligned} \label{exp:base}
&U^\eps := 1 + u^0_p + \sum_{i = 1}^{N_{1}} \eps^{\frac{i}{2}} (u^i_E + u^i_P) + \eps^{\frac{N_2}{2}} u =: \bar{u} +  \eps^{\frac{N_2}{2}} u \\
&V^\eps := v^0_p + v^1_E + \sum_{i = 1}^{N_{1}-1} \eps^{\frac i 2} (v^i_P + v^{i+1}_E) + \eps^{\frac{N_{1}}{2}} v^{N_{1}}_p +  \eps^{\frac{N_2}{2}} v =: \bar{v} + \eps^{\frac{N_2}{2}} v, \\
&P^\eps := \sum_{i = 0}^{N_1+1} \eps^{\frac{i}{2}} P^i_p +  \sum_{i = 1}^{N_{1}} \eps^{\frac{i}{2}} P^i_E + \eps^{\frac{N_2}{2}} P = \bar{P} + \eps^{\frac{N_2}{2}}P.
\end{aligned}
\end{align}

\noindent Above, 
\begin{align}
[u^0_E, v^0_E] := [1, 0], \qquad [u^i_p, v^i_p] = [u^i_p(x, y), v^i_p(x, y)], \qquad [u^i_E, v^i_E] = [u^i_E(x, Y), v^i_E(x, Y)], 
\end{align}
and the expansion parameters $N_{1}, N_2$ will be specified in Theorem \ref{thm.main} for the sake of precision. We note that these are not optimal choices of these parameters, but we chose them large for simplicity. Certainly, it will be possible to bring these numbers significantly smaller.  

First, we note that, given the expansion \eqref{exp:base}, we enforce the vertical boundary conditions for $i = 0,...N_1$, $j = 1, ... N_1$,
\begin{align}
&u^i_p(x, 0) = - u^i_E(x, 0),  & &v^j_E(x, 0) = - v^{j-1}_p(x, 0), & &[u, v]|_{y = 0} = 0, \\
&v^i_p(x, \infty) = u^i_p(x, \infty) = 0, & & u^j_E(x, \infty) = v^j_E(x, \infty) = 0, & &[u, v]|_{y = \infty} = 0. 
\end{align}
which ensure the no-slip boundary condition for $[U^\eps, V^\eps]$ at $y = 0, \infty$. As we do not include an Euler field to cancel out last $v^{N_1}_p$, we need to also enforce the condition $v^{N_1}_p|_{y = 0} = 0$, which is a particularity for the $N_1$'th boundary layer. 

The side $\{x =0\}$ has a distinguished role in this setup as being where ``in-flow" datum is prescribed. This datum is prescribed ``at the level of the expansion", in the sense described below (see the discussion surrounding equations \eqref{datum:given}). Moreover, as $x \rightarrow \infty$, one expects the persistence of just one quantity from \eqref{exp:base}, which is the leading order boundary layer, $[\bar{u}^0_p, \bar{v}^0_p]$, and the remaining terms from \eqref{exp:base} are expected to decay in $x$. This decay will be established rigorously in our main result. 

As is standard in this type of problem, the expansion \eqref{exp:base} is part of the ``prescribed data". Indeed, the point is that we assume that the Navier-Stokes velocity field $[U^\eps, V^\eps]$ attains the expansion \eqref{exp:base} initially at $\{x = 0\}$, and then aim to prove that this expansion propagates for $x > 0$. Given this, there are two categories of ``prescribed data" for the problem at hand: 
\begin{itemize}

\item[(1)] the inviscid Euler profiles,

\item[(2)] initial datum at $\{x = 0\}$ (for relevant quantities from \eqref{exp:base}).
\end{itemize}

We shall now describe the datum that we take for our setting.  First, we take the inviscid Euler profile to be 
\begin{align} \label{choice:Euler}
[u^0_E, v^0_E, P^0_E] := [1, 0, 0]
\end{align}
as given. We have selected \eqref{choice:Euler} as the simplest shear flow with which to work. Our analysis can be extended with relatively small (but cumbersome) modifications to general Euler shear flows of the form 
\begin{align}
[u^0_E, v^0_E, P^0_E] := [u^0_E(Y), 0, 0],
\end{align}
under mild assumptions on the shear profile $u^0_E(Y)$. 

Apart from prescribing the outer Euler profile, we also get to pick ``Initial datum", that is at $\{x = 0\}$ of various terms in the expansion \eqref{exp:base}. Specifically, the prescribed initial datum comes in the form of the functions: 
\begin{align} \label{datum:given}
u^i_P|_{x = 0} =: U^i_P(y), \qquad v^j_E|_{x = 0} =: V^j_E(Y),
\end{align}
for $i = 0,...,N_1$, and $j = 1,..,N_1$. Note that we do not get to prescribe, at $\{x = 0\}$, all of the terms appearing in \eqref{exp:base}. On the one hand, to construct $[u^i_p, v^i_p]$, we use that   $u^i_p$ obeys a degenerate parabolic equation, with $x$ occupying the time-like variable and $y$ the space-like variable and that 
  $v^i_p$ can be  recovered from $u^i_p$ via the divergence-free condition. Therefore, only $u^i_p|_{x = 0}$ is necessary to determine these quantities.  

On the other hand,  to construct  the Euler profiles $[u^i_E, v^i_E]$ for $i = 1,..,N_1$, we use an  elliptic problem for $v^i_E$ (in the special case of \eqref{choice:Euler}, it is in fact $\Delta v^i_E = 0$). As such, we prescribe the datum for $v^i_E$, as is displayed in \eqref{datum:given}, and then recover $u^i_E$ via the divergence-free condition. Therefore, only $v^i_E|_{x = 0}$ is necessary to determine these quantities. 

We wish to emphasize that although \eqref{exp:base} is a classical ansatz (in fact, Prandtl's original proposal \cite{Prandtl} was in precisely this setting of 2D stationary flows over a plate), the rigorous justification of \eqref{exp:base} in the steady setting has only been obtained very recently, see \cite{GI1}, \cite{GI3}, \cite{Varet-Maekawa}, \cite{GN} for instance. Moreover, all of these works require crucially that $x << 1$ (the analog of ``short-time"), in order to extract a subtle and delicate stability mechanism as $\eps \rightarrow 0$ of $U^\eps \rightarrow \bar{u}$. This stability mechanism was based, in \cite{GI1}, \cite{GI3}, \cite{GN}, on virial-type estimates for certain carefully selected weighted quantities \textit{that are compatible with the system aspect of the Navier-Stokes equations.} In the case of \cite{Varet-Maekawa}, it was based on a novel Rayleigh-Airy-Airy iteration, which also required $x << 1$. Prior to the present work, it is far from clear how to extract a stability mechanism (again, $U^\eps \rightarrow \bar{u}$ as $\eps \rightarrow 0$) that applies for even $x \sim O(1)$, let alone as $x \rightarrow \infty$ with precise asymptotics.

An alternative point of view is on the Fourier side in the tangential variable $x$. The results \cite{GI1}, \cite{GI3}, \cite{Varet-Maekawa}, \cite{GN} extract a stability mechanism, $U^\eps \rightarrow \bar{u}$, which relies crucially on including only high frequencies in $x$ (due to the requirement of $x << 1$) and excluding even frequencies of order $1$. Prior to our work, it was not known whether these known stability mechanisms can work in the presence of frequencies of order $1$ ($x \sim 1$), let alone allowing for zero frequency issues ($x \rightarrow \infty$).

\subsection{Asymptotics as $x \rightarrow \infty$} \label{subsection:asy:x}

We will now discuss more precisely the role of the $x$-variable, specifically emphasizing the role that $x$ plays as a ``time-like" variable, controlling the ``evolution" of the fluid. The importance of studying the large $x$ behavior of both the Prandtl equations and the Navier-Stokes equations is not just mathematical (in analogy with proving global wellposedness/ decay versus finite-$x$ blowup), but is also of importance physically due to the possibility of boundary layer separation, which is a large $x$ phenomena (which was noted by Prandtl himself in his original 1904 paper). 

We shall discuss first the large-$x$ asymptotics at the level of the Prandtl equations, \eqref{BL:0:intro:intro} - \eqref{BL:1:intro:intro}, which govern $[\bar{u}^0_p, \bar{v}^0_p]$. It turns out that there are two large-$x$ regimes for $[\bar{u}^0_p, \bar{v}^0_p]$, depending on the sign of the Euler pressure gradient: 

\begin{itemize}
\item[(1)] Favorable pressure gradient, $\p_x P^E \le 0$: $[\bar{u}^0_p, \bar{v}^0_p]$ exists globally in $x$, and becomes asymptotically self-similar, 

\item[(2)] Unfavorable pressure gradient, $\p_x P^E > 0$, $[\bar{u}^0_p, \bar{v}^0_p]$ may form a finite-$x$ singularity, known as ``separation".
\end{itemize}
In this work, our choice of $[1, 0]$ for the outer Euler flow guarantees that we are in setting (1), that of a favorable pressure gradient. 

This dichotomy above was introduced by Oleinik, \cite{Oleinik}, \cite{Oleinik1}, who established the first mathematically rigorous results on the Cauchy problem \eqref{BL:0:intro:intro} - \eqref{BL:1:intro:intro}. Indeed, Oleinik established that solutions to  \eqref{BL:0:intro:intro} - \eqref{BL:1:intro:intro} are locally (in $x$) well-posed in both regimes (1) and (2), and globally well-posed in regime (1) (under suitable hypothesis on the datum, which we do not delve into at this stage). 

Now, we investigate what  it means for $[\bar{u}^0_p, \bar{v}^0_p]$ to become asymptotically self-similar.  In order to describe this behavior more quantitatively, we need to introduce the Blasius solutions. Four years after Prandtl's seminal 1904 paper, H. Blasius introduced the (by now) famous ``Blasius boundary layer" in \cite{Blasius}, which takes the following form 
\begin{align} \label{Blasius:1}
&[\bar{u}_\ast^{x_0}, \bar{v}_\ast^{x_0}] = [f'(z), \frac{1}{\sqrt{x + x_0}} (z f'(z) - f(z)) ], \\ \label{Blasius:2}
&z := \frac{y}{\sqrt{x + x_0}} \\ \label{Blasius:3}
&ff'' + f''' = 0, \qquad f'(0) = 0, \qquad f'(\infty) = 1, \qquad \frac{f(z)}{z} \xrightarrow{\eta \rightarrow \infty} 1, 
\end{align}
where above, $f' = \p_z f(z)$ and  $x_0$ is a free parameter. Physically, $x_0$ has the meaning that at $x = - x_0$, the fluid interacts with the ``leading edge" of, say, a plate (hence the singularity at $x = - x_0$). Our analysis will treat any fixed $x_0 > 0$ (one can think, without loss of generality, that $x_0 = 1$). In fact, we will make the following notational convention which enables us to omit rewriting $x_0$ repeatedly:
\begin{align}
[\bar{u}_\ast, \bar{v}_\ast] := [\bar{u}_\ast^{1}, \bar{v}_\ast^{1}]. 
\end{align}
We emphasize that the choice of $1$ above could be replaced with any positive number, without loss of generality. 

The Blasius solutions, $[\bar{u}_\ast^{x_0}, \bar{v}_\ast^{x_0}]$ are distinguished solutions to the Prandtl equations in several senses. First, physically, they have demonstrated remarkable agreement with experiment (see \cite{Schlicting} for instance). Mathematically, their importance is two-fold. First, they are self-similar, and second, they act as large-$x$ attractors for the Prandtl dynamic. Indeed, a classical result of Serrin, \cite{Serrin}, states: 
\begin{align} \label{simple:Prandtl}
\lim_{x \rightarrow \infty} \| \bar{u}^0_p - \bar{u}_\ast \|_{L^\infty_y} \rightarrow 0 
\end{align}
for a general class of solutions, $[\bar{u}^0_p, \bar{v}^0_p]$ of \eqref{BL:0:intro:intro}.

This was revisited by the first author in the work \cite{IyerBlasius}, who established a refined description of the above asymptotics, in the sense 
\begin{align} \label{asy:blas:1}
\| \bar{u}^0_p - \bar{u}_\ast \|_{L^\infty_y} \lesssim o(1) \langle x \rangle^{- \frac 1 2 + \sigma_\ast}, \text{ for any } 0 < \sigma_\ast << 1,  
\end{align}
which is the essentially optimal decay rate from the point of view of regarding $\bar{u}^0_p$ as a parabolic equation with one spatial dimension. The work of \cite{IyerBlasius} used energy methods and virial-type identities, whereas the work of \cite{Serrin} was based on maximum principle methods. 

The case of (2) above (the setting of unfavorable pressure gradient) has been treated in the work of \cite{MD} as well as in the paper of \cite{Zhangsep} for the Prandtl equation with $\p_x P^E > 0$ (which appears as a forcing term on the right-hand side of \eqref{BL:0:intro:intro} - \eqref{BL:1:intro:intro} in their setting). These results establish the physically important phenomenon of separation, which occurs when $\p_y \bar{u}^0_p(x, 0) = 0$ for some $x > 0$, even though the datum starts out with the monotonicity $\p_y \bar{u}^0_p(0, 0) > 0$. 

Our main theorem below establishes two pieces of asymptotic information as $x \rightarrow \infty$:  
\begin{align} \label{main:estimate}
\| u^\eps - \bar{u}^0_p \|_{L^\infty_y} \lesssim \sqrt{\eps} \langle x \rangle^{- \frac 1 4 + \sigma_\ast}, \qquad \| \bar{u}^0_p - \bar{u}_\ast \|_{L^\infty_y} \le o(1) \langle x \rangle^{- \frac 1 4 + \sigma_\ast}.
\end{align}
This means that the full Navier-stokes velocity field undergoes this type of stabilization as $x \rightarrow \infty$. As a by-product, it recovers (using different techniques than \cite{Serrin} and \cite{IyerBlasius}) stability information of the form \eqref{asy:blas:1}. \textit{We emphasize that this is the first result which characterizes the asymptotic in $x$ behavior for the full Navier-Stokes velocity field.} 

We would also like to emphasize that establishing a large $x$ stabilization and decay mechanism for the Navier-Stokes equations, uniformly in the viscosity, requires a completely new functional framework for the Navier-Stokes equations which is absent from the Prandtl setting. Seeing as the stability for $x << 1$ (``short-time") of the boundary layer was only recently established in \cite{GI1}, \cite{GI3}, \cite{Varet-Maekawa}, it is far from clear that stability mechanism as $\eps \rightarrow 0$ (which from the various works \cite{GI1}, \cite{GI3}, \cite{Varet-Maekawa}, \cite{GN} is already quite involved and delicate for small $x << 1$) is at all consistent with the refined asymptotic in $x$ behavior predicted by \eqref{asy:blas:1}.

Therefore, we interpret our result, stated below in Theorem \ref{thm.main}, as a confluence of two stability mechanisms: stability as $U^\eps \rightarrow \bar{u}$, and the large $x$ stability observed in the substantially simpler Prandtl equations, \eqref{simple:Prandtl}. Prior to our work, it is far from clear how these two mechanisms interact (or if they even \textit{can} interact favorably), specifically because prior known results which have demonstrated stability $U^\eps \rightarrow \bar{u}$ require crucially that $x << 1$.

\begin{remark} One may notice when comparing the estimate \eqref{asy:blas:1} with the second estimate in \eqref{main:estimate} that the estimate from \eqref{main:estimate} is weaker by a decay factor of $\langle x \rangle^{- \frac 1 4}$. This is simply due to the fact that in the present paper, our aim is not to obtain the sharpest possible asymptotics of $\bar{u}^0_p \rightarrow \bar{u}_\ast$, but rather to close the first estimate from \eqref{main:estimate}. It is certainly possible to optimize our arguments for $\bar{u}^0_p$ to recover the sharper decay rate, but we have opted for simplicity in this matter, as the second estimate from \eqref{main:estimate} suffices to carry out our analysis of $u^\eps \rightarrow \bar{u}^0_p$.  
\end{remark}

\subsection{Main Theorem}

The main result of our work is the following. 

\begin{theorem} \label{thm.main} Fix $N_{1} = 400$ and $N_2 = 200$ in \eqref{exp:base}. Fix the leading order Euler flow to be 
\begin{align}
[u^0_E, v^0_E, P^0_E] := [1, 0, 0].
\end{align}

Assume the following pieces of initial data at $\{x = 0\}$ are given for $i = 0,...,N_{1}$, and $j = 1,...N_1$,
\begin{align} \label{datum:in}
 u^i_p|_{x = 0} =: U^i_p(y), && v^j_E|_{x = 0} =: V_E^j(Y)
\end{align}
where we make the following assumptions on the initial datum \eqref{datum:in}: 
\begin{itemize}
\item[(1)] For $i = 0$, the boundary layer datum $\bar{U}^0_p(y)$ is in a neighborhood of  Blasius, defined in \eqref{Blasius:1}. More precisely, we will assume 
\begin{align} \label{near:blasius}
\| (\bar{U}^0_p(y) - \bar{u}_\ast(0, y) ) \langle y \rangle^{m_0} \|_{C^{\ell_0}} \le \delta_\ast, 
\end{align}
where $0 < \delta_\ast << 1$ is small relative to universal constants, where $m_0, \ell_0$, are large, explicitly computable numbers. Assume also the difference $\bar{U}^0_p(y) - \bar{u}_\ast(0, y)$ satisfies generic parabolic compatibility conditions at $y = 0$.

\item[(2)] For $i = 1,..,N_1$, the boundary layer datum, $U^i_p(\cdot)$ is sufficiently smooth and decays rapidly:
\begin{align} \label{hyp:1}
\| U^i_p \langle y \rangle^{m_i} \|_{C^{\ell_i}} \lesssim 1,
\end{align}
where $m_i, \ell_i$ are large, explicitly computable constants (for instance, we can take $m_0 = 10,000$, $\ell_0 = 10,000$ and $m_{i+1} = m_i - 5$, $\ell_{i + 1} = \ell_i - 5$), and satisfies generic parabolic compatibility conditions at $y = 0$.

\item[(3)] The Euler datum $V^i_E(Y)$ satisfies generic elliptic compatibility conditions. 

\item[(4)] Assume Dirichlet datum for the remainders, that is 
\begin{align} \label{Dirichlet}
[u, v]|_{x = 0} = [u, v]|_{x = \infty} = 0.
\end{align}
\end{itemize}

Then there exists an $\eps_0 << 1$ small relative to universal constants such that for every $0 < \eps \le \eps_0$, there exists a unique solution $(u^\eps, v^\eps)$ to system  \eqref{NS:1} - \eqref{NS:3}, which satisfies the expansion \eqref{exp:base} in the quadrant, $\mathcal{Q}$. Each of the intermediate quantities in the expansion \eqref{exp:base} satisfies the following estimates for $i = 1,...,N_1$
\begin{align}
&\| \p_x^k \p_y^j u^i_p \langle z \rangle^M \|_{L^\infty_y} \le C_{M, k, j} \langle x \rangle^{- \frac 1 4 - k - \frac j 2 + \sigma_{\ast}} && \| v^i_p \|_{L^\infty_y} \le C_{M, k, j} \langle x \rangle^{- \frac 3 4 - k - \frac j 2 + \sigma_\ast} \\
&\| \p_x^k \p_Y^j u^i_E \|_{L^\infty_y} \le C_{k,j} \langle x \rangle^{- \frac 1 2 - k - j} && \| v^i_E \|_{L^\infty_y}  \le C_{k,j} \langle x \rangle^{- \frac 1 2 - k - j},
\end{align} 
where $\sigma_\ast := \frac{1}{10,000}$. Finally, the remainder $(u, v)$ exists globally in the quadrant, $\mathcal{Q}$, and satisfies the following estimates 
\begin{align}
\| u, v \|_{\mathcal{X}} \lesssim 1, 
\end{align}
where the space $\mathcal{X}$ will be defined precisely in \eqref{X:norm}. 
\end{theorem}
As an immediate corollary, we obtain the following asymptotics, which are valid uniformly for $\eps \le \eps_0$ and for all $x > 0$. 
\begin{corollary} The solution $(u^\eps, v^\eps)$ to \eqref{NS:1} - \eqref{NS:3} satisfies the following asymptotics 
\begin{align}
\| u^\eps - (1 + u^0_p) \|_{L^\infty_y} \lesssim \sqrt{\eps} \langle x \rangle^{- \frac 1 4 + \sigma_\ast}   \qquad \| v^\eps - \sqrt{\eps} (v^0_p + v^1_E) \|_{L^\infty_y} \lesssim \eps \langle x \rangle^{- \frac 1 2}, 
\end{align}
\end{corollary}

\begin{remark}[Generalizations] There are several ways in which our result, as stated in Theorem \ref{thm.main}, can easily be generalized, but we have foregone this generality for the sake of clarity and concreteness. We list these here. 
\begin{itemize}
\item[(1)] As stated, the leading order boundary layer, $\bar{u}^0_p$ will be near the self-similar Blasius profile. This is ensured by the assumption \eqref{near:blasius} (and proven rigorously in our construction). In reality, our proof uses only mild properties of the Blasius profile, and we can therefore generalize the type of boundary layers we consider. The most general class of boundary layers that we can treat would be those (globally defined) $\bar{u}^0_p$ that can be expressed as $\bar{u}^0_p = \tilde{u}^0_p + \hat{u}^0_{p}$, where $\tilde{u}_{pyy} \le 0$, $\tilde{u}^0_{py} \ge 0$, $\tilde{u}^0_p$ satisfies estimates \eqref{water:1} - \eqref{v:blasius:2}, and $\hat{u}^0_p$ satisfies estimates \eqref{water:65}.

\item[(2)] We have taken, again to make computations simpler, the leading order Euler vector field $[u_E, v_E] = [1, 0]$. It would not require new ideas to generalize our work to general shear flows $[u_E, v_E] = [b(Y), 0]$, where $b(Y)$ satisfies mild hypotheses (for instance, $b \in C^\infty, \frac 1 2 \le b \le \frac 3 2$, $\p_Y^k b(Y)$ decays rapidly as $Y \rightarrow \infty$ for $k \ge 1$), though it would make the expressions more complicated. The case of Euler flows that are not shear flows, however, poses more challenges and would require some new ideas.

\item[(3)] The datum assumed at $\{x = 0\}$ for the remainders, \eqref{Dirichlet}, can easily be generalized to any smooth vector-field, $[u, v]|_{x = 0} = [b_1(y), \sqrt{\eps} b_2(y)]$, where $b_1, b_2$ are smooth, rapidly decaying functions. 
\end{itemize} 
\end{remark}

It is convenient to think of Theorem \ref{thm.main} in two parts: the first is a result on the \textit{construction of the approximate solution}, $[\bar{u}, \bar{v}]$ from \eqref{exp:base}, given all of the necessary initial/ boundary datum described in Theorem \ref{thm.main}. The second is a result on \textit{asymptotic stability} of $[\bar{u}, \bar{v}]$, which amounts to controlling the differences from \eqref{exp:base}, $u := \eps^{- \frac{N_2}{2}} (U^\eps - \bar{u}), v := \eps^{- \frac{N_2}{2}} (V^\eps - \bar{v})$. We state these two steps as two distinct theorems, which, when combined, yield the result of Theorem \ref{thm.main}. 

\begin{theorem}[Construction of Approximate Solution] \label{thm:approx} Assume the boundary and initial conditions are prescribed as specified in Theorem \ref{thm.main}. Define $\sigma_\ast = \frac{1}{10,000}$. Then for $i = 1,...,N_1$, for $M \le m_i$ and $2k+j \le \ell_i$, the quantities $[u^i_p, v^i_p]$ and $[u^i_E, v^i_E]$ exist globally, $x > 0$, and the following estimates are valid
\begin{align} \label{blas:conv:1}
&\| \p_x^k \p_y^j ( \bar{u}^0_p - \bar{u}_\ast) \langle z \rangle^M \|_{L^\infty_y} \lesssim \delta_\ast \langle x \rangle^{- \frac 1 4 - k - \frac j 2 + \sigma_\ast}, \\ \label{blas:conv:2}
&\| \p_x^k \p_y^j ( \bar{v}^0_p - \bar{v}_\ast) \langle z \rangle^M \|_{L^\infty_y} \lesssim \delta_\ast \langle x \rangle^{- \frac 3 4 - k - \frac j 2 + \sigma_\ast} \\ \label{water:65}
&\| \langle z \rangle^M \p_x^k \p_y^j u_p^{i} \|_{L^\infty_y} \lesssim \langle x \rangle^{- \frac 1 4 - k - \frac j 2 + \sigma_\ast} \\ \label{water:54}
&\| \langle z \rangle^M \p_x^k \p_y^j v_p^{i} \|_{L^\infty_y} \lesssim \langle x \rangle^{- \frac 3 4 - k - \frac j 2 + \sigma_\ast}, \\  \label{water:78}
&\|(Y \p_Y)^l \p_x^k \p_Y^j v^{i}_E \|_{L^\infty_Y} \lesssim \langle x \rangle^{- \frac 1 2 - k - j}, \\ \label{water:88}
&\| (Y \p_Y)^l \p_x^k \p_Y^j u^{i}_E \|_{L^\infty_Y} \lesssim \langle x \rangle^{- \frac 1 2 - k - j}.
\end{align}
Moreover, the more precise estimates stated in Assumptions \eqref{assume:1} -- \eqref{assume:3} are valid. Finally, define the contributed forcing by:
\begin{align}
\begin{aligned} \label{forcing:remainder}
F_R := &(U^\eps U^\eps_x + V^\eps U^\eps_y + P^\eps_x - \Delta_\eps U^\eps ) - (\bar{u} \bar{u}_x + \bar{v} \bar{u}_y + \bar{P}_x - \Delta_\eps \bar{u}) \\
G_R := &(U^\eps V^\eps_x + V^\eps V^\eps_y + \frac{P^\eps_y}{\eps} - \Delta_\eps V^\eps ) -  (\bar{u} \bar{v}_x + \bar{v} \bar{v}_y + \frac{\bar{P}_y}{\eps} - \Delta_\eps \bar{v}). 
\end{aligned}
\end{align}
Then the following estimates hold on the contributed forcing: 
\begin{align} \label{est:forcings:part1}
\| \p_x^j \p_y^k F_R \langle x \rangle^{\frac{11}{20}+ j + \frac k 2} \| + \sqrt{\eps} \|  \p_x^j \p_y^k G_R \langle x \rangle^{\frac{11}{20}+ j + \frac k 2} \| \le \eps^{5}. 
\end{align}
\end{theorem}

\begin{theorem}[Stability of Approximate Solution]\label{thm:2} Assume the boundary and initial conditions are prescribed as specified in Theorem \ref{thm.main}. There exists a unique, global solution, $[u, v]$, to the problem \eqref{vel:form} - \eqref{vel:eqn:2}, where the modified unknowns $(U, V)$ defined through \eqref{vm:1} satisfy the estimate 
\begin{align}
\| U, V \|_{\mathcal{X}} \lesssim \eps^{5},
\end{align} 
where the $\mathcal{X}$ norm is defined precisely in \eqref{X:norm}. 
\end{theorem}

Of these two steps, by far the most challenging is the second step, the stability analysis of Theorem \ref{thm:2}. \textit{This paper is devoted exclusively to this stability analysis,} whereas the first step, construction of approximate solutions, is obtained in a companion paper, \cite{IM21}. As a result, for the remainder of this paper, we will take Theorem \ref{thm:approx} as given and use it as a ``black-box". This serves to more effectively highlight the specific new techniques we develop for this stability result.

\subsection{Existing Literature} \label{existing}

The boundary layer theory originated with Prandtl's seminal 1904 paper, \cite{Prandtl}. First, we would like to emphasize that this paper presented the boundary layer theory in precisely the present setting: for 2D, steady flows over a plate (at $Y = 0$). In addition, Prandtl's original paper discussed the physical importance of understanding \eqref{exp:base} for large $x$, due to the possibility of boundary layer separation.

We will distinguish between two types of questions that are motivated by the ansatz, \eqref{ansatz:1:1}. First, there are questions regarding the description of the leading order boundary layer, $[\bar{u}^0_p, \bar{v}^0_p]$, and second, there are questions regarding the study of the $O(\sqrt{\eps})$ remainder, which, equivalently, amounts to questions regarding the validity of the asymptotic expansion \eqref{ansatz:1:1}.

A large part of the results surrounding the system \eqref{BL:0:intro:intro} - \eqref{BL:1:intro:intro} were already discussed in Section \ref{subsection:asy:x}, although the results discussed there were more concerned with the large $x$ asymptotic  behavior. We point the reader towards \cite{MD} for a study of separation in the steady setting, using modulation and blowup techniques. For local-in-$x$ behavior, let us supplement the references from Section \ref{subsection:asy:x} with the results of \cite{GI2}, which established higher regularity for \eqref{BL:0:intro:intro} - \eqref{BL:1:intro:intro} through energy methods, and the recent work of \cite{Zhifei:smooth} which obtains global $C^\infty$ regularity using maximum principle methods. 

%\textcolor{blue}{Of these results, the most relevant to the present work is that of \cite{IyerBlasius}. Indeed, the work of \cite{IyerBlasius} was the first to introduce the cancellation for the linearized operator $\bar{u} u_x + \bar{u}_y v - u_{yy}$ that is due to the classical von-Mise change of variables in an energetic setting. 
%%  (and it was mentioned in this work that this cancellation be utilized as one ingredient to prove the present result). 
% This cancellation pointed out in the work \cite{IyerBlasius} is precisely what motivates our ``Good Variables", seen in \eqref{good:variables:2}. }

We now discuss the validity of the ansatz (\ref{ansatz:1:1}). While rigorous results on the Prandtl equation itself date all the way back to Oleinik, proving \textit{stability} of the boundary layer in the inviscid limit has proven to be substantially more difficult and has only recently been achieved in the steady setting. Moreover, all of the known stability results have been for $x << 1$ (the analog of ``short time").  The classical setting we consider here, with the no-slip condition, was first treated, locally in the $x$ variable by the works \cite{GI1} - \cite{GI2}, \cite{Varet-Maekawa}, and the related work of \cite{GI3}. These works of \cite{GI1} - \cite{GI2} are distinct from that of \cite{Varet-Maekawa} in the sense that the main concern of \cite{GI1} - \cite{GI2} are $x$-dependent boundary layer profiles, and in particular addresses the classical Blasius solution. On the other hand, the work of \cite{Varet-Maekawa} is mainly concerned with shear solutions $(U(y), 0)$ to the forced Prandtl equations (shears flows are not solutions to the homogeneous Prandtl equations), which allows for Fourier analysis in the $x$ variable. Both of these works are \textit{local-in-$x$} results, which can demonstrate the validity of expansion \eqref{exp:base} for $0 \le x \le L$, where $L << 1$ is small (but of course, fixed relative to the viscosity $\eps$). We also mention the recent work of \cite{Gao-Zhang}, which generalized the work of \cite{GI1} - \cite{GI2} to the case of Euler flows which are not shear for $0 < x < L << 1$. 

%The central idea in \cite{GI1} - \cite{GI2} is to introduce a quotient estimate to obtain coercivity of the linearized operator $\bar{u} u_x + \bar{u}_y v - u_{yy}$ which circumvents the degeneracy of $\bar{u}$ at $y = 0$ for $L << 1$. More specifically, the crucial quotient estimate introduced in \cite{GI1} - \cite{GI2} has the advantage of controlling the quantity $\p_y (\frac{v}{\bar{u}})|_{y = 0}$, which is a quantity had thus far been out of reach and a serious obstacle for controlling the remainder solution. 
%
%\textcolor{blue}{We note that for the present work, it is also the case that these same two ingredients play a role: (1) a  quotient-type estimate introduced by \cite{GI1} and (2) the cancellation introduced by \cite{IyerBlasius} are two of the ingredients that initiate the analysis. However, these two ingredients need to be supplemented with several more ideas to achieve our complete global-in-$x$ result (with, of course, asymptotics as $x \rightarrow \infty$). These ideas will all be discussed below in Section \ref{section:ideas}. }

We also point the reader towards the works \cite{GN}, \cite{Iyer}, \cite{Iyer2a} - \cite{Iyer2c}, and \cite{Iyer3}. All of these works are under the assumption of a moving boundary at $\{Y = 0\}$, which while they face the difficulty of having a transition from $Y = 0$ to $Y = \infty$, crucially eliminate the degeneracy of $\bar{u}^0_p$ at $\{Y =0\}$, which is a major difficulty posed by the boundary layer theory. The work of \cite{Iyer2a} - \cite{Iyer2c} is of relevance to this paper, as the question of global in $x$ stability was considered (again, under the assumption of a moving $\{Y = 0\}$ boundary, which significantly simplifies matters). 

For unsteady flows, there is also a  large literature studying  expansions of the type \eqref{exp:base}. We refrain from discussing this at too much length because the unsteady setting is quite different from the steady setting. Rather, we point the reader to (an incomplete) list of references. First, in the analyticity setting, for small time, the seminal works of \cite{Caflisch1}, \cite{Caflisch2} establish the stability of expansions \eqref{exp:base}. This was extended to the Gevrey setting in \cite{DMM}, \cite{DMM20}. The work of \cite{Mae} establishes stability under the assumption of the initial vorticity being supported away from the boundary.  The reader should also see the related works \cite{LXY}, \cite{Taylor}, \cite{TWang}, \cite{Wang}. 

When the regularity setting goes from analytic/ Gevrey to Sobolev, there have also been several works in the opposite direction, which demonstrate, again in the unsteady setting, that expansion of the type \eqref{exp:base} should not be expected. A few works in this direction are \cite{Grenier}, \cite{GGN1}, \cite{GGN2}, \cite{GGN3}, \cite{GN2}, as well as the remarkable series of recent works of \cite{GrNg1}, \cite{GrNg2}, \cite{GrNg3} which settle the question and establish invalidity in Sobolev spaces of expansions of the type (\ref{exp:base}). The related question of $L^2$ (in space) convergence of Navier-Stokes flows to Euler has been investigated by many authors, for instance in \cite{CEIV}, \cite{CKV}, \cite{CV}, \cite{Kato}, \cite{Masmoudi98}, and \cite{Sueur}. 

There is again the related matter of wellposedness of the unsteady Prandtl equation. This investigation was initiated by \cite{Oleinik}, who obtained global in time solutions on $[0, L] \times \mathbb{R}_+$ for $L << 1$ and local in time solutions for any $L < \infty$, under the crucial monotonicity assumption $\p_y u|_{t = 0} > 0$. The $L << 1$ was removed by \cite{Xin} who obtained global in time weak solutions for any $L < \infty$. These works relied upon the Crocco transform, which is only available under the monotonicity condition. Also under the monotonicity condition, but without using the Crocco transform, \cite{AL} and \cite{MW} obtained local in time existence. \cite{AL} introduced a Nash-Moser type iterative scheme, whereas \cite{MW} introduced a good unknown which enjoys an extra cancellation and obeys good energy estimates. The related work of \cite{KMVW} removes monotonicity and replaces it with multiple monotonicity regions.  

Without monotonicity of the datum, the wellposedness results are largely in the analytic or Gevrey setting. Indeed, \cite{GVDie},  \cite{GVM}, \cite{Vicol}, \cite{Kuka}, \cite{Lom}, \cite{LMY}, \cite{Caflisch1} - \cite{Caflisch2}, \cite{IyerVicol} are some results in this direction. Without assuming monotonicity, in Sobolev spaces, the unsteady Prandtl equations are, in general, illposed: \cite{GVD}, \cite{GVN}. Finite time blowup results have also been obtained in \cite{EE}, \cite{KVW}, \cite{Hunter}. Moreover, the issue of boundary layer separation in the unsteady setting has been tackled by the series of works \cite{Collot1}, \cite{Collot2}, \cite{Collot3} using modulation and blowup techniques.  

The above discussion is not comprehensive, and we have elected to provide a more in-depth of the steady theory due to its relevance to the present paper. We refer to the review articles, \cite{E}, \cite{Temam} and references therein for a more complete review of other aspects of the boundary layer theory. 

We would, however, like to point out that going from results on Prandtl (a \textit{significantly simplified model equation}) to a corresponding stability theorem about Navier-Stokes has proven to be highly nontrivial (and in fact, generically false in the unsteady setting, \cite{GrNg1} -- \cite{GrNg3}). In the time-dependent case, the well-known Tollmien-Schlichting instability (see for instance \cite{GGN2}) shows that there is a \textit{destabilizing effect of the viscosity}, which creates growing modes at the Navier-Stokes level (thus requiring Gevrey/ analytic spaces to prove Navier-Stokes to Prandtl type stability results). Part of the novelty of our work is to provide a framework which is robust enough to simultaneously (1) prove the stability of the boundary layer in Sobolev spaces as $\eps \rightarrow 0$ which, even for $x << 1$ in the steady setting, has been done in very few cases (\cite{GI1}, \cite{GI3}, \cite{Varet-Maekawa}, \cite{GN}) and (2) capture the detailed asymptotic in $x$ (analogous in this setting to ``large time") behavior that is known at the level of the Prandtl system.

It is also striking to compare our result to what is expected in the unsteady setting. In the unsteady setting, certain well-known ``geometric" criteria exist, for instance monotonicity, to ensure (local in time) wellposedness in Sobolev spaces for the Prandtl equations. However, these criteria (even one adds concavity) are not enough to establish a stability result of the type $U^\eps \rightarrow \bar{u}$ due to the destabilizations due to the viscosity. For this reason, the sharpest known stability results are in Gevrey spaces (see for instance \cite{DMM}, \cite{DMM20}). Remarkably, in the steady setting, we see that the \textit{same} geometric criteria which guarantee regularity of Prandtl equation (monotonicity and concavity) suffice to prove global stability \textit{in Sobolev spaces} of $U^\eps \rightarrow \bar{u}$.

As we will explain below, we will be seeing a ``destabilizing effect" of the Navier-Stokes equations, although it is distinct from the instabilities of the unsteady setting. In our setting, the Navier-Stokes system requires the appearance of certain singular quantities at $\{y = 0\}$ that do not appear for any analysis of the Prandtl equations. In turn, these quantities correspond to a loss of derivative (specifically of the derivative $x \p_x$) in our energy scheme, for instance seen through \eqref{turn:2} and \eqref{nonlinear:est:intro}. This destabilizing effect is due to the pressure in the Navier-Stokes system, and therefore would clearly not have been seen at the Prandtl level (for instance in \cite{IyerBlasius}). Since the apparent manifestation of this is a loss of $x \p_x$ (in some sense the natural scaling field in the tangential direction), this effect will also not have been observed in any of the $x << 1$ stability results (\cite{GI1}, \cite{Varet-Maekawa}, \cite{GN}). Given the (in a sense, \textit{expected}) destabilizing effect of viscosity, a striking aspect of our work is to, despite this, find a framework that establishes stability both as $\eps \rightarrow 0$ and as $x \rightarrow 0$ in Sobolev spaces. 

\subsection{Main Ingredients} \label{section:ideas}

We will now describe the main ideas that enter our analysis. In order to do so, we need to describe the equation that is satisfied by the remainders, $(u, v)$, in the expansion \eqref{exp:base}, which is the following, 
\begin{align} \label{vel:eqn:1:intro}
&\mathcal{L}[u, v] + \begin{pmatrix} \p_x \\ \frac{\p_y}{\eps} \end{pmatrix} P = \mathcal{N}(u,v) +  \text{Forcing }, \qquad u_x + v_y = 0, \text{ on } \mathcal{Q},
\end{align} 
where $F$ is a forcing term that exists due to the fact that $(\bar{u}, \bar{v})$ is not an exact solution to the Navier-Stokes equations, (we leave it undefined now, as it does not play a central role in the present discussion), and $\mathcal{N}(u, v)$ contain quadratic terms. The operator $\mathcal{L}[u, v]$ is a vector-valued linearized operator around $[\bar{u}, \bar{v}]$, and is defined precisely via  
\begin{align}
\begin{aligned} \label{vel:form:intro}
\mathcal{L}[u, v] := \begin{cases} \mathcal{L}_1 :=  \bar{u} u_x + \bar{u}_{y} v + \bar{u}_{x} u + \bar{v} u_y - \Delta_\eps u   \\ \mathcal{L}_2 := \bar{u} v_x + u \bar{v}_{x} + \bar{v} v_y + \bar{v}_{y} v - \Delta_\eps v.\end{cases}
\end{aligned}
\end{align}
The main goal in the study of \eqref{vel:eqn:1:intro} is to obtain an estimate of the form $\| u, v \|_{\mathcal{X}} \lesssim 1$, for an appropriately defined space $\mathcal{X}$ (which importantly needs to control the $L^\infty$ norm). We note that this is a variable-coefficients operator, for which Fourier analysis is not conducive for obtaining estimates. 

\subsubsection{The New Point of View}

We will discuss our point of view as compared to prior works on the Navier-Stokes to Prandtl stability (\cite{GI1}, \cite{GI3}, \cite{Varet-Maekawa}, \cite{GN}). These prior works all in the small $x$ regime (equivalently, the high-frequencies in $x$ regime). In some sense, these papers all introduce a transform or change of unknown which factors the vectorial Rayleigh operator, $\begin{pmatrix} \bar{u} \p_x u + \bar{u}_y v \\ \bar{u} \p_x v + \bar{v}_y v \end{pmatrix} = \begin{pmatrix} - \bar{u}^2 \p_y \{ \frac{v}{\bar{u}} \} \\ \bar{u}^2 \p_x \{ \frac{v}{\bar{u}} \} \end{pmatrix}$ as a (almost) divergence-form operator on a modified unknown, $\frac{v}{\bar{u}}$.  The primary purpose of the change of unknown is to avoid losing $x$ derivatives, (since high frequencies in $x$ were being considered). This gives one starting point to study the vector valued operator \eqref{vel:form:intro} (for instance, to design virial type multipliers), though this point of view breaks down once $x \sim 1$ (or equivalently, when $O(1)$ frequencies in $x$ are introduced). The breakdown occurs due to low-frequency commutators which arise in this process from the diffusive terms. 

The new point of view we introduce in this work is the following: there is a mechanism by which the linearized transport operator (Rayleigh) and the diffusion actually ``talk to each-other" in \eqref{vel:form:intro} and, in fact, produce a damping effect as $x \rightarrow \infty$. This mechanism is known in the significantly simpler setting of the Prandtl equations through a change of variables, ``the von-Mise transform", and will be discussed in Point (1) below. Most importantly, we develop a point of view on the Navier-Stokes equations which shows that this mechanism is consistent and cooperates with the mechanism of \textit{stability in $\eps \rightarrow 0$} of the boundary layer, which itself, even for $x << 1$, has only recently been understood (see \cite{GI1}, \cite{GI3}, \cite{Varet-Maekawa}, \cite{GN}) through a variety of techniques.

Prior works on the Navier-Stokes to Prandtl stability, \cite{GI1}, \cite{GI3}, \cite{Varet-Maekawa}, \cite{GN} do not capitalize on this feature because the von-Mise mechanism appears much too delicate to interact favorably with the pressure of Navier-Stokes (in fact, this difference between Prandtl and Navier-Stokes is a central simplification of the boundary layer theory)! Indeed, changing coordinates in the Navier-Stokes equations and trying to emulate a von-Mise type transform seems completely fruitless, as the new coordinate system interacts particularly poorly with the pressure terms. 

We establish in this work that indeed, the way these two operators ``talk to eachother" is robust enough to carry through to the Navier-Stokes setting, which distinguishes our work from the prior results on the stability of boundary layers. However, to carry through this mechanism to the Navier-Stokes setting, we need to take significant steps which are not present in (1) any prior analysis on boundary layer stability (for instance, \cite{GI1}, \cite{GI3}, \cite{Varet-Maekawa}, \cite{GN}) and (2) any Prandtl analysis (in particular, from \cite{IyerBlasius}). 

We first list these new steps, and then itemize them for a more thorough discussion. Our first task is to extract, in $(x, y)$ coordinate system, the main mechanism behind the Transport-Diffusion interaction of the von-Mise transform. This is discussed in point (1) below. We identify that the transform \textit{does not work} well with the pressure terms of Navier-Stokes due to a mismatch in the homogeneity of $\bar{u}$, which we correct in our ``renormalized unknowns" \eqref{good:variables:2}, which contain a renormalized version of the von-Mise velocity unknowns, \eqref{real:phi}. Even at this first step, we need to deviate from what is done at the Prandtl level, for instance in \cite{IyerBlasius}. We establish first that these renormalized velocities \textit{still keep enough of the cancellation enjoyed by the traditional von-Mise velocity, \eqref{real:phi}.}

This change in homogeneity, which is particular to the Navier-Stokes setting, subsequently requires the development of a completely new functional framework to close our analysis that has been absent from all prior works. The central details of this framework are discussed in points (4), (5), (6) below. At a very high-level, one can see that certain necessary estimates from \eqref{turn:1} -- \eqref{turn:3} lose derivatives in $x \p_x$, which is, again, absent from any prior Prandtl-type analyses and any local-in-$x$ Navier-Stokes analyses. The loss of derivative is created by excess powers of $\frac{1}{\bar{u}}$ in certain crucial quantities, which again, is due specifically to the Navier-Stokes setting, and can be thought of as an unavoidable ``destabilizing effect" of the viscosity (which is well-known in the unsteady setting and actually precludes a result of the type we are proving from holding in the time-dependent setting). Given this well-known \textit{destabilizing effect of viscosity} from other settings, it is completely non-obvious that this type of loss can be recovered in any framework consistent with the Navier-Stokes equations.

\subsubsection{Specific Discussion of New Ideas}

\begin{itemize}

\item[(1)] \underline{Damping Mechanism of ``Twisted Differences"}: We will extract the ``main part" of $\mathcal{L}_1$, \eqref{vel:form:intro}, as the operator
\begin{align} \label{Lmain:def}
\mathcal{L}_{main} := \bar{u} u_x + \bar{u}_y v - u_{yy}.
\end{align}
One can think of the (scalar-valued) $\mathcal{L}_{main}$ as the ``Prandtl" component of \eqref{vel:form:intro}, though it is simpler than \eqref{vel:form:intro} for a variety of reasons (for instance, it is scalar-valued). To ease the discussion, we also make the assumption that $\bar{u}, \bar{v}$ solve the Prandtl equation, even though in reality this is only true to leading order in $\sqrt{\eps}$. That this is a central object in the study of $\mathcal{L}$ itself is well known, and has been discussed extensively in the works of the past several years, such as \cite{GI1}, \cite{Varet-Maekawa}, \cite{IyerBlasius}. 

The perspective taken in \cite{GI1} is to view the operator $\mathcal{L}_{main}$ as being comprised of two separate operators, the Rayleigh piece, $\bar{u} u_x + \bar{u}_y v$, and the diffusion, $- u_{yy}$. Crucially, \cite{GI1} was able to establish that one could obtain coercivity of $\mathcal{L}_{main}$ for $0 \le x \le L$ for $L$ small, using a new ``quotient estimate" by applying the multiplier $\p_x \frac{v}{\bar{u}}$ to the $x$-differentiated version of $\mathcal{L}_{main}$. We also draw a parallel to the approach of \cite{Varet-Maekawa} where their ``Rayleigh-Airy iteration" is comprised of viewing the Rayleigh piece of $\mathcal{L}_{main}$, and the Airy (or diffusive) piece as two separate operators.  

Using this perspective, these prior results are able to generate inequalities of the form $\|U, V \|_{X_0} \lesssim L \|U, V \|_{X_{\frac 1 2}}$ and $\|U, V \|_{X_{\frac 1 2}} \lesssim \|U, V \|_{X_0}$ (for appropriately defined spaces $X_0, X_{\frac 1 2}$, not necessarily the ones we have selected here). For this reason, the work \cite{GI1} requires $0 < L << 1$ to close their scheme.  

There is, in fact, a mechanism by which both components of $\mathcal{L}_{main}$ actually ``talk to each other" (and, in fact, this is a damping mechanism as $x \rightarrow \infty$). This link between the two components of $\mathcal{L}_{main}$ is provided by way of a \textit{change of unknown} at the velocity level.

\hspace{3 mm} We may re-interpret $\mathcal{L}_{main}$ as the same operator which controls the decay displayed in the Prandtl system, namely \eqref{asy:blas:1}. That is, we temporarily regard the $\bar{u}$ as a background Prandtl profile, say Blasius, and $u$ is to represent the difference between Blasius and some other Prandtl profile, $\bar{u}^0_p$. Schematically, replace  
\begin{align}
&\bar{u} \text{ in } \eqref{Lmain:def} \rightarrow \bar{u}_\ast, \qquad u \text{ in } \eqref{Lmain:def} \rightarrow \bar{u}^0_p - \bar{u}_\ast. 
\end{align}

 To obtain \eqref{asy:blas:1}, one introduces the change of variables and change of unknowns:
\begin{align} \label{def:phi:1:1}
\phi(x, \psi) := |\bar{u}^0_p(x, \psi)|^2 - |\bar{u}_\ast(x, \psi)|^2, 
\end{align}
where $\psi$ is the associated stream function. First, we will clarify the above abuse of notation. Indeed, to understand \eqref{def:phi:1:1}, we need to define the inverse map via the relation
\begin{align}
(\psi, \bar{u}^0_p) \mapsto y = y(\psi; \bar{u}^0_p) \iff \psi = \int_0^{y(\psi; \bar{u}^0_p)} \bar{u}^0_p. 
\end{align}
Then the abuse of notation from \eqref{def:phi:1:1} really means the following
\begin{align} \label{compare:y}
 \phi(x, \psi) := |\bar{u}^0_p(x, y(\psi; \bar{u}^0_p))|^2 - |\bar{u}_\ast(x, y(\psi; \bar{u}_\ast))|^2,
\end{align}
which we interpret as a \textit{``twisted subtraction"} because we wish to compare $\bar{u}^0_p$ and $\bar{u}_\ast$ at two different $y$ values, depending on the solutions themselves. 

In the new, nonlinear, coordinate system one obtains the equation, 
\begin{align} \label{coordinates:NL}
\p_x \phi - \bar{u}^0_p \p_{\psi \psi} \phi + A \phi = 0, \qquad A := - 2 \frac{\p_{yy} \bar{u}_\ast}{\bar{u}_\ast \bar{u}^0_p},
\end{align}
whenever $\bar{u}^0_p$ and $\bar{u}_\ast$ satisfy the Prandtl equation. Assuming now $\p_{yy} \bar{u}_\ast \le 0$, which is true for the Blasius solution $\bar{u}_\ast$ and that $\p_{yy} \bar{u}^0_p \le 0$ (which \textit{a-posteiori} becomes true, up to harmless remainders, upon controlling $\phi$ in sufficiently strong norms), the above equation admits a good energy estimate: 
\begin{align} \label{energy:damping:1}
\frac{\p_x}{2} \int \phi^2 \ud \psi + \int \bar{u}^0_p |\p_\psi \phi|^2 \ud \psi - \frac{1}{2} \int \p_{\psi \psi} \bar{u}^0_p |\phi|^2 \ud \psi + \int A \phi^2 \ud \psi = 0. 
\end{align} 
Note that both coefficients $- \frac 1 2 \p_{\psi \psi} \bar{u}^0_p$ and $A$ are nonnegative, and hence contribute damping terms. This is precisely the ``von-Mise" damping mechanism, and at the crux is that all the pieces of $\mathcal{L}_{main}$ talk to each-other.

To extend this to the more complicated setting of NS, we need to extract the two crucial points from this mechanism: 

\begin{itemize}
\item[(1a)] The precise manner in which one forms the difference between two solutions $\bar{u}^0_p, \bar{u}_\ast$ is via a ``twisted subtraction"; 

\item[(1b)] This ``twisted subtraction" leads to a damping mechanism as $x \rightarrow \infty$, as in \eqref{energy:damping:1}. We note that one manifestation of \eqref{energy:damping:1} is that one should expect the basic energy norm (called $\| \cdot \|_{X_0}$ here) to be ``stand-alone", as is shown in \eqref{turn:1}. 
\end{itemize}
We now aim to make (1a) and (1b) more robust, in a setting that will work well with linearized and nonlinear energy estimates. In particular, we prefer to work in the original coordinate system, $(x, y)$, as this is more suitable for the Navier-Stokes system. 

We begin with $(1a)$. Let now $y_1 := y(\psi; \bar{u}^0_p)$ and $y_2 := y(\psi; \bar{u}_\ast)$. A simple computation can show that, to leading order in the difference, $u$, between $\bar{u}^0_p$ and $\bar{u}_\ast$, 
\begin{align}
y_1 - y_2 = - \frac{1}{\bar{u}_\ast} \int_0^y u + O(u^2). 
\end{align}
Using this linearization, we can re-visit the quantity $\phi$ from \eqref{def:phi:j} and express it as 
\begin{align} \label{real:phi}
\phi = \bar{u}_\ast (\bar{u}^0_p(x, y_1) - \bar{u}_\ast(x, y_2)) + O(u^2) = \bar{u}_\ast (u - \frac{\p_y \bar{u}_\ast}{\bar{u}_\ast} \psi) + O(u^2),
\end{align}
where $\psi := \int_0^y u$. 

The leading order of the formula \eqref{real:phi} then motivates our introduction of the ``Good Variables", at the velocity level. More specifically, we define 
\begin{align} \label{good:variables:2}
 U := \frac{1}{\bar{u}} (u - \frac{\bar{u}_y}{\bar{u}} \psi), \qquad V := \frac{1}{\bar{u}} (v + \frac{\bar{u}_x}{\bar{u}} \psi ).
\end{align}
We note that this selection of $(U, V)$ compared to the right-hand side of \eqref{real:phi} differs in homogeneity of $\bar{u}^2$ (temporarily ignoring the matter of replacing $\bar{u}_\ast$ (Blasius) by the more general background profile used in practice, $\bar{u}$). This is because the homogeneity selected in \eqref{good:variables:2} also works well with the full Navier-Stokes system.

 We note that the choice of homogeneity of $\bar{u}$ in \eqref{good:variables:2} is a serious difference between Prandtl and Navier-Stokes \textit{at the very outset of the analysis}. In analyses of the scalar Prandtl equation (for instance \cite{IyerBlasius}), the natural quantity is actually \eqref{real:phi}. The Navier-Stokes equations, however, are severely constraining due to the presence of the pressure (again, the lack of a pressure is one of the \textit{main simplifications} of the boundary layer theory, and it is bound to have a significant impact on the stability analysis $U^\eps \rightarrow \bar{u}$). This comparatively \textit{singular} $U$ for Navier-Stokes, \eqref{good:variables:2}, does create losses of derivatives in our energy scheme, \eqref{turn:1} - \eqref{turn:3}, since singular weights $\frac{1}{\bar{u}}$ cost excess derivatives to control. Due to this, it is \textit{a-priori} unclear whether the von-Mise mechanism described above is actually consistent with the Navier-Stokes equations and with the stability as $\eps \rightarrow 0$. Our analytic work is devoted to designing and controlling several new norms in order to overcome this destabilizing effect (see points (4), (5), (6) below).

We now address $(1b)$. For the unknowns $(U, V)$, the operator $\mathcal{L}_{main}$ reads (upon invoking the Prandtl equation for the coefficients $\bar{u}, \bar{v})$
\begin{align} \label{reform}
\mathcal{L}_{main}[u, v] := \mathcal{T}[U] - u_{yy}, 
\end{align}
where the transport operator $\mathcal{T}[U] \sim \bar{u}^2 U_x + \bar{u} \bar{v} U_y + 2 \bar{u}_{yy} U$. The ``von-Mise" mechanism, which was observed in \cite{IyerBlasius} in an energetic context, is that upon taking inner product with $U$ and invoking the condition $\bar{u}_{yy} < 0$ (which is true for the Blasius profile), one obtains coercivity of $\mathcal{L}_{main}$. Indeed, computing now (entirely in $(x, y)$ coordinates), we obtain (after integration by parts) 
\begin{align} \n
\int \mathcal{L}_{main}[u,v] U \ud y = & \int \mathcal{T}[U] U \ud y + \int \p_y (\bar{u} U + \frac{\bar{u}_y}{\bar{u}} \psi) U_y \ud y \\ \n
\approx & \frac{\p_x}{2} \int \bar{u}^2 U^2 \ud y + \frac 3 2\int \bar{u}_{yy} U^2 \ud y + \int \bar{u} U_y^2 \ud y - \int 2 \bar{u}_{yy} U^2 \ud y  \\ \label{est:lin:intro:crucial}
\approx &  \frac{\p_x}{2} \int \bar{u}^2 U^2 \ud y + \int \bar{u} U_y^2 \ud y - \int \frac 1 2 \bar{u}_{yy} U^2 \ud y,  
\end{align}
where we have omitted several harmless terms (hence the $\approx$). The point is that, as in \eqref{energy:damping:1}, the dangerous Rayleigh contribution $\int \frac 3 2 \bar{u}_{yy}U^2$ is cancelled out by the diffusive commutator term $- 2 \bar{u}_{yy} U^2$, leaving an excess factor of $- \frac 1 2 \int \bar{u}_{yy} U^2$. This term acts as a damping term as $x \rightarrow \infty$ due to the property that $\bar{u}_{yy} \le 0$, analogously to \eqref{energy:damping:1}. Note that, if this were not the case, we would see growth as $x \rightarrow \infty$. The calculation \eqref{est:lin:intro:crucial} is precisely the version of \eqref{energy:damping:1} in $(x, y)$ coordinates. 

Therefore, our starting point is working with the ``good variables" \eqref{good:variables:2}, but in a robust enough manner to extend to the full Navier-Stokes system, and to capture large-$x$ dynamics.

\item[(2)] \underline{Sharp $L^\infty$ decay and the design of the space $\|U, V \|_{\mathcal{X}}$:} A consideration of the nonlinear part of $\mathcal{N}(u, v)$ in \eqref{vel:eqn:1:intro} demonstrates that, at the very least, one needs to control a final norm that is strong enough to encode pointwise decay of the form 
\begin{align} \label{decay:v:intro}
|\sqrt{\eps} v | \langle x \rangle^{\frac 1 2} \lesssim \|U, V \|_{\mathcal{X}}.
\end{align}
This is due to having to control trilinear terms of the form $\int B(x, y) v u_y u_x \langle x \rangle$, where $B(x, y)$ is a bounded function. This baseline requirement motivates our choice of space $\mathcal{X}$, defined in \eqref{X:norm}, which contains enough copies of the $x \p_x$ scaling vector-field of the solution in order to obtain the crucial decay estimate \eqref{decay:v:intro}. This is a sharp requirement, and our norm $\mathcal{X}$ is just barely strong enough to control this decay.  

\item[(3)] \underline{Cauchy-Kovalevskaya weighted energy}: Upon reformulation into \eqref{reform}, and to deal with the large $x$ problem, we need to apply a Cauchy-Kovalevskaya type weighted multiplier of the form $(U \langle x \rangle^{- \delta}, \eps V \langle x \rangle^{- \delta})$ for a $\delta$ chosen small enough. The purpose is to produce the positive terms $\| \bar{u} U \langle x \rangle^{-\frac{1}{2} - \frac{\delta}{2} } \|^2 + \| \sqrt{\eps} \bar{u} V \langle x \rangle^{- \frac 1 2 - \frac{\delta}{2}} \|$ that appear in the $X_0$ energy (see the precise definition of $X_0$ below in \eqref{def:X0}, in which we made a specific choice of $\delta = \frac{1}{100}$ for the sake of concreteness). In turn, these are crucially used to control error terms in the $X_0$ energy estimate. 

The idea of using Cauchy-Kovalevskaya weights has been employed before in the setting of steady Prandtl even for bounded $x$ (see, for instance, in chronological order: \cite{Iyer3}, \cite{IyerBlasius}, \cite{Gao-Zhang} as some examples). However our setting of asymptotic in $x$ is particularly tricky and has not been encountered in either of the aforementioned works due to the confluence of two reasons. First for infinite $x$, one \textit{cannot use} linear CK weights. Second, due to the requirement that the vector-field multiplier must be divergence-free, we need to select a multiplier of the form $(U \langle x \rangle^{- \delta}, \eps V \langle x \rangle^{- \delta} + \eps \delta \frac{\psi}{\bar{u}} \langle x \rangle^{-\delta-1} )$. 

Due to the fact that $\p_{xx} \langle x \rangle^{-\delta} > 0$ (which is exactly false for linear CK weights), the combination $\eps V \langle x \rangle^{- \delta} + \eps \delta \frac{\psi}{\bar{u}} \langle x \rangle^{-\delta-1}$ produces terms of competing sign, which are not obviously positive. We are able to handle such contributions by selecting CK weights that in turn enable us to employ Hardy-type inequalities \textit{with sharp constants} that allow for coercivity (see Lemma \ref{precise:1}).  

\item[(4)] \underline{Linearized energy estimates and the scaling vector field $S = x \p_x$:} In order to control the designer norm $\| U, V \|_{\mathcal{X}}$, we perform a sequence of estimates which results in the following loop (the reader is encouraged to consult \eqref{def:X0} - \eqref{def:Xn} for the definitions of these spaces), for $n= 0, 1, ..., 10$:
\begin{align} \label{turn:1}
&\| U, V \|_{X_0}^2 \lesssim \mathcal{F}_0 + \mathcal{N}_0, \\  \label{turn:2}
&\| U, V \|_{X_{n + \frac 1 2} \cap Y_{n + \frac 1 2}}^2 \lesssim C_\delta \| U, V \|_{X_{\le n}}^2  + \delta \|U, V \|_{X_{n+1}}^2 + \mathcal{F}_{n+ \frac 1 2} + \mathcal{N}_{n + \frac 1 2}, \\  \label{turn:3}
&\|U, V \|_{X_{n+1}}^2 \lesssim \|U, V \|_{X_{\le n + \frac 1 2}}^2  + \mathcal{F}_{n + 1} + \mathcal{N}_{n+1}, 
\end{align}
for a small number $0 < \delta << 1$, and where the implicit constant above is independent of the chosen $\delta$. Above the $\mathcal{F}$ terms represent forcing terms, which depend on the approximate solution, and the $\mathcal{N}$ terms represent quadratic terms. The coupling of these estimates is required by the vector aspect of the   the full linearized Navier-Stokes operator $\mathcal{L}$. To keep matters simple, the reader can identify these spaces with ``regularity of $x \p_x$". That is, $X_0$ is a baseline norm, $X_{\frac 1 2}, Y_{\frac 1 2}$ contain (in a sense made precise by the definitions of these norms) estimates on $(x\p_x)^{\frac 1 2}$ of the quantities in $X_0$, $X_1$ is basically $\| (x \p_x) U, (x \p_x) V\|_{X_0}$. 

It turns out that estimation of the nonlinear terms schematically work in the following manner: 
\begin{align} \label{nonlinear:est:intro}
|\mathcal{N}_n| \lesssim \eps \|U, V \|_{X_n}^2 \|U, V \|_{X_{n+\frac 1 2}} , \qquad |\mathcal{N}_{n + \frac 1 2}| \lesssim \eps \|U, V \|_{X_{\le n + \frac 1 2}}^3
\end{align}

The difficulty in closing our scheme becomes clear upon comparing the linear estimates in \eqref{turn:1} - \eqref{turn:3} and the estimation of the nonlinearity from \eqref{nonlinear:est:intro}: \textit{the linear estimates lose half $x \p_x$ derivative for half-integer order spaces, whereas the nonlinear estimates lose a half $x \p_x$ derivative for integer order spaces.} This is a new obstacle that has only appeared in the present work. We shall now discuss the origin of both the ``linear loss of derivative", appearing in \eqref{turn:2}, and the ``nonlinear loss of derivative", appearing in \eqref{nonlinear:est:intro}, together with our technique to eliminate these losses.  

\item[(5)] \underline{Loss of $x \p_x$ derivative due to degeneracy of $\bar{u}$ and weights of $x$:} The ``loss of half-$x\p_x$ derivative" at the linearized level, that is \eqref{turn:2}, is due to degeneracy of $\bar{u}$ at $y= 0$. The reader is invited to consult, for instance, the estimation of term \eqref{est:Xhalf:loss:deriv} in our energy estimates, which displays such a loss. To summarize, we have to estimate the following integral when performing the estimate of $X_{\frac 12}$, \eqref{turn:2}
\begin{align}
I_{sing} := |\int x \bar{u}_y U_x U_y \ud y \ud x|. 
\end{align}
A consultation with the $X_{\frac 1 2}$ and the previously controlled $X_0$ norm shows that this quantity is out of reach due to the confluence of two issues: the degeneracy of the weight $\bar{u}$ (and the lack of degeneracy of the coefficient $\bar{u}_y$) and the weight of $x$ appearing in $I_{sing}$. We emphasize that this type of ``loss-of-$x \p_x$" is a confluent issue: it only appears if one is concerned with global-in-$x$ matters in the presence of the degenerate weights $\bar{u}$, and hence appears for the first time in this work. 

For bounded $x$ values, one can simply appeal to the extra positive terms appearing from the viscosity in the $X_{n+ \frac 1 2}, Y_{n + \frac 1 2}$ norms. However, as we are here concerned with large $x$, a general principle we uncover is that we often \textit{``lose $x \p_x$ derivative, but do not lose $\sqrt{\eps} \p_x$ derivative"}. Therefore, it is almost never advantageous, in large $x$ matters, to invoke the tangential diffusion $- \eps \p_{xx}$. 

Due to this peculiarity, we design the scheme (and our norm $\mathcal{X}$, see \eqref{X:norm}) to terminate at the $X_{11}$ stage, as opposed to what appears to the be the more natural stopping point of $X_{11.5}, Y_{11.5}$. 

\item[(6)] \underline{Nonlinear Change of Variables (and nonlinearly modified norms)}: There is a major price to pay for the truncation of our energy scheme at the $X_{11}$ level, which comes from the nonlinear loss of $x \p_x$ derivative displayed in \eqref{nonlinear:est:intro}. Let us explain further the reason for this loss, temporarily setting $n = 11$ in the first inequality of \eqref{nonlinear:est:intro}.  Indeed, considering the trilinear quantity 
\begin{align}
T_{sing} := \int u_y \p_x^{11} v \p_x^{11} U \langle x \rangle^{22} \ud y \ud x,
\end{align}
and comparing to the controls provided by the $\| \cdot \|_{\mathcal{X}}$ norm, such a quantity is out of reach (due to growth as $x \rightarrow \infty$). Estimating this type of quantity would not be out of reach if we had the right to include $X_{11.5}, Y_{11.5}$ into the $\mathcal{X}$-norm, but due to the issue raised above, we must truncate our energy scheme at the $X_{11}$ level. 

To contend with this difficulty, we introduce a further \textit{nonlinear change of variables} that has the effect of cancelling out these most singular terms at the top order. This amounts to replacing the linearized good variables \eqref{good:variables:2} with another, nonlinear version, which is defined in \eqref{Chan:var:2}. We note that this difficulty does not appear in any previous work, due to the ability, for bounded values of $x$, to appeal to the positive contributions of the tangential viscosity. 

In turn, the energy estimate for the new nonlinear good unknown requires estimation of several trilinear terms in the nonlinearly modified norms $\Theta_{11}$, defined in \eqref{def:Theta:11}. We subsequently establish the equivalence of measuring the nonlinear good unknown in the modified norm $\Theta_{11}$ to measuring the original good unknown in our original space $\mathcal{X}$ (see for instance, the analysis in Section \ref{subsection:NLMN}). 

We emphasize that, in order to establish the equivalence of measuring the new nonlinear good unknowns in the nonlinearly modified norm to the full $\mathcal{X}$ norm, we need to rely upon the full strength of the $\mathcal{X}$-norm. 

\item[(7)] \underline{Weighted in $x$ and $\bar{u}$ mixed $L^p_xL^q_y$ embeddings:} Due to the inherent nonlinear nature of this top order analysis for $\Theta_{11}$, we rely upon several mixed-norm estimates of $L^p_x L^q_y$ with precise weights in $x$ and $\bar{u}$ in order to close our analysis of $\Theta_{11}$. This requirement is amplified upon noting that $u_{yy}$ (1) lacks the regularity in $y$ that the background $\bar{u}_{yy}$ has, due to the lack of higher-order $y$ derivative quantities in our $\|\cdot \|_{\mathcal{X}}$ (which appear to be difficult to achieve due to $\{y = 0\}$ boundary effects) and (2) lacks decay as $y \rightarrow \infty$ (we do not control weights in $y$ in our $\mathcal{X}$ space). Thus we must always place $u_{yy}$ in $L^2_y$ in the vertical direction, and develop the appropriate weighted in $\bar{u}$ and $x$ $L^p_x L^q_y$ embeddings. These, in turn, are developed in Sections \ref{Lpq:embed:section}, \ref{pw:section:1}, \ref{subsection:NLMN}, and again, rely upon the full strength of our $\mathcal{X}$ norm in order to close. Again, we emphasize that these analyses are completely new. 

%\item[(8)] \underline{Sharp decay of approximate solutions \& asymptotic stability of Blasius}: Several error terms that arise in our analysis require the sharp $L^\infty_y$ decay of the terms that comprise the approximate solution from \eqref{exp:base}. Moreover, the sharp decay and asymptotics of these approximate solutions is required in order to translate between the nonlinear good unknown, measured in the $\Theta_{11}$ norm, versus the original good unknown, measured in the $\mathcal{X}$ norm. 
%
%This analysis is conducted in Section \ref{section:PL} and results in Theorem \ref{thm:approx}. We remark that our analysis from Section \ref{section:PL} and our nonlinear change of variables from Section \ref{section:top:order} in fact results in yet another proof of the asymptotic stability of Blasius, \eqref{asy:blas:1} (see Proposition \ref{Prop:nonlinear}). In fact, the method developed in the present work is even more robust than that of \cite{Serrin} and \cite{IyerBlasius}, because we are able to work entirely in $(x, y)$ coordinates. Moreover, one can think of the construction of the approximate solutions as a ``preliminary version" of our technique in a simpler context (the linearized Prandtl equations are simpler than the full system of the Navier-Stokes equations). 
\end{itemize}

\subsection{Notational Conventions}

We first define (in contrast with the typical bracket notation) $\langle x \rangle := 1+ x$. We also define the quantity 
\begin{align} \label{z:choice}
z := \frac{y}{\sqrt{x + 1}} = \frac{y}{\sqrt{\langle x \rangle}},
\end{align}
due to our choice that $x_0 = 1$ (which we are again making without loss of generality). The cut-off function $\chi(\cdot): \mathbb{R}_+ \rightarrow \mathbb{R}$ will be reserved for a particular decreasing function, $0 \le \chi \le 1$, satisfying 
\begin{align} \label{def:chi}
\chi(z) = \begin{cases} 1 \text{ for } 0 \le z \le 1 \\ 0 \text{ for } 2 \le z < \infty  \end{cases}
\end{align}
Regarding norms, we define for functions $u(x, y)$, 
\begin{align}
\| u \| := \|u \|_{L^2_{xy}} = \Big( \int u^2 \ud x \ud y \Big)^{\frac 1 2}, \qquad \|u \|_\infty := \sup_{(x,y) \in \mathcal{Q}} |u(x, y)|. 
\end{align}
We will often need to consider ``slices", whose norms we denote in the following manner
\begin{align}
\| u \|_{L^p_y} := \Big( \int u(x, y)^p \ud y \Big)^{\frac 1 p}.
\end{align}
We use the notation $a \lesssim b$ to mean $a \le Cb$ for a constant $C$, which is independent of the parameters $\eps, \delta_\ast$. We define the following scaled differential operators 
\begin{align}
\nabla_\eps := \begin{pmatrix} \p_x \\ \frac{\p_y}{\sqrt{\eps}} \end{pmatrix}, \qquad \Delta_\eps := \p_{yy} + \eps \p_{xx}.
\end{align}
For derivatives, we will use both $\p_x f$ and $f_x$ to mean the same thing. For integrals, we will use $\int f := \int_0^\infty \int_0^\infty f(x, y) \ud y \ud x$. These conventions are taken unless otherwise specified (by appending a $\ud y$ or a $\ud x$), which we sometimes need to do. 

We will often use the parameter $\delta$ to be a generic small parameter, that can change in various instances. The constant $C_\delta$ will refer to a number that may grow to $\infty$ as $\delta \downarrow 0$. 

\subsection{Plan of the paper}

\hspace{5 mm} The plan of the paper is as follows. Throughout the paper, we take Theorem \ref{thm:approx} as given and use it as a ``black box". Indeed, Theorem \ref{thm:approx} is proven in the companion paper, \cite{IM21}. Section \ref{remainder:section:1} is devoted to introducing the Navier-Stokes system satisfied by the remainders, $(u, v)$, and proving basic properties of the associated linearized operator. Section \ref{remainder:section:2} is devoted to developing the functional framework, notably the various components of the space $\mathcal{X}$, in which we analyze the remainders, $(u, v)$. Section \ref{remainder:section:3} is devoted to providing the energy estimates to control $\|U, V \|_{X_n}$ and $\|U, V \|_{X_{n + \frac 1 2} \cap Y_{n + \frac 1 2}}$ for $n = 0,...,10$. Section \ref{section:top:order} contains our top ($11$'th order) analysis, which notably includes our nonlinear change of variables and nonlinearly modified norms. Section \ref{section:NL} contains the nonlinear analysis to close the complete $\mathcal{X}$ norm estimate for $(U, V)$.

\section{The Remainder System} \label{remainder:section:1}

\subsection{Presentation of Equations} \label{subsection:background}

\subsubsection{Background Profiles, $[\bar{u}, \bar{v}]$}

We recall the definition of $[\bar{u}, \bar{v}]$ from \eqref{exp:base}. In addition, for a few of the estimates in our analysis, we will require slightly more detailed information on these background profiles, in the form of decomposing into an Euler and Prandtl component. Indeed, define 
\begin{align} \label{split:split:1}
\bar{u}_P &:= \bar{u}^0_P + \sum_{i = 1}^{N_1} \eps^{\frac i 2} u^i_P  , &  \bar{u}_E &:=  \sum_{i = 1}^{N_{1}} \eps^{\frac{i}{2}}  u^i_E.
\end{align}
We will now summarize the quantitative estimates on $[\bar{u}, \bar{v}]$ that we will be using in the analysis of \eqref{vel:eqn:1} - \eqref{vel:eqn:2}.

First, let us recall now the Blasius profiles, defined in \eqref{Blasius:1} - \eqref{Blasius:3}, which are a family (due to the parameter $x_0$) of exact solutions to \eqref{BL:0:intro:intro} - \eqref{BL:1:intro:intro}. Recall also that, without loss of generality, we set $x_0 = 1$. We now record the following quantitative estimates on the Blasius solution: 
\begin{lemma} For any $k, j, M \ge 0$, 
\begin{align} \label{water:1}
&\| \langle z \rangle^M \p_x^k \p_y^j (\bar{u}_\ast - 1) \|_{L^\infty_y} \le  C_{M, k, j} \langle x \rangle^{- k - \frac j 2}, \\ \label{v:blasius}
&\|  \langle z \rangle^M \p_x^k \p_y^j (\bar{v}_\ast - \bar{v}_\ast(x, \infty)) \|_{L^\infty_y} \le C_{M, k, j} \langle x \rangle^{ - \frac 1 2- k - \frac j 2}, \\ \label{v:blasius:2}
&\| \p_x^k \p_y^j \bar{v}_\ast \|_{L^\infty_y} \le C_{k,j} \langle x \rangle^{- \frac 1 2 -  k - \frac j 2}.
\end{align}
\end{lemma}
We also have the following properties of the Blasius profile, which are well known and which will be used in our analysis.
\begin{lemma} For $[\bar{u}_\ast, \bar{v}_\ast]$ defined in \eqref{Blasius:1}, the following estimates are valid
\begin{align} \label{Blas:prop:1}
&|\p_y \bar{u}_\ast(x, 0)| \gtrsim \langle x \rangle^{- \frac 1 2}, \\  \label{Blas:prop:2}
&\p_{yy}\bar{u}_{\ast} \le 0. 
\end{align}
\end{lemma}

We will now state the estimates we will be using about our approximate solution. Note that we state these estimates as \textit{assumptions} for the purpose of this present paper. However, they are established rigorously in our companion paper according to Theorem \ref{thm:approx}.
\begin{assumption} \label{assume:1} For $0 \le j,m,k, M, l \le 20$, the following estimates are valid
\begin{align} \label{prof:u:est}
& \| \p_x^j (y\p_y)^m \p_y^k \bar{u} x^{j + \frac k 2} \|_\infty + \| \frac{1}{\bar{u}} \p_x^j \bar{u} x^j  \|_{\infty}  \le  C_{k,j} , \\ \label{prof:v:est}
& \| \p_x^j (y\p_y)^m \p_y^k \bar{v} x^{j + \frac k 2 + \frac 1 2}  \|_\infty   + \| \frac{1}{\bar{u}} \p_x^j \bar{v} x^{j + \frac 1 2}  \|_\infty  \le  C_{k,j}, \\ \label{est:Eul:piece}
&\eps^{- \frac 1 2}\| \p_x^j (Y \p_Y)^l \p_Y^k \bar{u}_E \langle x \rangle^{j+k + \frac 1 2} \|_\infty + \| \p_x^j (Y \p_Y)^l \p_Y^k \bar{v}_E \langle x \rangle^{j+k + \frac 1 2} \|_\infty \le  C_{k,j} , \\ \label{est:Pr:piece}
&\| \p_x^j (y\p_y)^k \p_y^l \bar{u}_P \langle z \rangle^M  \langle x \rangle^{j + \frac l 2}\|_\infty   \le  C_{k,j,M} \\ \label{est:PR:bar:v}
&\| \p_x^j (y\p_y)^k \p_y^l \bar{v}_P \langle z \rangle^M  \langle x \rangle^{j + \frac l 2+ \frac 1 2}\|_\infty   \le  C_{k,j,M}.
\end{align}
\end{assumption}

We will need estimates which amount to showing that $\bar{u}$ remains a small perturbation of $\bar{u}^0_p$. 
\begin{assumption} \label{lemma:bofz} Define a monotonic function $b(z) := \begin{cases} z \text{ for } 0 \le z \le \frac{3}{4} \\ 1 \text{ for } 1 \le z  \end{cases}$, where $b \in C^\infty$. Then 
\begin{align} \label{est:ring:1}
&\| \p_y^j \p_x^k (\bar{u} - \bar{u}^0_p) \langle x \rangle^{\frac j 2 + k + \frac{1}{50}} \|_{\infty} \le C_{k,j} \sqrt{\eps}, \\ \label{samezies:1}
&1 \lesssim \frac{\bar{u}^0_p}{b(z)}  \lesssim 1 \text{ and } 1 \lesssim \Big| \frac{\bar{u}}{\bar{u}^0_p} \Big| \lesssim 1, \\ \label{prime:pos}
&|\bar{u}_y|_{y = 0}(x)| \gtrsim \langle x \rangle^{- \frac 1 2}. 
\end{align}
\end{assumption}

We will need to remember the equations satisfied by the approximate solutions, $[\bar{u}, \bar{v}]$, which we state in the following assumption. 
\begin{assumption} \label{assume:3} Define the auxiliary quantities, 
\begin{align} \label{def:zeta}
&\zeta := \bar{u} \bar{u}_x + \bar{v} \bar{u}_y - \bar{u}^0_{pyy}, \\  \label{def:alpha}
&\alpha := \bar{u} \bar{v}_x + \bar{v} \bar{v}_y, 
\end{align}
For any $j, k, m \ge 0$, the following estimates hold:
\begin{align} \label{S:0}
|(x\p_x)^k (y\p_y)^m \zeta|  \lesssim &\sqrt{\eps} \langle x \rangle^{- (1 + \frac{1}{50})}  \\ \label{est:zeta:2}
|(x \p_x)^k (x^{\frac 1 2} \p_y)^j \zeta| \lesssim & \sqrt{\eps} \langle x \rangle^{- (1 + \frac{1}{50})}  \\ \label{S:1}
|(x \p_x)^k (y \p_y)^m \alpha| \lesssim & \bar{u} \langle x \rangle^{-\frac 3 2}.
\end{align}
\end{assumption}

\subsubsection{System on $[u, v]$}

We are now going to study the nonlinear problem for the remainders, $[u, v]$. We define the linearized operator in velocity form via 
\begin{align}
\begin{aligned} \label{vel:form}
\mathcal{L}[u, v] := \begin{cases} \mathcal{L}_1 :=  \bar{u} u_x + \bar{u}_{y} v + \bar{u}_{x} u + \bar{v} u_y - \Delta_\eps u   \\ \mathcal{L}_2 := \bar{u} v_x + u \bar{v}_{x} + \bar{v} v_y + \bar{v}_{y} v - \Delta_\eps v\end{cases}
\end{aligned}
\end{align}
Our objective is to study the problem 
\begin{align} \label{vel:eqn:1}
&\mathcal{L}[u, v] + \begin{pmatrix} P_x \\ \frac{P_y}{\eps} \end{pmatrix} = \begin{pmatrix} F_R \\ G_R \end{pmatrix} + \begin{pmatrix}  \mathcal{N}_1(u, v) \\ \mathcal{N}_2(u, v) \end{pmatrix}, \qquad u_x + v_y = 0, \text{ on } \mathcal{Q} \\ \label{vel:eqn:2}
&[u, v]|_{y = 0} = [u, v]|_{y \uparrow \infty} = 0, \qquad [u,v]|_{x = 0} = [u, v]|_{x = \infty} = 0. 
\end{align}
Above, the forcing terms $F_R$ and $G_R$ are defined in \eqref{forcing:remainder}, and obey estimates \eqref{est:forcings:part1}. The nonlinear terms are given by 
\begin{align} \label{def:N1:N2}
\mathcal{N}_1(u, v) := \eps^{\frac{N_2}{2}} (uu_x + vu_y), \qquad \mathcal{N}_2(u, v) := \eps^{\frac{N_2}{2}} (uv_x + vv_y).
\end{align}

\noindent In vorticity form, the operator is 
\begin{align} \label{vort:form}
\begin{aligned}
\mathcal{L}_{vort}[u, v] := & - u_{yyy} + 2 \eps v_{xxy} + \eps^2 v_{xxx} - \bar{u} \Delta_\eps v + v \Delta_\eps \bar{u}  - u \Delta_\eps \bar{v} + \bar{v} \Delta_\eps u. 
\end{aligned}
\end{align}

\subsection{The good unknowns}

We first introduce the unknowns 
\begin{align} \label{vm:1}
\qquad q = \frac{\psi}{\bar{u}}, \qquad U = \p_y q, \qquad V := -\p_x q,
\end{align}
from which it follows that 
\begin{align} \label{formula:1}
u = \bar{u} U + \bar{u}_y q , \qquad v = \bar{u} V - \bar{u}_x q.
\end{align}
An algebraic computation using \eqref{formula:1} yields the following 
\begin{align}  \label{inserting}
\bar{u} u_x + \bar{u}_y v + \bar{v} u_y + \bar{u}_x u = \bar{u}^2 U_x  + \bar{u} \bar{v} U_y + (2 \bar{u} \bar{u}_x + 2 \bar{v} \bar{u}_y) U + (\bar{u} \bar{u}_x + \bar{v} \bar{u}_y)_y q. 
\end{align}
%We will now use the equation satisfied by $[\bar{u}, \bar{v}]$ to simplify the coefficients above. Recalling the definition of $[\bar{u}, \bar{v}]$ from \eqref{def:bar:u:n}, we obtain 
%\begin{align} \label{expression:1}
%\bar{u} \bar{u}_x + \bar{v} \bar{u}_y = \bar{u}^0_{pyy} + r^{(0)}_1 +  r^{(1)}_1,  
%\end{align}
%where $r_1^{(0)}$ is of order $\sqrt{\eps}$, and is specifically defined by 
%\begin{align} \label{def:r10}
%r_1^{(0)} := \sqrt{\eps} u^1_{Ex} + \sqrt{\eps} u^1_{pyy}, 
%\end{align}
%and where $r_1^{(1)}$ is of order $\eps$, and is defined by 
%\begin{align} \label{def:r11}
%r^{(1)}_1 := & \text{\sameer{add this expression.}}
%\end{align}
%To compactify the notation, we will now define 
%\begin{align}
%\zeta := r_1^{(0)} + r_1^{(1)}.  
%\end{align}
Recalling \eqref{def:zeta}, we obtain the identity 
\begin{align} \n
\bar{u} u_x + \bar{u}_y v + \bar{v} u_y + \bar{u}_x u = \mathcal{T}_1[U] + \bar{u}^0_{pyyy} q + 2\zeta U + \zeta_y q,
\end{align}
where we define the operator $\mathcal{T}_1[U]$ via  
\begin{align} \label{def:T1}
&\mathcal{T}_1[U] := \bar{u}^2 U_x + \bar{u} \bar{v} U_y + 2 \bar{u}^0_{pyy} U.
\end{align}

We now perform a similar computation for the transport terms in $\mathcal{L}_2$. Again, a computation using \eqref{formula:1} yields the following
\begin{align} \n
\bar{u} v_x + u \bar{v}_x + \bar{v} v_y + \bar{v}_y v = & \bar{u} \p_x ( \bar{u} V - \bar{u}_x q ) + \bar{v}_x ( \bar{u} U + \bar{u}_yq ) + \bar{v} \p_y (\bar{u} V - \bar{u}_x q) + \bar{v}_y (\bar{u} V - \bar{u}_x q) \\ \n
= & \bar{u}^2 V_x + \bar{u} \bar{v} V_y + (2 \bar{u} \bar{u}_x + \bar{u}_y v + \bar{u} \bar{v}_y) V + (\bar{u} \bar{v}_x - \bar{u}_x \bar{v}) U \\ \n
&+ (- \bar{u} \bar{u}_{xx} + \bar{u}_y \bar{v}_x - \bar{u}_{xy} \bar{v} - \bar{u}_x \bar{v}_y)q \\ \label{vM:exp:eq:2}
= & \bar{u}^2 V_x + \bar{u} \bar{v} V_y + (\bar{u} \bar{u}_x + \bar{v} \bar{u}_y) V + \alpha U + \p_y \alpha q, 
\end{align}
where we have defined the coefficients $\alpha$ in \eqref{def:alpha}. We now again use \eqref{def:zeta} to simplify the coefficient of $V$ in \eqref{vM:exp:eq:2}, which yields 
\begin{align} \n
\bar{u} v_x + u \bar{v}_x + \bar{v} v_y + \bar{v}_y v = & \mathcal{T}_2[V] + \alpha U + \alpha_y q + \zeta V, 
\end{align}
where we have defined the operator $\mathcal{T}_2[V]$ via 
\begin{align}  \label{def:T2}
&\mathcal{T}_2[V] := \bar{u}^2 V_x + \bar{u} \bar{v} V_y + \bar{u}^0_{pyy}V.
\end{align}

We thus write our simplified system as 
\begin{align} \label{sys:sim:1}
&\mathcal{T}_1[U]  + 2 \zeta U + \zeta_y q - \Delta_\eps u + \bar{u}^0_{pyyy} q + P_x = F_R + \mathcal{N}_1, \\  \label{sys:sim:2}
&\mathcal{T}_2[V] + \alpha U + \alpha_y q + \zeta V + \frac{P_y}{\eps} - \Delta_\eps v = G_R + \mathcal{N}_2, \\  \label{sys:sim:3}
&U_x + V_y = 0,
\end{align}
and we note crucially that due to division by a factor of $\bar{u}$, we do not get a Dirichlet boundary condition at $\{y = 0\}$ for $U$, although we retain that $V|_{y = 0} = 0$. Summarizing the boundary conditions on $[U, V]$, we have 
\begin{align} \label{BC:UVYW}
U|_{x = 0} = V|_{x = 0} = 0, \qquad U|_{y = \infty} = V|_{y = \infty} = 0, \qquad U|_{x = \infty} = V|_{x = \infty} = 0, \qquad V|_{y = 0} = 0. 
\end{align}

%Expanding the operators $\tilde{\mathcal{L}}_1, \tilde{\mathcal{L}}_2$ in terms of $[U, V, q]$ yields 
%\begin{align} \n
%\tilde{\mathcal{L}}_1 = & \bar{u} \tilde{u}_s U_x + \bar{u} \tilde{v}_s U_y + \Big( 2 \bar{u}_y \tilde{v}_s + \bar{u} \tilde{u}_{sx}  + \bar{u}_x \tilde{u}_s \Big) U \\ \label{S1:form}
%& + \Big( \bar{u} \tilde{u}_{sy} - \tilde{u}_s \bar{u}_y \Big) V + \Big( \tilde{u}_s \bar{u}_{xy} + \tilde{u}_{sx} \bar{u}_y - \bar{u}_x \tilde{u}_{sy} + \bar{u}_{yy} \tilde{v}_s \Big) q, \\ \n
%\tilde{\mathcal{L}}_2 = & \bar{u} \tilde{u}_s V_x + \bar{u} \tilde{v}_s V_y + \Big( \bar{u} \tilde{v}_{sy} + \bar{u}_y \tilde{v}_s + 2 \bar{u}_x \tilde{u}_s \Big) V \\  \label{S2:form}
%& + \Big( \bar{u} \tilde{v}_{sx} - \bar{u}_x \tilde{v}_s \Big) U + \Big( - \bar{u}_x \tilde{v}_{sy} - \tilde{v}_s \bar{u}_{xy} + \bar{u}_y \tilde{v}_{sx}  - \tilde{u}_s \bar{u}_{xx} \Big) q.
%\end{align}

It will be convenient also to introduce the vorticity formulation, which we will use to furnish control over the $Y_{n + \frac 1 2}$ norms, which reads 
\begin{align} \n
&\p_y \mathcal{T}_1[U] - \eps \p_x \mathcal{T}_2[V] - u_{yyy} - 2 \eps u_{xxy} + \eps^2 v_{xxx} + \p_y^4 (\bar{u}^0_{pyyy}q) \\ \label{eq:vort:pre}
=& 2 (\zeta U)_y + (\zeta_y q)_y - \eps (\alpha U)_x - \eps (\alpha_y q)_x - \eps (\zeta V)_x + \p_y F_R - \eps \p_x G_R + \p_y \mathcal{N}_1 - \eps \p_x \mathcal{N}_2. 
\end{align}

We also now apply $\p_x^{(n)}$ to \eqref{sys:sim:1} - \eqref{sys:sim:3}, which produces the system for $U^{(n)} := \p_x^n U, V^{(n)} := \p_x^n V$, 
\begin{align} \label{sys:sim:n1}
&\mathcal{T}_1[U^{(n)}]  + 2 \zeta U^{(n)} + \zeta_y q^{(n)} - \Delta_\eps u^{(n)} + \bar{u}^0_{pyyy} q^{(n)} + P^{(n)}_x = \p_x^n F_R + \p_x^n \mathcal{N}_1 -  \mathcal{C}_1^n , \\  \label{sys:sim:n2}
&\mathcal{T}_2[V^{(n)}] + \alpha U^{(n)} + \alpha_y q^{(n)} + \zeta V^{(n)} + \frac{P^{(n)}_y}{\eps} - \Delta_\eps v^{(n)} = \p_x^n G_R + \p_x^n \mathcal{N}_2 - \mathcal{C}_2^n, \\  \label{sys:sim:n3}
&U^{(n)}_x + V^{(n)}_y = 0,
\end{align}
where the quantities $\mathcal{C}_1^n, \mathcal{C}_2^n$ contain lower order commutators, and are specifically defined by 
\begin{align} \label{def:C1n}
\mathcal{C}_1^n := & \sum_{k = 0}^{n-1} \binom{n}{k} ( \p_x^{n-k} \zeta U^{(k)} - \p_x^{n-k} \zeta_y q^{(k)} ), \\  \label{def:C2n}
\mathcal{C}_2^n := & \sum_{k = 0}^{n-1} \binom{n}{k} ( \p_x^{n-k} \alpha U^{(k)} + \p_x^{n-k} \alpha_y q^{(k)} + \p_x^{n-k} \zeta V^{(k)}  ).
\end{align}

\section{The Space $\mathcal{X}$} \label{remainder:section:2}

In this section, we provide the basic functional framework for the analysis of the remainder equation, \eqref{vel:eqn:1} - \eqref{vel:eqn:2}. In particular, we define our space $\mathcal{X}$ and develop the associated embedding theorems that we will need. 

\subsection{Definition of Norms}

To define the basic energy norm, we will define the following weight function 
\begin{align}
g(x)^2 := 1 + \langle x \rangle^{-\frac{1}{100}}. 
\end{align}
The purpose of $g$ is to act like $1$ as $x$ gets large, but as $g' < 0$, this will provide extra control for $(U, V)$ near $x = 0$, due to the presence of the final two terms in \eqref{def:X0} below.  

Define the basic energy norm via 
\begin{align} \n
\|U, V \|_{X_0}^2 := &\| \sqrt{\bar{u}} U_y g  \|^2 + \| \sqrt{\eps} \sqrt{\bar{u}} U_x g \|^2 + \| \eps \sqrt{\bar{u}} V_x  g\|^2 \\  \n
+& \| \sqrt{-\bar{u}_{yy}} U g  \|^2 + \eps \| \sqrt{-\bar{u}_{yy}} V g \|^2+ \| \sqrt{\bar{u}_y} U g \|_{y = 0}^2 \\  \label{def:X0}
+& \| \bar{u} U \langle x \rangle^{-\frac{1}{2} - \frac{1}{200} } \|^2 + \| \sqrt{\eps} \bar{u} V \langle x \rangle^{-1 - \frac{1}{200}} \|. 
\end{align}

To define higher-order norms, we need to define increasing cut-off functions, $\phi_n(x)$, for $n = 1,...,12$, where $0 \le \phi_n \le 1$, and which satisfies 
\begin{align} \label{def:phi:j}
\phi_n(x) = \begin{cases} 0 \text{ on } 0 \le x \le 200 + 10n \\ 1 \text{ on } x \ge 205 + 10n.  \end{cases}
\end{align}
The ``half-level" norms will be defined as (for $n = 0,...,10$),
\begin{align} \label{def:half:norm}
&\| U, V \|_{X_{n + \frac 1 2}} := \| \bar{u} U^{(n)}_x x^{n + \frac 1 2} \phi_{n+1} \| + \sqrt{\eps} \| \bar{u} V^{(n)}_x x^{n + \frac 1 2} \phi_{n+1}  \|, \\ \n
&\| U, V \|_{Y_{n+ \frac 1 2}} := \| \sqrt{\bar{u}} U^{(n)}_{yy} x^{n + \frac 1 2} \phi_{n+1} \| + \| \sqrt{\bar{u}} \sqrt{\eps} U^{(n)}_{xy} x^{n + \frac 1 2} \phi_{n+1} \| \\ \label{def:half:norm:Y}
& \qquad \qquad \qquad + \| \sqrt{\bar{u}} \eps U^{(n)}_{xx}  x^{n + \frac 1 2} \phi_{n+1} \|  + \| \sqrt{\bar{u}_y} U^{(n)}_y x^{n+\frac 1 2} \phi_{n+1}  \|_{y = 0}, \\ 
&\| U, V \|_{X_{n + \frac 1 2} \cap Y_{n + \frac 1 2}} := \| U, V \|_{X_{n + \frac 1 2}} + \| U, V \|_{Y_{n + \frac 1 2}}.
\end{align}

\begin{remark} The motivation for the above space (and the corresponding notation) is that we think of the quantities measured by the $X_{\frac 1 2} \cap Y_{\frac 1 2}$ norms as ``half $x \p_x$ derivative" more than those in $X_0$. Indeed, if we (on a \textit{very} formal level) consider the heat equation scaling of $\p_x \approx \p_{yy}$, then this is the case. The requirement of including two norms $X_{\frac 1 2}$ and $Y_{\frac 1 2}$ is that, if we had a scalar equation, we could immediately deduce properties of $U_{yy}$ from those on $U_x$. However, since we have a complicated system, this is not possible. Information about $U_{yy}$ needs to be obtained through separate control of the norm $Y_{\frac 1 2}$. 
\end{remark}

We would now like to define higher order versions of the $X_0$ norm, which we do via 
\begin{align} \n
\|  U, V \|_{X_n} := &\| \sqrt{\bar{u}} U^{(n)}_{y} x^n \phi_n \|^2 + \| \sqrt{\eps} \sqrt{\bar{u}} U^{(n)}_{x} x^n \phi_n \|^2 + \| \eps \sqrt{\bar{u}} V^{(n)}_{x}  x^n \phi_k\|^2 \\  \label{def:Xn}
+& \| \sqrt{-\bar{u}_{yy}} U^{(n)} x^n \phi_n  \|^2 + \eps \| \sqrt{-\bar{u}_{yy}} V^{(n)} x^n \phi_n \|^2+ \| \sqrt{\bar{u}_y} U^{(n)} x^n \phi_n \|_{y = 0}^2,
\end{align}

We will need ``local" versions of the higher-order norms introduced above. According to \eqref{def:phi:j}, since $\phi_1 = 1$ only on $x \ge 215$, we will need higher regularity controls for $0 \le x \le 215$. Define now a sequence of parameters, $\rho_j$, according to  
\begin{align}
\rho_2 = 0, \qquad \rho_j = \rho_{j-1} + 5
\end{align}
Set now the cut-off functions $\psi_2(x) = 1$, and 
\begin{align}
\psi_j(x) := \begin{cases} 0 \text{ for } x < \rho_j \\ 1 \text{ for } x \ge \rho_{j} + 1 \end{cases} \text{ for } 3 \le j \le 11
\end{align}

Our complete norm will be 
\begin{align}  \label{X:norm}
\| U, V \|_{\mathcal{X}} := &  \sum_{n = 0}^{10} \Big( \| U, V \|_{X_n} + \| U, V \|_{X_{n + \frac 1 2}} + \| U, V \|_{Y_{n + \frac 1 2}} \Big) + \| U, V \|_{X_{11}} + \| U, V \|_E,
\end{align}
where quantity $\|U, V\|_E$ will be defined below, in \eqref{def:E:norm}. We will also set the parameter $M_1 = 24$. 
%We also introduce the notation for $n \ge 1$, 
%\begin{align}
%&\|U, V \|_{\mathcal{X}_{n}} := \sum_{j = 0}^n \| U, V \|_{X_j} + \sum_{j = 0}^{n-1} \| U, V \|_{\tilde{X}_{j + \frac 1 2}}, \\
%&\|U, V \|_{\mathcal{X}_{n - \frac 1 2}} := \sum_{j = 0}^{n-1} \| U, V \|_{X_j} + \sum_{j = 0}^{n-1} \| U, V \|_{\tilde{X}_{j + \frac 1 2}}
%\end{align}

It will be convenient to introduce the following notation to simplify expressions, where $k = 1, ..., 11$, 
\begin{align}
\| U, V \|_{\mathcal{X}_{\le k}} := & \sum_{j = 0}^{k} \| U, V \|_{X_j} + \sum_{j = 0}^{k-1} \| U, V \|_{X_{j + \frac 1 2} \cap Y_{j + \frac 1 2}}, \\
\| U, V \|_{\mathcal{X}_{\le k - \frac 1 2}} := & \sum_{j = 0}^{k-1} \| U, V \|_{X_j} + \sum_{j = 0}^{k-1} \| U, V \|_{X_{j + \frac 1 2} \cap Y_{j + \frac 1 2}},
\end{align}
and the ``elliptic" part of the norm, \eqref{X:norm} via 
\begin{align} \label{def:E:norm}
\| U, V \|_{E} := \sum_{k = 1}^{11} \eps^k \| (\p_x^k u_y, \sqrt{\eps} \p_x^k u_x, \eps \p_x^k v_x ) \psi_{k+1} \| + \sum_{k = 1}^{11} \eps \| (\p_x^k u_y, \sqrt{\eps} \p_x^k u_x, \eps \p_x^k v_x ) \gamma_{k-1,k} \|.
\end{align}
Above, the functions $\gamma_{k,k+1}(x)$ are additional cut-off functions defined by 
\begin{align}
\gamma_{k-1,k}(x) := \begin{cases} 0 \text{ on } x \le 197 + 10k \\ x \ge 198 + 10k \end{cases}.
\end{align}
The point of $\gamma_{k-1,k}$ is to satisfy the following two properties: $\gamma_{k-1,k}$ is supported on the set where $\phi_{k-1} =1$ and $\phi_k$ is supported on the set when $\gamma_{k-1,k} = 1$. The inclusion of the $\| U, V \|_E$ norm above is to provide information near $\{x = 0\}$ (comparing the region where $\psi_{k} = 1$ versus where $\phi_k = 1$). The estimation of $\|U, V \|_E$ is through elliptic regularity, and therefore cannot give any useful asymptotic in $x$ information on the solution. 

%\begin{remark} The scale of norms comprising the space $\mathcal{X}$ is similar to those used in Section \ref{Section:approx} to construct solutions to the linearized Prandtl equation, but differs in some key ways. First, we do not control weights of $z$ here, as the solution is no longer expected to be self-similar in this weight (which reflects the parabolic nature of the linearized Prandtl equation). Second, we need to cut-off away from $x = 0$ here, which also explains the presence of the elliptic piece of the norm, $\|U, V \|_E$. The reason is to avoid boundary contributions from the tangential diffusion, $- \eps u_{xx}, - \eps^2 v_{xx}$, which are absent from the linearized Prandtl equations. 
%\end{remark}

\subsection{Hardy-type Inequalities}

We first recall from \eqref{z:choice} that $z = \frac{y}{\sqrt{x + 1}}$. We now prove the following lemma. 
\begin{lemma}\label{lemma:hardy:9} For $0 < \gamma << 1$, and for any function $f \in H^1_y$, for all $x \ge 0$, the following inequality is valid:  
\begin{align} \label{Hardy:1}
\| f \|_{L^2_y}^2 \lesssim \gamma \| \sqrt{ \bar{u}^0_p}  f_y \langle x \rangle^{\frac 1 2} \|_{L^2_y}^2 + \frac{1}{\gamma^2} \| \bar{u}^0_p f \|_{L^2_y}^2. 
\end{align}
\end{lemma}
\begin{proof} We square the left-hand side of \eqref{Hardy:1} and localize the integral based on $z$ via 
\begin{align}
\int f^2 \ud y = \int f^2 \chi(\frac{z}{\gamma}) \ud y + \int f^2 (1 - \chi(\frac{z}{\gamma})) \ud y. 
\end{align}
For the localized component, we integrate by parts in $y$ via 
\begin{align}
\int f^2 \chi(\frac{z}{\gamma}) \ud y = \int \p_y (y) f^2 \chi(\frac{z}{\gamma}) \ud y = - \int 2 y f f_y \chi(\frac{z}{\gamma}) \ud y - \frac{1}{\gamma} \int \frac{y}{\sqrt{x}} f^2 \chi'(\frac{z}{\gamma}) \ud y. 
\end{align}
We estimate each of these terms via 
\begin{align}
\Big| \int y f f_y \chi(\frac{z}{\gamma})\ud y \Big| \lesssim \| f  \|_{L^2_y} \|   \sqrt{x} \sqrt{\bar{u}^0_p} \sqrt{\gamma} f_y \|_{L^2_y} \le \delta \| f \|_{L^2_y}^2 + C_\delta \gamma x \| \sqrt{\bar{u}^0_p} f_y \|_{L^2_y}^2,
\end{align}
where above, we have used \eqref{samezies:1}. For the far-field term, we estimate again by invoking \eqref{samezies:1} via 
\begin{align}
|\int f^2 (1 - \chi(\frac{z}{\gamma})) \ud y| = |\int \frac{1}{|\bar{u}^0_p|^2} |\bar{u}^0_p|^2 f^2 (1 - \chi(\frac{z}{\gamma})) \ud y| \lesssim \frac{1}{\gamma^2} \| \bar{u}^0_p f \|_{L^2_y}^2.
\end{align}
We have thus obtained 
\begin{align}
\| f \|_{L^2_y}^2 \le \delta \| f \|_{L^2_y}^2 + C_\delta \gamma x \| \sqrt{\bar{u}^0_p} f_y \|_{L^2_y}^2 + \frac{C}{\gamma^2} \| \bar{u}^0_p f \|_{L^2_y}^2, 
\end{align}
and the desired result follows from taking $\delta$ small relative to universal constants and absorbing to the left-hand side. 
\end{proof}

We will often use estimate \eqref{Hardy:1} in the following manner 
\begin{corollary} For any $0 < \gamma << 1$,  
\begin{align} \label{bob:1}
 \| U_y \| + \| \sqrt{\eps} U_x \| + \| \eps V_x \| \le \gamma \| U, V \|_{Y_{\frac 1 2}} + C_\gamma \| U, V \|_{X_0}. 
\end{align}
\end{corollary}
\begin{proof} This follows immediately upon taking $f = U_y$, $\sqrt{\eps} U_x$, or $\eps V_x$ in \eqref{Hardy:1}.
\end{proof}

\begin{lemma} Assume $f(0, y) = f(\infty, y) = 0$. Let $\tilde{\chi}$ be an increasing cut-off function, $\tilde{\chi}' \ge 0$. Then for any $\sigma > 0$,  
\begin{align}
\Big\| \frac{1}{\langle x \rangle^{\frac 1 2+\sigma}} f \tilde{\chi}(z) \Big\| \lesssim_\sigma \| f_x \langle x \rangle^{\frac 1 2 -\sigma} \tilde{\chi}(z) \|. 
\end{align}
\end{lemma}
\begin{proof} We square the left-hand side via 
\begin{align} \n
\Big\| \frac{1}{\langle x \rangle^{\frac 1 2+\sigma}} f \tilde{\chi}(z) \Big\|^2 = & \int \langle x \rangle^{- 1 - 2\sigma} f^2 \tilde{\chi}(z) = - \int \frac{\p_x }{ 2 \sigma} \langle x \rangle^{-2\sigma} f^2 \tilde{\chi}(z) \\ \label{local:hardy:1}
= & \frac{1}{\sigma} \int \langle x \rangle^{-2\sigma} f f_x \tilde{\chi}(z) + \frac{1}{2\sigma} \int \langle x \rangle^{-2\sigma} f^2 \p_x \tilde{\chi}(z). 
\end{align}
Next, we observe that the right-most term in \eqref{local:hardy:1} is signed negative and can thus be moved to the left-side, as 
\begin{align}
\p_x \tilde{\chi}(z) = \p_x \tilde{\chi}(\frac{y}{\sqrt{x}}) = - \frac{z}{2x} \tilde{\chi}'(z) \le 0. 
\end{align}
We thus may conclude by estimating the first term from \eqref{local:hardy:1} by Cauchy-Schwartz. 
\end{proof}

This will often be used in conjunction with the Hardy-type inequality \eqref{Hardy:1}, via:
\begin{corollary} For any $k \ge 0$, and for any $0 < \sigma << 1$, any $0 < \gamma << 1$, 
\begin{align} \label{Hardy:three:a}
\| \frac{1}{\langle x \rangle^{\frac 1 2+\sigma}} U \| & \le \gamma \| \sqrt{\bar{u}} U_y \| + C_{\gamma, \sigma} \| \bar{u} U_x \langle x \rangle^{\frac 1 2} \|, \\ \label{Hardy:four:a}
\| \frac{1}{\langle x \rangle^{\frac 1 2+\sigma}} \sqrt{\eps}V \| &\le \gamma \| \sqrt{\bar{u}} \sqrt{\eps} V_y \| + C_{\gamma, \sigma} \| \sqrt{\eps} \bar{u} V_x \langle x \rangle^{\frac 1 2} \|, \\ \label{Hardy:three}
\| U^{(k)}_x \langle x \rangle^{k + \frac 1 2} \| &\le \gamma \| \sqrt{\bar{u}} U^{(k+1)}_y \langle x \rangle^{k+1} \| + C_{\gamma, \sigma} \| \bar{u} U^{(k)}_x \langle x \rangle^{k +\frac 1 2} \|, \\ \label{Hardy:four}
\| \sqrt{\eps}V^{(k)}_x  \langle x \rangle^{k + \frac 1 2}   \| & \le \gamma \| \sqrt{\bar{u}} \sqrt{\eps} V^{(k+1)}_y \langle x \rangle^{k+1} \| + C_{\gamma, \sigma} \| \sqrt{\eps} \bar{u} V^{(k)}_x \langle x \rangle^{k+\frac 1 2} \|
\end{align}
\end{corollary}

We now need to record a Hardy-type inequality in which the precise constant is important. More precisely, the fact that the first coefficient on the right-hand side below is very close to $1$ will be important. 
\begin{lemma} For any function $f(x): \mathbb{R}_+ \rightarrow \mathbb{R}$ satisfying $f(0) = 0$ and $f \rightarrow 0$ as $x \rightarrow \infty$, there exists a $C > 0$ such that
\begin{align} \label{precise:1}
\int \langle x \rangle^{-3.01} \bar{u}^2 f^2 \ud x \le \frac{1}{1.01}  \int \langle x \rangle^{-1.01} \bar{u}^2 f_x^2 \ud x + \frac{2}{1.01} \int \langle x \rangle^{-2.01} \bar{u} \bar{u}_x f^2 \ud x.  
\end{align}
\end{lemma}
\begin{proof} We compute the quantity on the left-hand side of above via 
\begin{align} \n
\int \langle x \rangle^{-3.01} \bar{u}^2 f^2 \ud x = &- \int \frac{\p_x}{2.01} \langle x \rangle^{-2.01} \bar{u}^2 f^2 \ud x \\ \n
=& \frac{2}{2.01} \int \langle x \rangle^{-2.01} \bar{u}^2 f f_x \ud x + \frac{2}{2.01} \int \langle x \rangle^{-2.01}  \bar{u} \bar{u}_x f^2 \ud x \\ 
\le & \frac{1}{2.01} \int \langle x \rangle^{-3.01} \bar{u}^2 f^2 \ud x + \frac{1}{2.01} \int \langle x \rangle^{-1.01} \bar{u}^2 f_x^2 \ud x + \frac{2}{2.01} \int \langle x \rangle^{-2.01} \bar{u} \bar{u}_x f^2 \ud x,
\end{align}
which, upon bringing the first term on the right-hand side to the left, gives the inequality \eqref{precise:1} with the precise constants. 
\end{proof}

\subsection{$L^p_x L^q_y$ Embeddings} \label{Lpq:embed:section}

We will now state some $L^p_x L^q_y$ type embedding theorems on $(U, V)$ using the specification of $\| U, V \|_{\mathcal{X}}$.  
\begin{lemma} For $1 \le j \le 10$, 
\begin{align} \label{Lpq:emb:V:1}
\eps^{\frac 1 4}\| V^{(j)} \langle x \rangle^j \phi_{j+1} \|_{L^2_x L^\infty_y} \lesssim & \| U, V \|_{\mathcal{X}_{\le j+1}}, \\ \label{Lpq:emb:V:2}
\eps^{\frac 1 4}\|  \bar{u} V^{(j)} \langle x \rangle^j \phi_{j+1} \|_{L^2_x L^\infty_y} \lesssim & \| U, V \|_{\mathcal{X}_{\le j + \frac 1 2}}
\end{align}
\end{lemma}
\begin{proof} We begin with \eqref{Lpq:emb:V:1}. For this, we first freeze $x$ and integrate from $y = \infty$ to obtain 
\begin{align} \n
\eps^{\frac 1 2}|V^{(j)}|^2 \langle x \rangle^{2j} \phi_{j+1}^2 = & 2 \eps^{\frac 1 2} |\int_y^\infty V^{(j)} V^{(j)}_y\langle x \rangle^{2j} \phi_{j+1}^2 \ud y' | \lesssim \| \eps^{\frac 1 2} V^{(j)} \langle x \rangle^{j - \frac 1 2} \phi_{j+1} \|_{L^2_y} \| U^{(j)}_x \langle x \rangle^{j+ \frac 1 2} \phi_{j+1} \|_{L^2_y} \\
= & \| \eps^{\frac 1 2} V^{(j-1)}_x \langle x \rangle^{(j-1)+ \frac 1 2} \phi_{j+1} \|_{L^2_y} \| U^{(j)}_x \langle x \rangle^{j + \frac 1 2} \phi_{j+1} \|_{L^2_y}.
\end{align}
We now take $L^2_x$ and appeal to \eqref{Hardy:three} - \eqref{Hardy:four}.

Similarly, we compute 
\begin{align} \n
\eps^{\frac 1 2}\bar{u}_\ast^2 |V^{(j)}|^2 \langle x \rangle^{2j} \phi_{j+1}^2 \le & 2 \eps^{\frac 1 2} \bar{u}_\ast^2 |\int_y^\infty V^{(j)} V^{(j)}_y\langle x \rangle^{2j} \phi_{j+1}^2 \ud y' | \lesssim \eps^{\frac 1 2}  \int_y^\infty \bar{u}_\ast^2| V^{(j-1)}_x ||U^{(j)}_x|\langle x \rangle^{2j} \phi_{j+1}^2 \ud y' \\
\lesssim & \eps^{\frac 1 2} \| \bar{u}_\ast V^{(j-1)}_x \langle x \rangle^{j - \frac 1 2} \phi_{j+1}\| \| \bar{u}_\ast U^{(j)}_x \langle x \rangle^{j + \frac 1 2} \phi_{j+1} \| 
\end{align}
where above we have used that $\bar{u}_\ast(y') \ge \bar{u}_\ast(y) \ge 0$ when $y' \ge y$ to bring the $\bar{u}_\ast^2$ factor inside the integral. We now use \eqref{samezies:1} to conclude. 
\end{proof}

We will now need to translate the information on $(U, V)$ from the norms stated above to information regarding $(u, v)$. 
\begin{lemma} For $2 \le j \le 10$, $0 \le k \le 10$, $1 \le m \le 11$, and for any $0 < \delta << 1$,   
\begin{align} \label{L2:uv:eq}
\| u^{(k)}_x x^{k+ \frac 1 2} \phi_{k+1} \| + \| \sqrt{\eps} v^{(k)}_x x^{k + \frac 1 2} \phi_{k+1} \| \le &  C_\delta \| U, V \|_{\mathcal{X}_{\le k + \frac 1 2}} + \delta \| U, V\|_{\mathcal{X}_{\le k+1}} \\ \label{L2:uv:eq:2}
\| u^{(m)}_y x^{m} \phi_m \| +\sqrt{\eps} \| u^{(m)}_x x^m \phi_{m} \| + \eps \| v^{(m)}_x x^{m} \phi_{m} \|  \lesssim &  \| U, V \|_{\mathcal{X}_{\le m}} \\ \label{mixed:L2:orig:1}
\| \frac{1}{\bar{u}} \p_x^j v \langle x\rangle^{j } \phi_{j+1} \|_{L^2_x L^\infty_y}  \lesssim &  \| U, V \|_{\mathcal{X}_{\le j + 1}}, \\ \label{mL2again}
\| \p_x^j v \langle x\rangle^{j } \phi_{j+1} \|_{L^2_x L^\infty_y}  \lesssim &  \| U, V \|_{\mathcal{X}_{\le j + \frac 1 2 }}. 
\end{align}
\end{lemma}
\begin{proof} To prove these estimates, we simply use \eqref{formula:1} to express $(u, v)$ in terms of $(U, V)$. We do this now, starting with \eqref{L2:uv:eq}. Differentiating \eqref{formula:1} $k+1$ times in $x$, we obtain 
\begin{align} \n
\| u^{(k)}_x x^{k+ \frac 1 2} \phi_{k+1} \| \le &\sum_{l = 0}^{k+1} \binom{k+1}{l} ( \| \p_x^l \bar{u} \p_x^{k - l} U_x x^{k+\frac 1 2} \phi_{k+1} \| + \| \p_x^{l} \bar{u}_y \p_x^{k-l}V x^{k + \frac 1 2} \phi_{k+1} \| ) \\ \n
\lesssim & \sum_{l = 0}^{k+1} \| \frac{\p_x^l \bar{u}}{ \bar{u}} x^l \|_\infty \| \bar{u} \p_x^{k-l} U_x \langle x \rangle^{k - l + \frac 1 2} \phi_{k+1} \| + \| \p_x^l \bar{u}_y y x^l \|_\infty \| \p_x^{k-l} V_y x^{k-l + \frac 1 2} \phi_{k+1} \| \\ \n
\lesssim & \sum_{l = 0}^{k+1}  \| \bar{u} \p_x^{k-l} U_x \langle x \rangle^{k - l + \frac 1 2} \phi_{k+1} \| + \| \bar{u} \p_x^{k-l} V_y x^{k-l + \frac 1 2} \phi_{k+1} \| +\|  \sqrt{\bar{u}} \p_x^{k+1-l} U_y x^{k+1 -l} \phi_{k+1} \|   \\
\lesssim & \sum_{l = 0}^{k+1} (\| U, V \|_{\mathcal{X}_{\le k-l + \frac 1 2}} + \|  U, V \|_{\mathcal{X}_{\le k+1-l}} ),
\end{align}
which establishes \eqref{L2:uv:eq}. The analogous proof works for the second quantity in \eqref{L2:uv:eq}. 

For the first quantity in \eqref{L2:uv:eq:2}, we obtain the identity 
\begin{align}
u_y^{(m)} = \sum_{l =0}^m \binom{m}{l} (\p_x^l \bar{u} \p_x^{m-l} U_y + 2 \p_x^l \bar{u}_y \p_x^{m-l} U + \p_x^l \bar{u}_{yy} \p_x^{m-l} q).
\end{align}
We now estimate 
\begin{align} \n
\| u_y^{(m)} x^m \phi_m \| \lesssim &\sum_{l = 0}^m \Big\| \frac{\p_x^l \bar{u}}{\bar{u}} \langle x \rangle^l \Big\|_\infty \| \sqrt{\bar{u}} \p_x^{m-l} U_y \langle x \rangle^{m-l} \phi_{m-l} \| \\ \n
&+ \sum_{l = 0}^{m-1} \| \p_x^l \bar{u}_y \langle x \rangle^{l + \frac 1 2}\|_\infty \| \p_x^{m-1-l} U_x \langle x \rangle^{m-1-l +\frac 1 2} \phi_{m-1-l} \| \\ \n
& + \| \p_x^m \bar{u}_y y x^m \|_\infty \Big\| \frac{U - U(x, 0)}{y} \Big\|  + \| \p_x^m \bar{u}_y \langle x \rangle^{m+\frac 1 4} \|_{L^\infty_x L^2_y} \| U(x, 0) \langle x \rangle^{-\frac 1 4}  \|_{L^2_x} \\ \n
&+ \sum_{l = 0}^{m-1} \| \p_x^l \bar{u}_{yy} \langle x \rangle^{l + \frac 1 2} y \|_\infty \| \p_x^{m-1-l}  \frac{q_x}{y} \langle x \rangle^{m-1-l +\frac 1 2} \phi_{m-1-l} \| \\
& + \| \p_x^m \bar{u}_{yy} y^2 x^m \|_\infty \| \frac{q - y U(x, 0)}{\langle y \rangle^2} \|, 
\end{align}
where we use above that $m \ge 1$ in \eqref{L2:uv:eq:2}.

We now arrive at the mixed norm estimates in \eqref{mixed:L2:orig:1}. For this, we first record the identity 
\begin{align}
\p_x^j v = \sum_{l = 0}^{j} \binom{j}{l} ( \p_x^l \bar{u} \p_x^{j-l} V - \p_x^l \bar{u}_x \p_x^{j-l} q ).
\end{align}
From here, we compute 
\begin{align} \n
\| \frac{1}{\bar{u}} \p_x^j v \langle x \rangle^j \phi_j \|_{L^2_x L^\infty_y} \lesssim &\sum_{l = 0}^{j-1} \Big\| \frac{\p_x^l \bar{u}}{ \bar{u}} \langle x \rangle^l \Big\|_\infty \| \p_x^{j-l} V \langle x \rangle^{j-l} \phi_{j} \|_{L^2_x L^\infty_y} \\ \n
&+ \sum_{l = 0}^{j-2} \Big\| \frac{\p_x^{l+1} \bar{u}}{ \bar{u}} \langle x \rangle^{l+1} \Big\|_\infty \| \p_x^{j-l} q \langle x \rangle^{j-l-1} \phi_{j} \|_{L^2_x L^\infty_y} \\ \n
& + \| \p_x^j \bar{u} \langle x \rangle^{j - \frac 1 2} y \|_\infty \Big\| \frac{V}{y} \langle x \rangle^{\frac 1 2} \phi_j\Big\|_{L^2_x L^\infty_y} + \| \p_x^{j+1} \bar{u} y \langle x \rangle^{j + \frac 1 2} \|_\infty \Big\|  \frac{q}{y} \langle x \rangle^{- \frac 1 2} \phi_j \Big\|_{L^2_x L^\infty_y} \\ \label{RAC:1}
\lesssim & \sum_{l = 0}^{j-1} \| U, V \|_{\mathcal{X}_{\le l+1}}  + \Big\| \frac{V}{y} \langle x \rangle^{\frac 1 2} \phi_j \Big\|_{L^2_x L^\infty_y} +\Big\|  \frac{q}{y} \langle x \rangle^{- \frac 1 2}  \phi_j \Big\|_{L^2_x L^\infty_y},
\end{align}
where we have invoked \eqref{Lpq:emb:V:1}. To conclude, we need to estimate the final two terms appearing above. First, we have by using $V|_{y = 0} = 0$, 
\begin{align}
\Big\| \frac{V}{y} \langle x \rangle^{\frac 1 2} \phi_j \Big\|_{L^2_x L^\infty_y} \lesssim \| U_x \langle x \rangle^{\frac 1 2} \phi_j \|_{L^2_x L^\infty_y} \lesssim \| U, V \|_{\mathcal{X}_{\le 1.5}},
\end{align}
where above we have used the estimate 
\begin{align} \n
\langle x \rangle U_x^2 = & |\int_y^\infty \langle x \rangle U^{(1)} U^{(1)}_{y} \ud y'| \lesssim \| U_x \langle x \rangle^{\frac 1 2} \|_{L^2_y} \| U_{xy} \langle x \rangle \|_{L^2_y} \\ \n
\lesssim & (\| \bar{u} U_x \langle x \rangle^{\frac 1 2} \|_{L^2_y} + \| \sqrt{\bar{u}} U_{xy} \langle x \rangle \|_{L^2_y} )(\| \sqrt{\bar{u}} U^{(1)}_y \langle x \rangle \|_{L^2_y} + \| \sqrt{\bar{u}} U^{(1)}_{yy} \langle x \rangle^{\frac 3 2} \|_{L^2_y} ), 
\end{align}
which, upon taking supremum in $y$ and subsequently integrating in $x$, yields 
\begin{align}
\| U_x \langle x \rangle^{\frac 1 2}\phi_j \|_{L^2_x L^\infty_y}^2 \lesssim \| U, V \|_{\mathcal{X}_{\le 1.5}}^2.
\end{align}
An analogous estimate applies to the third term from \eqref{RAC:1}. The estimate \eqref{mL2again} works in a nearly identical manner, invoking \eqref{Lpq:emb:V:2} instead of \eqref{Lpq:emb:V:1}. 
\end{proof}

\subsection{Pointwise Decay Estimates} \label{pw:section:1}

A crucial feature of space $\mathcal{X}$ is that it is strong enough to control sharp pointwise decay rates of various quantities, which are in turn used to control the nonlinearity. To be precise, we need to treat large values of $x$ (the more difficult case) in a different manner than small values of $x$. Large values of $x$ will be treated through the weighted norms in \eqref{X:norm}, whereas small values of $x$ will be treated with the $\dot{H}^k$ component of \eqref{X:norm}, at an expense of $\eps^{-M_1}$. We recall from \eqref{def:phi:j} that $\phi_{12}(x)$ is only non-zero when $\phi_{j} = 1$ for $1 \le j \le 11$. 

\begin{lemma} For $0 \le k \le 8$, and for $j = 0, 1$,  
\begin{align} \label{UV:decay}
&\| \p_y^j U^{(k)} \langle x \rangle^{k + \frac 1 4 + \frac j 2} \phi_{12}\|_{L^\infty_y} \lesssim \| U, V \|_{\mathcal{X}}, &&\| \sqrt{\eps} V^{(k)} \langle x \rangle^{k + \frac 1 2} \phi_{12} \|_{L^\infty_y} \lesssim \| U, V \|_{\mathcal{X}}.  
\end{align}
\end{lemma}
\begin{proof} We first establish the $U$ decay via 
\begin{align} \n
U_x^2 \langle x \rangle^{\frac 5 2} \phi_{12}^2 = &\Big| - \int_y^\infty 2 U_x U_{xy} \langle x \rangle^{\frac 5 2} \phi_{12}^2 \Big|\\ \n
\le & \Big| \int_y^\infty \int_x^\infty 2 U_{xx} U_{xy} \langle x \rangle^{\frac 5 2} \phi_{12}^2 \Big| + \Big| 2 \int_y^\infty \int_x^\infty U_x U_{xxy} \langle x \rangle^{\frac 5 2} \phi_{12}^2\Big| \\ \n
&+  5 \Big| \int_y^\infty \int_x^\infty U_x U_{xy} \langle x \rangle^{\frac 3 2}\phi_{12}^2 \Big| + \Big| \int_y^\infty \int_x^\infty 4 U_x U_{xy} \phi_{12} \phi_{12}' \Big| \\   \n
\lesssim & \| U_{xx} \langle x \rangle^{\frac 3 2} \phi_{12} \| \| U_{xy} \langle x \rangle  \phi_{12} \| + \| U_x \langle x \rangle^{\frac 1 2} \phi_{12} \| \| U_{yxx} \langle x \rangle^2 \phi_{12} \| \\ 
&+ \| U_x \langle x \rangle^{\frac 1 2} \phi_{12}\| \| U_{xy} \langle x \rangle \phi_{12} \| + \| U_x \phi_{11} \| \| U_{xy} \phi_{11} \|  .
\end{align}

We now perform the same calculation for the $V$ decay in \eqref{UV:decay}. We begin with $V_x$, via
\begin{align} \n
\eps V_x^2 \langle x \rangle^{3} \phi_{12}^2 = &\Big| \int_y^\infty 2 \eps V_x V_{xy} \langle x \rangle^3 \phi_{12}^2 \ud y' \Big| \\ \n
\le & \int_y^\infty \int_x^\infty 2 \eps | V_{xx} | | V_{xy} | \langle x \rangle^{3} \phi_{12}^2 \ud y' \ud x' + \int_y^\infty \int_x^\infty 2 \eps | V_{x} | | V_{xxy} | \langle x \rangle^{3} \phi_{12}^2 \ud y' \ud x' \\ \n
& + \int_y^\infty \int_x^\infty 6 \eps | V_{x} | | V_{xy} | \langle x \rangle^{2} \phi_{12}^2 \ud y' \ud x' + |\int_y^\infty \int_x^\infty 4\eps |V_x| |V_{xy}| \langle x \rangle^{3} \phi_{12} \phi_{12}' \ud y' \ud x' \\ \n
\lesssim & \sqrt{\eps} \| \sqrt{\eps} V^{(1)}_x \langle x \rangle^{\frac 3 2} \phi_{12}\| \| U^{(1)}_x \langle x \rangle^{\frac 3 2} \phi_{12} \| + \sqrt{\eps} \| \sqrt{\eps} V_x \langle x \rangle^{\frac 1 2} \phi_{12} \| \| U^{(2)}_x \langle x \rangle^{2.5} \phi_{12} \| \\
& + \sqrt{\eps} \| \sqrt{\eps} V_x \langle x \rangle^{\frac 1 2} \phi_{12} \| \| U^{(1)}_x \langle x \rangle^{\frac 3 2} \phi_{12} \| + \sqrt{\eps} \| \sqrt{\eps} V_x \langle x \rangle^{\frac 1 2} \phi_{11} \| \| U^{(1)}_x \langle x \rangle^{\frac 3 2} \phi_{11} \|. 
\end{align}
From here the result follows upon invoking \eqref{Hardy:three} - \eqref{Hardy:four}. The same computation can be done for higher derivatives, and for $U, V$ themselves we use the estimate 
\begin{align}
U \phi_{12} = \phi_{12} \int_x^\infty U_x \lesssim \|U, V \|_{\mathcal{X}} \int_x^\infty \langle x' \rangle^{- \frac 5 4} \ud x' \lesssim \langle x \rangle^{- \frac 1 4} \| U, V \|_{\mathcal{X}},
\end{align}
and similarly for $V$, we integrate 
\begin{align}
V \phi_{12} = \phi_{12} \int_x^\infty V_x \lesssim \eps^{- \frac 1 2} \|U, V \|_{\mathcal{X}} \int_x^\infty \langle x' \rangle^{- \frac 3 2} \ud x' \lesssim \langle x \rangle^{- \frac 1 2} \eps^{- \frac 1 2} \| U, V \|_{\mathcal{X}}.
\end{align}
This concludes the proof. 
\end{proof}

%We now establish similar inequalities with the additional weight of $\bar{u}$, which has the effect of ``saving one derivative". 
%\begin{lemma}For $0 \le k \le 10$, 
%\begin{align}
%&\| \sqrt{\bar{u}} U^{(k)} \langle x \rangle^{k + \frac 1 4} \|_{L^\infty_y} \lesssim \| U, V \|_{\mathcal{X}}, &&\| \sqrt{\bar{u}} \sqrt{\eps} V^{(k)} \langle x \rangle^{k + \frac 1 2} \|_{L^\infty_y} \lesssim \| U, V \|_{\mathcal{X}}.  
%\end{align}
%\end{lemma}
%\begin{proof}
%
%\end{proof}

It is also necessary that we establish decay estimates on the original unknowns $(u, v)$. For this purpose, we define another auxiliary cut-off function in the following manner 
\begin{align} \label{psi:twelve:def}
\psi_{12}(x) := \begin{cases} 0 \text{ for } 0 \le x \le 60 \\ 1 \text{ for } x \ge 61 \end{cases}
\end{align}
The main point in specifying $\psi_{12}$ in this manner is so that its support is contained where $\psi_j = 1$ for $j = 2,...,11$ and simultaneously $\psi_{12} = 1$ on the set where $\phi_1$ is supported.  Then, we have the following lemma.  
\begin{lemma} For $1 \le k \le 8$, and $j = 1, 2$, 
\begin{align} \label{pw:dec:u}
\| u^{(k)} x^{k + \frac 1 4} \psi_{12} \|_\infty + \| u^{(k)}_y x^{\frac 3 4} \psi_{12}  \|_\infty + \| \frac{1}{\bar{u}} u^{(k)} x^{k + \frac 1 4}  \psi_{12}\|_\infty \lesssim & \eps^{-M_1} \| U, V \|_{\mathcal{X}}   \\ \label{pw:v:1}
\| v^{(k)} x^{k + \frac 1 2} \psi_{12} \|_\infty  + \| \frac{v^{(k)}}{\bar{u}} x^{k + \frac 1 2} \psi_{12} \|_\infty \lesssim &  \eps^{-M_1} \| U, V \|_{\mathcal{X}}  \\ \label{mixed:emb}
\| \p_y^j u^{(k)} x^{k + \frac j 2} \psi_{12}\|_{L^\infty_x L^2_y} \lesssim &  \eps^{-M_1} \| U, V \|_{\mathcal{X}}.
\end{align}
In addition, 
\begin{align} \label{Linfty:wo}
\| u \langle x \rangle^{\frac 1 4} \|_\infty + \| \sqrt{\eps} v \langle x \rangle^{\frac 1 2} \|_\infty \lesssim  \eps^{-M_1} \| U, V \|_{\mathcal{X}}
\end{align}
\end{lemma}
\begin{proof} We note that  standard Sobolev embeddings gives $\| u^{(k)} \psi_{12} \|_\infty \lesssim \| u^{(k)} \psi_{12} \|_{H^2} \lesssim \eps^{-M_1} \| U, V \|_{\mathcal{X}}$, and similarly for the remaining quantities in \eqref{pw:dec:u} - \eqref{mixed:emb}. Next, we appeal to \eqref{formula:1} to obtain 
\begin{align} \n
\| u^{(k)} x^{k + \frac 1 4} \phi_{12} \|_{\infty}  \le & \sum_{l = 0}^k \binom{k}{l} (\|  \p_x^l \bar{u} \p_x^{k-l} U x^{k + \frac 1 4} \phi_{12} \|_\infty + \| \p_x^l \bar{u}_y \p_x^{k-l} q x^{k + \frac 1 4} \phi_{12} \|_\infty) \\ \n
\lesssim & \sum_{l = 0}^k \| \p_x^l \bar{u} x^l \|_\infty \| \p_x^{k-l} U x^{k-l + \frac 1 4} \phi_{12} \|_\infty + \| \p_x^l \bar{u}_y y x^{l} \|_\infty \| \p_x^{k-l} \frac{q}{y} x^{k-l + \frac 1 4} \phi_{12} \|_\infty \\ \n
\lesssim &\sum_{l = 0}^k  \| \p_x^{k-l} U x^{k-l + \frac 1 4} \phi_{12} \|_\infty \lesssim \| U, V \|_{\mathcal{X}}, 
\end{align}
where we have invoked the Hardy inequality (in $L^\infty_y$), admissible as $q|_{y = 0} = 0$, as well as estimate \eqref{UV:decay}. To conclude, by using that $x < 400$ is bounded on the set  $\{\psi_{12} = 1 \} \cap \{ \phi_{12} < 1 \}$, we have 
\begin{align} \n
\| u^{(k)} x^{k + \frac 1 4} \psi_{12} \|_\infty \le &  \| u^{(k)} x^{k + \frac 1 4} \psi_{12} (1 - \phi_{12}) \|_\infty +  \| u^{(k)} x^{k + \frac 1 4} \psi_{12} \phi_{12} \|_\infty \\
\lesssim  & \eps^{-M_1} \| U, V \|_{\mathcal{X}} + \| u^{(k)} x^{k + \frac 1 4}  \phi_{12} \|_\infty \lesssim \eps^{-M_1} \| U, V \|_{\mathcal{X}}. 
\end{align}
The remaining estimates in \eqref{pw:dec:u} - \eqref{Linfty:wo} work in largely the same manner. 
\end{proof}

\section{Global \textit{a-priori} Bounds} \label{remainder:section:3}

In this section, we perform our main energy estimates, which control the $\|U, V \|_{X_0}, \|U, V \|_{X_{\frac 1 2}}, \|U, V \|_{Y_{\frac 1 2}}$, and their higher order counterparts, up to $\|U, V \|_{X_{10}}, \|U, V \|_{X_{10.5}}, \|U, V \|_{Y_{10.5}}$. When we perform these estimates, we recall the notational convention for this section, which is that, unless otherwise specified $\int g := \int_0^\infty \int_0^\infty g(x, y) \ud y \ud x$. For this section, we define the operator 
\begin{align}
\text{div}_\eps(\bold{M}) :=  \p_x M_1 + \frac{\p_y}{\eps}  M_2, \text{ where } \bold{M} = (M_1, M_2). 
\end{align}

\subsection{$X_0$ Estimates}

%Delta Estimate%

\begin{lemma} \label{Lem:2} Let $(U, V)$ be a solution to \eqref{sys:sim:1} - \eqref{sys:sim:3}. Then the following estimate is valid, 
\begin{align} \label{basic:X0:est:st}
\| U, V \|_{X_{0}}^2 \lesssim \mathcal{T}_{X_0} + \mathcal{F}_{X_0}, 
\end{align}
where 
\begin{align} \label{def:TX0}
&\mathcal{T}_{X_0} := \int \mathcal{N}_1 U g(x)^2  +  \int \eps \mathcal{N}_2 (V g(x)^2 +  \frac{1}{100} q \langle x \rangle^{-1 - \frac{1}{100}}), \\ \label{def:FX0}
&\mathcal{F}_{X_0} := \int F_R U g(x)^2 + \int \eps G_R (V g(x)^2 +  \frac{1}{100} q \langle x \rangle^{-1 - \frac{1}{100}}). 
\end{align}
\end{lemma}
\begin{proof} We apply the multiplier 
\begin{align} \label{mult:X0}
\bold{M}_{X_0} := [U g^2, \eps V g^2 - \eps q \p_x (g^2)] = [U g^2, \eps V g^2 + \eps \frac{1}{100} q \langle x \rangle^{-1 - \frac{1}{100}}]
\end{align}
to the system \eqref{sys:sim:1} - \eqref{sys:sim:3}. We note that $\text{div}_\eps(\bold{M}_{X_0}) = 0$, and moreover that the normal component vanishes at $y = 0, y = \infty$. Therefore, the pressure term vanishes via 
\begin{align} \n
&\int P_x U g^2 + \int \frac{P_y}{\eps} (\eps V g^2 + \eps \frac{1}{100} q \langle x \rangle^{-1 - \frac{1}{100}} ) = - \int P \text{div}_\eps (\bold{M}_{X_0})= 0. 
\end{align}

\vspace{3 mm}

\noindent \textit{Step 1: $\mathcal{T}_1[U]$ Terms:} We now arrive at the transport terms from $\mathcal{T}_1$, defined in \eqref{def:T1}, which produce 
\begin{align} \n
\int  \mathcal{T}_1[U] U g(x)^2 \ud y \ud x = &\int (\bar{u}^2 U_x + \bar{u} \bar{v} U_y + 2 \bar{u}^0_{pyy} U) U g^2 \ud y \ud x \\ \n
= & \frac{1}{200}\int \bar{u}^2 U^2 \langle x \rangle^{-1-\frac{1}{100}} - \int \bar{u} \bar{u}_x U^2 g^2 - \frac 12 \int \p_y (\bar{u} \bar{v}) U^2 g^2 + \int 2 \bar{u}^0_{pyy} U^2 g^2 \\ \n
= & \frac{1}{200} \int \bar{u}^2 U^2 \langle x \rangle^{-1 - \frac{1}{100}} - \frac 1 2 \int (\bar{u} \bar{u}_x + \bar{v} \bar{u}_y) U^2 g^2 + \int 2 \bar{u}^0_{pyy} U^2 g^2 \\ \label{T1:cont}
= & \frac{1}{200} \int \bar{u}^2 U^2 \langle x \rangle^{-1 - \frac{1}{100}} + \frac 3 2 \int \bar{u}^0_{pyy} U^2 g^2 - \frac 1 2 \int \zeta U^2 g^2 =: \sum_{i = 1}^3 T^0_i,  
\end{align}
where we have invoked \eqref{def:zeta}. The term $T^0_1$ is a positive contribution towards the $X_0$ norm. The term $T^0_2$ will be cancelled out below, see \eqref{cancel:now:0}, and so we do not need to estimate it now. The third term from \eqref{T1:cont}, $T^0_3$, will be estimated via 
\begin{align}  \label{pourquoi}
\Big| \int \zeta U^2 g^2 \Big| \lesssim & \sqrt{\eps} \int U^2 \langle x \rangle^{-(1+\frac{1}{50})} \lesssim  \sqrt{\eps} \Big( \| \bar{u} U \langle x \rangle^{-\frac 1 2 - \frac{1}{200}} \|^2 +  \| \sqrt{\bar{u}} U_y \|^2  \Big) \lesssim \sqrt{\eps} \| U, V \|_{X_0}^2, 
\end{align}
where we have invoked the pointwise estimates \eqref{S:0}, as well as the Hardy-type inequality \eqref{Hardy:1}. 

We note that an analogous estimate applies to the $q$ term from \eqref{sys:sim:1}, which we record now 
\begin{align} \label{hardy:q:est}
\Big| \int  \zeta_y q U g^2 \Big| \lesssim \sqrt{\eps} \| \frac{q}{y} \langle x \rangle^{- \frac 1 2 - \frac{1}{100}} \|  \| U \langle x \rangle^{- \frac 1 2 - \frac{1}{100}} \| \lesssim \sqrt{\eps} \| U \langle x \rangle^{- \frac 1 2 - \frac{1}{100}} \|^2, 
\end{align}
where we have invoked the pointwise estimates \eqref{S:0} on the quantity $|y \p_y \zeta|$, and used the standard Hardy inequality in $y$, admissible as $q|_{y = 0} = 0$.  Estimate \eqref{hardy:q:est} concludes in the same manner as \eqref{pourquoi}. 

\noindent \textit{Step 2: Diffusive Terms} We would now like to treat the diffusive terms from \eqref{sys:sim:1} - \eqref{sys:sim:2}. We group the term $\bar{u}^0_{pyyy}q$ from \eqref{sys:sim:1} with this treatment for the purpose of achieving a cancellation. More precisely, we will begin by treating the following quantity:
\begin{align} \n
\int \Big( - \p_y^2 u + \bar{u}^0_{pyyy} q \Big)U g^2 = & \int u_y U_y g^2 + \int_{y = 0} u_y U g^2 \ud x - \frac 1 2 \int \bar{u}^0_{pyyyy} q^2 g^2 \\ \n
= & \int \p_y (\bar{u} U + \bar{u}_y q) U_y g^2 + \int_{y = 0} u_y U g^2 \ud x -  \frac 1 2 \int \bar{u}^0_{pyyyy} q^2  g^2 \\ \n
= & \int \bar{u} U_y^2 g^2 + 2 \bar{u}_y U U_y g^2 + \bar{u}_{yy} q U_y g^2 + \int_{y = 0} u_y U g^2  \ud x \\ \n
& - \frac 1 2 \int \bar{u}^0_{pyyyy} q^2 g^2\\ \n
= & \int \bar{u} U_y^2 g^2 - 2\bar{u}_{yy} U^2 g^2+ \frac 1 2 \bar{u}_{yyyy} q^2 g^2 - \frac 1 2 \bar{u}^0_{pyyyy}q^2 g^2\\ \n
& - \int_{y = 0} \bar{u}_y U^2 g^2 + \int_{y = 0} u_y Ug^2 \ud x \\  \n
= & \int \bar{u} U_y^2 g^2- 2\int \bar{u}_{yy} U^2 g^2 + \frac 1 2 \int \p_y^4 (\bar{u} - \bar{u}^0_p) q^2 g^2+ \int_{y = 0} \bar{u}_y U^2 g^2 \ud x \\ \label{earl:2}
= & D^0_1 + D^0_2 + D^0_3 + D^0_4. 
\end{align}
We have used that 
\begin{align}
u_y|_{y = 0} = ( \bar{u} U_y + 2 \bar{u}_y U + \bar{u}_{yy}q)|_{y = 0} = 2 \bar{u}_y U|_{y = 0}. 
\end{align}
Both $D^0_1, D^0_4$ are positive contributions, thanks to \eqref{prime:pos}. We now note that the main contribution from the $D^0_2$ term cancels the contribution $T^0_2$, and generates a positive damping term of 
\begin{align} \n
T^0_2 + D^0_2 = & \frac 3 2 \int \bar{u}^0_{pyy} U^2 g^2 - 2 \int \bar{u} U^2 g^2 = - \frac 1 2 \int \bar{u}^0_{pyy} U^2 g^2 - 2 \int (\bar{u}_{yy} - \bar{u}^0_{pyy}) U^2 g^2 \\ \label{cancel:now:0}
= &- \frac 1 2 \int \p_{yy} \bar{u}_\ast U^2 g^2 - \frac 1 2 \int \p_y^2 (\bar{u}^0_{p} - \bar{u}_\ast) U^2 g^2 - 2 \int (\bar{u}_{yy} - \bar{u}^0_{pyy}) U^2 g^2 \\ \n
= :& D^0_{2,1} + D^0_{2,2} + D^0_{2,3}. 
\end{align}
The first term on the right-hand side above, due to $\bar{u}_\ast$, is a positive contribution due to \eqref{Blas:prop:2}. For the term $D^0_{2,2,}$, we estimate via 
\begin{align*}
| \frac 1 2 \int \p_y^2 (\bar{u}^0_{p} - \bar{u}_\ast) U^2 g^2| \lesssim \delta_\ast \int \langle x \rangle^{- \frac 5 4+ \sigma_\ast} U^2 \lesssim \delta_\ast \|U, V \|_{X_0}^2,
\end{align*}
where we have appealed to estimate \eqref{blas:conv:1}. 

The remaining contribution, $D^0_{2,2}$, we estimate by invoking the pointwise estimate $|\bar{u}_{yy} - \bar{u}^0_{pyy}| \lesssim \sqrt{\eps} \langle x \rangle^{-1-\frac{1}{50}}$ due to \eqref{est:ring:1}. The third term from \eqref{earl:2}, $D^0_3$, is estimated by 
\begin{align}
|\int \p_y^4 (\bar{u} - \bar{u}^0_p) q^2 g^2| \lesssim \sqrt{\eps} \| \frac{q}{y} \langle x \rangle^{- \frac 1 2 - \frac{1}{100}} \|^2 \lesssim \sqrt{\eps} \| U \langle x \rangle^{- \frac 1 2 - \frac{1}{100}} \|^2 \lesssim \sqrt{\eps} \| U, V \|_{X_0}^2, 
\end{align}
where we have invoked estimate \eqref{est:ring:1} and subsequently the standard Hardy inequality in $y$, as $q|_{y = 0} = 0$.  

The next diffusive term is 
\begin{align} \label{home:home:1}
- \int \eps u_{xx} U g(x)^2 = \int \eps u_x U_x g(x)^2 + 2 \int \eps u_x U g(x) g'(x). 
\end{align}
upon using that $U|_{x = 0} = 0$. The $g'$ term above is easily controlled by $\sqrt{\eps} \| U \langle x \rangle^{- \frac 1 2 - \frac{1}{200}} \|^2 + \sqrt{\eps} \| \sqrt{\eps} \sqrt{\bar{u}} U_x g \|^2 + \sqrt{\eps} \| \sqrt{\eps}V \langle x \rangle^{- \frac 1 2 - \frac{1}{200}} \|^2$ upon consulting the definition of $g$ to compute $g'$, and definition \eqref{formula:1} to expand $u_x$ in terms of $(U, V, q)$. 

%We estimate the $g'$ term above as an error term by expanding with the help of \eqref{formula:1} to obtain 
%\begin{align} \n
%\int \eps u_x U g g'(x) = & \int \eps \p_x (\bar{u} U + \bar{u}_y q) U g g'(x) \\ \n
%= & \int \eps \bar{u} U_x U g g'(x) + \int \eps \bar{u}_x U^2 g g'(x) - \int \eps \bar{u}_y U V g g'(x)+ \int \eps \bar{u}_{xy} q Ug g'(x) \\ 
%= & - \frac 1 2 \int \bar{u} \eps U^2 (g^2)'' + \frac 1 2 \int \eps \bar{u}_x U^2 gg'(x)  - \int \eps \bar{u}_y UV gg'(x) - \frac 1 2 \int \eps \bar{u}_{xyy} q^2 gg'(x). 
%\end{align}

We now address the first term on the right-hand side of \eqref{home:home:1}, which yields, 
\begin{align} \n
 \int \eps u_x U_x g^2 =& \int \eps \p_x (\bar{u} U + \bar{u}_y q) U_x g^2 \\ \n
= &\int \eps \bar{u} U_x^2 g^2 + \int \eps \bar{u}_x U U_x g^2 + \int \eps \bar{u}_{xy} q U_x g^2 - \int \eps \bar{u}_y U_xV g^2 \\ \n
= & \int \eps \bar{u} U_x^2 g^2- \frac 1 2 \int \eps \bar{u}_{xx} U^2g^2 - \eps  \int \bar{u}_x gg' U^2 - \int \eps \bar{u}_{xy} q V_y g^2 + \int \eps \bar{u}_y VV_y g^2 \\ \n
= & \int \eps \bar{u} U_x^2g^2 - \frac 1 2 \int \eps \bar{u}_{xx} U^2 g^2 - \eps  \int \bar{u}_x gg' U^2 + \int \eps \bar{u}_{xyy} q V g^2 + \int \eps \bar{u}_{xy} UV g^2 \\ \n
& - \frac 1 2\int \eps \bar{u}_{yy} V^2 g^2\\ \n
= &  \int \eps \bar{u} U_x^2 g^2- \frac 1 2 \int \eps \bar{u}_{xx} U^2g^2  + \frac 1 2 \int \eps \bar{u}_{xxyy} q^2 g^2 + \int \eps \bar{u}_{xy} UV g^2 \\ \label{juice:0}
&- \frac 1 2\int \eps \bar{u}_{yy} V^2 g^2  +  \int \eps \bar{u}_{xyy} q^2 gg'  - \eps  \int \bar{u}_x g g' U^2 = \sum_{i = 1}^7 D^1_i. 
\end{align} 
We observe that $D^1_1$ is a positive contribution. We estimate $D^1_7$ via 
\begin{align}
|\eps \int \bar{u}_x gg' U^2| \lesssim \eps \| U \langle x \rangle^{-1} \|^2 \lesssim \eps \| U, V \|_{X_0}^2, 
\end{align}
and similarly for $D^1_6$, where we have invoked the pointwise decay estimate on $\bar{u}_x$ in \eqref{prof:u:est}. The remaining terms, $D^1_2,...,D^1_5$ will be placed into the term $\mathcal{J}_1$ defined below in \eqref{error:J1}.

We now arrive at the $v_{yy}$ diffusive term from \eqref{sys:sim:2}, which reads after integrating by parts in $y$, 
\begin{align}
- \int v_{yy} (\eps V g^2 + \eps \frac{1}{100} q \langle x \rangle^{-1 - \frac{1}{100}}) = \int \eps v_y V_y g^2 -  \frac{1}{100} \int \eps v U_y \langle x \rangle^{- 1 - \frac{1}{100}}.
\end{align}
The second contribution on the right-hand side above is easily estimated by appealing to \eqref{formula:1} and estimate \eqref{prof:v:est}, which generates
\begin{align}
|\int \eps (\bar{u} V - \bar{u}_x q) U_y \langle x \rangle^{- 1- \frac{1}{100}}| \lesssim \sqrt{\eps} \Big( \|  V \langle x \rangle^{- \frac 1 2 - \frac{1}{200}} \| + \sqrt{\eps} \| \bar{u}_{x} y \|_\infty \| U \langle x \rangle^{- \frac 1 2 - \frac{1}{200}} \| \Big)\| \sqrt{\bar{u}} U_y \|. 
\end{align}
For the first contribution on the right-hand side, we have 
\begin{align} \n
 \int \eps v_y V_y g^2= & \int \eps \p_y (\bar{u} V - \bar{u}_x q) V_y g^2 \\ \n
= & \int \eps \bar{u} V_y^2 g^2+ \int \eps \bar{u}_y V V_y g^2- \int \eps \bar{u}_{xy} q V_y g^2 - \int \eps \bar{u}_x U V_y g^2 \\ \n
= & \int \eps \bar{u} V_y^2 g^2 - \frac 1 2 \int \eps \bar{u}_{yy} V^2 g^2 + \frac 1 2 \int \eps \bar{u}_{xxyy} q^2 g^2+ \int \eps \bar{u}_{xy} UV g^2 \\ \label{tennis:1}
&- \frac 1 2 \int \eps \bar{u}_{xx} U^2 g^2 + \int \eps \bar{u}_{xyy} q^2 gg' - \int \eps \bar{u}_x U^2 gg' = \sum_{i = 1}^7 D^2_i. 
\end{align}
We observe that $D^2_1$ is a positive contribution, whereas $D^2_6, D^2_7$ are estimated identically to $D^1_6, D^1_7$. The remaining terms, $D^2_2,...,D^2_5$ will be placed into the term $\mathcal{J}_1$, defined below in \eqref{error:J1}.

We now arrive at the final diffusive term, for which we first integrate by parts using the boundary condition $V|_{x = 0} = q|_{x = 0} = 0$, 
\begin{align}  \label{bup:1}
&\int - \eps^2 v_{xx} V g^2 - 2 \int \eps^2 v_{xx} q gg'  = \int \eps^2 v_x V_x g^2 + \int 2 \eps^2 v_x q (gg')'.  
\end{align}
The second term above, which contains a $g'$ factor, is easily controlled by a factor of $\sqrt{\eps} \| U, V \|_{X_0}^2$ by again appealing to the definition of $g$ and \eqref{formula:1}. For the first term on the right-hand side of \eqref{bup:1}, we have
\begin{align} \n
 \int \eps^2 v_x V_x g^2= &\int \eps^2 \p_x (\bar{u} V - \bar{u}_x q) V_x g^2\\ \n 
 = & \int \eps^2 (\bar{u} V_x + 2 \bar{u}_x V - \bar{u}_{xx}q ) V_x g^2 \\ \n
= & \int \eps^2 \bar{u} V_x^2 g^2 - 2 \int \eps^2 \bar{u}_{xx} V^2 g^2+ \frac 1 2\int \eps^2 \bar{u}_{xxxx} q^2 g^2 \\ \label{juice:1} 
& + \int \eps^2 \bar{u}_{xx} q^2 (gg')' + \int 2 \eps^2 \bar{u}_{xxx} q^2 gg' - \int 2 \eps^2 \bar{u}_{xx} V^2 g^2 = \sum_{i = 1}^6 D^3_i.
\end{align}
The terms with $g'$ above are easily controlled by a factor of $\sqrt{\eps} \| U, V \|_{X_0}^2$ by again appealing to the definition of $g$ and estimate \eqref{prof:u:est}. 

We now expand the damping terms, which is are terms $D^1_5$ and $D^2_2$, 
\begin{align} \n
D^1_5 + D^2_2 = &- \int \eps \bar{u}_{yy} V^2 g^2 = - \int \eps \bar{u}^0_{pyy} V^2 g^2 - \int \eps^2 (\bar{u}_{yy} - \bar{u}^0_{pyy}) V^2 g^2.
\end{align} 
We estimate the latter term above via an appeal to \eqref{est:ring:1}, which gives 
\begin{align}
\Big| \int \eps^2 (\bar{u}_{yy} - \bar{u}^0_{pyy})  V^2 g^2 \Big| \lesssim \eps \| \sqrt{\eps} V \langle x \rangle^{- \frac 1 2 - \frac{1}{100}} \|^2 \lesssim \eps \| U, V \|_{X_0}^2. 
\end{align}

We now consolidate the remaining terms from \eqref{juice:0}, \eqref{tennis:1}, \eqref{juice:1}. Specifically, we obtain
\begin{align*}
D^1_2 + D^1_3 + D^1_4 + D^2_3 + D^2_4 + D^2_5 + D^3_2 + D^3_3 + D^3_6 = \mathcal{J}_1, 
\end{align*}
where we have defined
\begin{align} \label{error:J1}
\mathcal{J}_1 := - \int \eps \bar{u}_{xx} U^2  g^2 - 2 \int \eps^2 \bar{u}_{xx} V^2  g^2+ \int 2 \eps \bar{u}_{xy} UV  g^2+ \int \Big( \eps \bar{u}_{xxyy} + \frac 1 2 \eps^2 \bar{u}_{xxxx} \Big) q^2 g^2.
\end{align}
To estimate these contributions, we simply use the fact that $\| \bar{u}_{xx} \langle x \rangle^2 \|_\infty \lesssim 1$, $\| \bar{u}_{xy} \langle x \rangle^{\frac 3 2} \|_\infty  \lesssim 1$, $\| y^2 (\bar{u}_{xxyy} + \bar{u}^P_{xxxx}) \langle x \rangle^2 \|_\infty \lesssim 1$, and $\| \bar{u}_{E xxxx} \|_{L^\infty_y} \lesssim \sqrt{\eps} \langle x \rangle^{-\frac 9 2}$ according to the estimates \eqref{prof:u:est} - \eqref{prof:v:est}. 

\vspace{2 mm}

\noindent \textit{Step 3: $\mathcal{T}_2[V]$ Terms} We now treat the terms arising from $\mathcal{T}_2[V]$, which has been defined in \eqref{def:T2}. Specifically, the result of applying the multiplier \eqref{mult:X0} is
\begin{align} \n
&\int \mathcal{T}_2[V] (\eps V g^2 + \eps (.01) q \langle x \rangle^{-1.01} )  \\ \n
=& \int \eps \mathcal{T}_2[V] V g^2 + \eps (.01)  \int \bar{u}^2 V_x q \langle x \rangle^{-1.01} + \eps (.01) \int (\bar{u} \bar{v} V_y + \bar{u}^0_{pyy} V) q \langle x \rangle^{-1.01} \\  \label{out:1}
= :& \tilde{T}^1_1 + \tilde{T}^1_2 + \tilde{T}^1_3.
\end{align}
For the first term on the right-hand side of \eqref{out:1}, $\tilde{T}^1_1$, we have 
\begin{align} \n
\tilde{T}^1_1 = \int \eps \mathcal{T}_2[V] V g^2 = & \int \eps \bar{u}^2 V_x V g^2 + \int \eps \bar{u} \bar{v} V_y V g^2 + \int \eps \bar{u}^0_{pyy} V^2 g^2 \\ \n
= & - \int \eps \bar{u} \bar{u}_x V^2 g^2 - \frac 1 2 \int \eps (\bar{u} \bar{v})_y V^2 g^2 + \int \eps \bar{u}^0_{pyy} V^2 g^2 - \int \eps \bar{u}^2 V^2 gg' \\ \n
= & - \frac 1 2 \int \eps (\bar{u} \bar{u}_x + \bar{v} \bar{u}_y) V^2 g^2 + \int \eps \bar{u}^0_{pyy} V^2 g^2 - \int \eps \bar{u}^2 V^2 gg'  \\ \label{juice:2}
= & \frac 1 2 \int \eps \bar{u}^0_{pyy} V^2 g^2 - \frac 1 2 \int \eps \zeta V^2 g^2 + \frac{1}{200} \int \eps \bar{u}^2 V^2 \langle x \rangle^{-1 - \frac{1}{100}}  = \sum_{i = 1}^3 T^1_i,
\end{align}
where we have invoked \eqref{def:zeta}. The term $T^1_2$ is estimated in an analogous manner to \eqref{pourquoi}, whereas $T^1_3$ is a positive contribution to the $X_0$ norm. 

We now need to address the contribution of $\tilde{T}^1_2$. Integrating by parts, we get
\begin{align} \n
\tilde{T}^1_2 = &.01  \int \eps \bar{u}^2 V^2 \langle x \rangle^{-1.01} -2 (.01)  \int \eps \bar{u} \bar{u}_x V q \langle x \rangle^{-1.01} + (.01)(1.01) \int \eps \bar{u}^2 V q \langle x \rangle^{-2.01} \\ \n
= & .01  \int \eps \bar{u}^2 V^2 \langle x \rangle^{-1.01} -2 (.01)  \int \eps \bar{u} \bar{u}_x V q \langle x \rangle^{-1.01} \\ \label{juice:3}
& - \frac 1 2 (.01)(1.01) (2.01) \int \eps \bar{u}^2 q^2 \langle x \rangle^{-3.01} - \frac 1 2 (.01)(1.01) (2.01) \int \eps \bar{u} \bar{u}_x q^2 \langle x \rangle^{-2.01} \\ \n
= & \tilde{T}^1_{2,1} +  \tilde{T}^1_{2,2} +  \tilde{T}^1_{2,3} +  \tilde{T}^1_{2,4}.
\end{align}
Of these, the third term, $\tilde{T}^1_{2,3}$ is very dangerous due to a lack of decay in $z$ for the coefficient. To treat it, we combine $\tilde{T}^1_{2,1}, \tilde{T}^1_{2,3}$ and $T^1_3$ to obtain the expression 
\begin{align} \n
\tilde{T}^1_{2,1} + \tilde{T}^1_{2,3} + T^1_3 = &\frac{3}{2}(.01) \int \eps \bar{u}^2 V^2 \langle x \rangle^{-1.01} - \frac{(.01)(1.01)(2.01)}{2} \int \eps \bar{u}^2 q^2 \langle x \rangle^{-3.01} \\  \n
\ge &\Big(\frac{3}{2}(.01) - \frac{(.01)(1.01)(2.01)}{2} \frac{1}{1.01}  \Big) \int  \eps \bar{u}^2 V^2 \langle x \rangle^{-1.01}  -  \frac{(.01)(1.01)(2.01)}{2} \frac{2}{1.01} \int \langle x \rangle^{-2.01} \bar{u} \bar{u}_x q^2 \\ \label{precision:1}
\ge & \frac{.01}{2} \int \eps \bar{u}^2 V^2 \langle x \rangle^{-1.01} - (.01)(2.01) \int \eps \langle x \rangle^{-2.01} \bar{u} \bar{u}_x q^2,  
\end{align}
where we have used the precise constants appearing in \eqref{precise:1}. 

We now estimate the error term from \eqref{precision:1} by now splitting $\bar{u} = \bar{u}_P + \bar{u}_E$, according to \eqref{split:split:1}. First, we have 
\begin{align}
|\int \eps \langle x \rangle^{-2.01} \bar{u} \bar{u}_{Px} q^2| \lesssim \| \bar{u}_{Px} y^2 \|_\infty \eps \| \frac{q}{y} \langle x \rangle^{-1.01} \|^2 \lesssim \eps \| U \langle x \rangle^{-1.01} \|^2 \lesssim \eps \| U, V \|_{X_0}^2, 
\end{align}
where we have used estimate \eqref{est:Pr:piece}. For the $\bar{u}_E$ component, we may use the small amplitude and importantly the enhanced decay in $x$ from \eqref{est:Eul:piece} to obtain 
\begin{align} 
|\int \eps \langle x \rangle^{-2.01} \bar{u} \bar{u}_{Ex} q^2| \lesssim \eps^{\frac 3 2} | \int \langle x \rangle^{- 3.51} q^2|  \lesssim \eps^{\frac 1 2} \| \sqrt{\eps}V \langle x \rangle^{- \frac 3 4} \|^2 \lesssim \sqrt{\eps} \| U, V \|_{X_0}^2.
\end{align}
The error terms $\tilde{T}^1_{2,2}$ and $\tilde{T}^1_{2,4}$ are estimated in a nearly identical manner. 

We now address the third terms from \eqref{out:1}, $\tilde{T}^1_3$, which upon integration by parts in $y$ gives 
\begin{align} \n
&\eps (.01) |\int (\bar{u} \bar{v} V_y + \bar{u}^0_{pyy} V) q \langle x \rangle^{-1.01}| \lesssim  \eps | \int  (\bar{u}^0_{pyy} - (\bar{u} \bar{v})_y) V q \langle x \rangle^{-1.01}| + \eps | \int \bar{u} \bar{v} V U \langle x \rangle^{-1.01}| \\
\lesssim & \sqrt{\eps} \Big( \| (\bar{u}^0_{pyy} - (\bar{u} \bar{v})_y) y \langle x \rangle^{\frac 1 2} \|_\infty + \| \bar{v} \langle x \rangle^{\frac 1 2} \|_\infty  \Big) \| U \langle x \rangle^{- \frac 3 4} \| \| \sqrt{\eps} V \langle x \rangle^{- \frac 3 4} \| \lesssim \sqrt{\eps} \|U, V \|_{X_0}^2,  
\end{align}
where we have appealed to the estimates \eqref{prof:u:est} as well as \eqref{Hardy:three:a} - \eqref{Hardy:four:a}.

\vspace{2 mm}

\noindent \textit{Step 4: Remaining terms in \eqref{sys:sim:2}} We first treat the $\alpha U$ terms from \eqref{sys:sim:2}, which we estimate via 
\begin{align}
\Big| \int \eps \alpha UV g^2\Big| \lesssim \eps \int \langle x \rangle^{- \frac 3 2}|UV| \lesssim \sqrt{\eps} \| U \langle x \rangle^{- \frac 3 4} \| \| \sqrt{\eps} V \langle x \rangle^{- \frac 3 4} \|.
\end{align}
Above we have relied on the coefficient estimate in \eqref{S:1}. The $\p_y \alpha q$ term from \eqref{sys:sim:2} as well as the corresponding contributions from the $q$ term in the multiplier \eqref{mult:X0} work in an identical manner. This concludes the proof of Lemma \ref{Lem:2}. 
\end{proof}

\subsection{$\frac 1 2$ Level Estimates}

We now provide estimates on the two half-level norms, $\|U, V \|_{X_{\frac 1 2}}$ and $\|U, V \|_{Y_{\frac 1 2}}$. 

\begin{lemma} \label{Lem:3} Let $(U, V)$ be a solution to \eqref{sys:sim:1} - \eqref{BC:UVYW}. Then for $0 < \delta << 1$, 
\begin{align} \label{Xh:right}
\| U, V \|_{X_{\frac 1 2}}^2 \lesssim C_\delta \| U, V \|_{X_0}^2 + \delta \| U, V \|_{X_1}^2 + \delta \| U, V \|_{Y_{\frac 1 2}}^2 + \mathcal{T}_{X_{\frac 1 2}} + \mathcal{F}_{X_{\frac 1 2}}, 
\end{align}
where 
\begin{align} \label{TX12:spec}
\mathcal{T}_{X_{\frac 1 2}} := & \int \mathcal{N}_1 U_x x \phi_1^2 +  \int \mathcal{N}_2 ( \eps V_x x \phi_1(x)^2 + \eps V \phi_1(x)^2 + 2\eps V x \phi_1 \phi_1'), \\ \label{FX12:spec}
\mathcal{F}_{X_{\frac 1 2}} := & \int F_R U_x x \phi_1^2 +  \int G_R ( \eps V_x x \phi_1(x)^2 + \eps V \phi_1(x)^2 + 2\eps V x \phi_1 \phi_1')
\end{align}
\end{lemma}
\begin{proof} We apply the weighted in $x$ vector-field 
\begin{align} \label{mult:2}
\bold{M}_{X_{\frac 1 2}} := [U_x x \phi_1(x)^2 , \eps V_x x \phi_1(x)^2 + \eps V \phi_1(x)^2 + 2\eps V x \phi_1 \phi_1']
\end{align}
as a multiplier to \eqref{sys:sim:1} - \eqref{sys:sim:3}. We first of all notice that $\text{div}_\eps(\bold{M}_{X_{\frac 1 2}}) = 0$, and thus, 
\begin{align} \n
&\int P_x U_x x \phi_1^2 + \int \frac{P_y}{\eps} (\eps V_x x \phi_1^2 + \eps V \phi_1^2 + 2 \eps V x \phi_1 \phi_1') = - \int P \text{div}_\eps(\bold{M}_{X_{\frac 1 2}}) = 0, 
\end{align}
where we use that $V|_{y = 0} = V_x|_{y = 0} = 0$ and $\phi_1$ to eliminate any contributions from $\{x = 0\}$.  
\vspace{3 mm}

\noindent \textit{Step 1: $\mathcal{T}_1[U]$ terms:} We address the terms from $\mathcal{T}_1$ which produces 
\begin{align} \n
\int \mathcal{T}_1[U] U_x x \phi_1^2 = & \int \bar{u}^2 U_x^2 x \phi_1^2 + \int \bar{u} \bar{v} U_y U_x x \phi_1^2+ \int 2 \bar{u}_{yy} U U_x x \phi_1^2 \\ \label{speak:1}
= & \int \bar{u}^2 U_x^2 x \phi_1^2+ \int \bar{u} \bar{v} U_y U_x x \phi_1^2- \int \bar{u}_{yy} U^2 \phi_1^2- \int \bar{u}_{xyy} x U^2 \phi_1^2 \\ \n
& - 2\int \bar{u}_{yy} U^2 x \phi_1 \phi_1' = :\sum_{i = 1}^5 T^{(2)}_i.
\end{align}
First, we observe $T^{(2)}_1$ is a positive contribution. We estimate $T^{(2)}_2$ via 
\begin{align} 
&|\int \bar{v} \bar{u} U_y U_x x \phi_1^2 | \lesssim \| \frac{\bar{v}}{\bar{u}} x^{\frac 1 2} \|_\infty \| \sqrt{\bar{u}} U_y \| \| \bar{u} U_x x^{\frac 1 2} \phi_1 \| \le \delta \| \bar{u} U_x x^{\frac 1 2} \phi_1 \|^2 + C_\delta \| \sqrt{\bar{u}} U_y \|^2,
\end{align}
where above we have invoked estimate \eqref{prof:v:est} for $\bar{v}$.
 
For $T^{(2)}_3$, we need to split $U = U(x, 0) + (U - U(x, 0))$, and subsequently estimate via 
\begin{align} \n
|\int \bar{u}_{yy} U^2 \phi_1^2 | \lesssim & \int |\bar{u}_{yy}| (U - U(x, 0))^2 \phi_1^2 + \int |\bar{u}_{yy}| U(x, 0)^2 \phi_1^2 \\ \n
\lesssim & \| y^2 \bar{u}_{yy} \|_\infty \Big\| \frac{U - U(x, 0)}{y} \phi_1 \Big\|^2 +\Big( \sup_x \int |\bar{u}_{yy} | x^{\frac 1 2} \ud y \Big) \| U x^{-\frac 1 4} \|_{L^2(x = 0)}^2 \\ \label{fray:1}
\lesssim & \| U_y \phi_1\|^2 + \| U, V \|_{X_0}^2 \le C_\delta \| U, V \|_{X_0}^2 + \delta \| U, V \|_{Y_{\frac 1 2}}^2, 
\end{align}
where above, we used Hardy inequality in $y$, which is admissible as $(U - U(x, 0))|_{y = 0} = 0$, as well as the inequality \eqref{bob:1}. $T^{(2)}_4$ works in an analogous manner, and so we omit it. For $T^{(2)}_5$, we note that the support of $\phi_1'$ is bounded in $x$, and so this term can trivially be controlled by $\| U, V \|_{X_0}^2$.

\vspace{2 mm}

\noindent \textit{Step 2: $\mathcal{T}_2[V]$ terms:} We now address the contributions from $\mathcal{T}_2[V]$. For this, we first note that, examining the multiplier \eqref{mult:2}, the contribution from $\eps V$ has already been treated in Lemma \ref{Lem:2}, and we therefore just need to estimate the contribution from the principal term, $\eps V_x x$. More precisely, we have already established the following estimate,
\begin{align}
|\int \mathcal{T}_2[V] \eps (V \phi_1^2 + 2 V x \phi_1 \phi_1'  )|  \lesssim \|U, V \|_{X_0}^2.  
\end{align}
We now estimate the contribution of the principal term, $\eps V_x x$. For this, recall the definition \eqref{def:T2}, 
\begin{align} \n
\int \eps \mathcal{T}_2[V] V_x x \phi_1^2 = & \int \eps \Big( \bar{u}^2 V_x + \bar{u} \bar{v} V_y + \bar{u}_{yy} V  \Big) V_x x \phi_1^2 \\ \n
= & \int \eps \bar{u}^2 V_x^2 x \phi_1^2 + \int \eps \bar{u} \bar{v} V_y V_x x \phi_1^2  - \frac 1 2 \int \eps \p_x (x \bar{u}_{yy}) V^2 \phi_1^2 - \int \eps x \bar{u}_{yy} V^2 \phi_1 \phi_1' \\ \n
= & T^{(3)}_1 + ... + T^{(3)}_4. 
\end{align}
We observe that $T^{(3)}_1$ is a positive contribution. The integrand in the term $T^{(3)}_4$ has a bounded support of $x$ and so can immediately be controlled by $\|U, V \|_{X_0}^2$. We may estimate $T^{(3)}_2$, and $T^{(3)}_3$ via 
\begin{align}
&|\int \eps \bar{u} \bar{v} V_y V_x x \phi_1^2| \lesssim \sqrt{\eps} \| \frac{ \bar{v}}{\bar{u}} x^{\frac 1 2}  \|_\infty \| \bar{u} U_x x^{\frac 1 2} \phi_1 \| \sqrt{\eps} \bar{u} V_x x^{\frac 1 2}\phi_1 \| \lesssim \sqrt{\eps} \| U, V \|_{X_{\frac 1 2}}^2, \\
&|\int \eps \p_x (x \bar{u}_{yy}) V^2 \phi_1^2| \lesssim \| y^2 \p_x (x \bar{u}_{yy}) \|_\infty \eps \Big\| \frac{V}{y} \phi_1\Big\|^2 \lesssim \eps \| V_y \phi_1\|^2  \le C_\delta \| U, V \|_{X_0}^2 + \delta \| U, V \|_{Y_{\frac 1 2}}^2,
\end{align}
where we have invoked the Hardy inequality \eqref{Hardy:1}. 

\vspace{2 mm}

\noindent \textit{Step 3: Diffusive Terms} We now address the main diffusive term, which is the contribution of $- u_{yy}$ in \eqref{sys:sim:1}. We, again, group the term $\bar{u}^0_{pyyy}q$ from \eqref{sys:sim:1} with this term. More precisely, we have 
\begin{align} \n
 \int \Big(- u_{yy} +& \bar{u}^0_{pyyy} q \Big) U_x x\phi_1^2 \\ \n
 = &  \int u_y U_{xy} x \phi_1^2+ \int_{y = 0} u_y U_x x\phi_1^2 \ud x + \int \bar{u}^0_{pyyy} q U_x x\phi_1^2 \\ \n
 = & \int \p_y (\bar{u} U + \bar{u}_y q) U_{xy} x \phi_1^2+ \int_{y = 0} \bar{u}_y U_x x \phi_1^2 \ud x + \int \bar{u}^0_{pyyy} q U_x x \phi_1^2\\ \n
 = & \int (2 \bar{u}_y U + \bar{u} U_y + \bar{u}_{yy} q \Big) U_{xy} x \phi_1^2+ \int_{y = 0} \p_y (\bar{u} U + \bar{u}_y q) U_x x \phi_1^2 \ud x \\ \n
 & +  \int \bar{u}^0_{pyyy} q U_x x \phi_1^2 \\ \n
 = & - \int U_y \p_x (2 x \bar{u}_y U) \phi_1^2 - \frac 1 2 \int U_y^2 \p_x (x \bar{u}) \phi_1^2 - \int \bar{u}_{yyy} q U_x x \phi_1^2 \\ \n
 & + \int \bar{u}^0_{pyyy} q U_x x \phi_1^2 - \int \bar{u}_{yy} UU_x x\phi_1^2+ 2 \int_{y = 0} \bar{u}_y U U_x x\phi_1^2 \\ \n
 & -  \int \bar{u} x U_y^2 \phi_1 \phi_1' - \int 4 \bar{u}_y UU_y x \phi_1 \phi_1' \\ \n
 = & - \int 2 x \bar{u}_y U_x U_y \phi_1^2- \int 2 x \bar{u}_{xy} U U_y \phi_1^2- \int 2 \bar{u}_y UU_y \phi_1^2- \frac 1 2 \int \bar{u} U_y^2\phi_1^2 \\ \n
 & - \frac 1 2 \int x \bar{u}_x U_y^2\phi_1^2 + \frac 1 2 \int \p_x (x \bar{u}_{yy}) U^2\phi_1^2 - \int_{y = 0} \p_x (x \bar{u}_y) U^2\phi_1^2 - \int (\bar{u}_{yyy} - \bar{u}^0_{pyyy}) qU_x x\phi_1^2 \\ \n
 & + \int x \bar{u}_{yy} U^2 \phi_1 \phi_1' -2 \int_{y = 0} U^2 x \bar{u}_y \phi_1 \phi_1'  -  \int \bar{u} x U_y^2 \phi_1 \phi_1' - \int 4 \bar{u}_y UU_y x \phi_1 \phi_1' \\ \n
 = & - \int 2 x \bar{u}_y U_x U_y \phi_1^2+ \int x \bar{u}_{xyy} U^2\phi_1^2 + \int \bar{u}_{yy} U^2\phi_1^2 - \frac 1 2 \int \bar{u} U_y^2 \phi_1^2\\ \n
 & - \frac 1 2 \int x \bar{u}_x U_y^2 \phi_1^2+ \frac 1 2 \p_x (x \bar{u}_{yy}) U^2\phi_1^2 - \int_{y = 0} \p_x (x \bar{u}_{y}) U^2\phi_1^2 - \int (\bar{u}_{yyy} - \bar{u}^0_{pyyy}) qU_x x\phi_1^2 \\ \n
 &  + \int_{y = 0} \p_x ( x \bar{u}_y ) U^2 \phi_1^2 + \int x \bar{u}_{yy} U^2 \phi_1 \phi_1' -2 \int_{y = 0} U^2 x \bar{u}_y \phi_1 \phi_1'  -  \int \bar{u} x U_y^2 \phi_1 \phi_1' \\ \n
 &- \int 4 \bar{u}_y UU_y x \phi_1 \phi_1'  \\ \n
 = & - \int 2 x \bar{u}_y U_x U_y\phi_1^2 + \frac 3 2 \int \p_x (x \bar{u}_{yy}) U^2 \phi_1^2- \frac 1 2 \int (\bar{u} + x \bar{u}_x ) U_y^2\phi_1^2  \\ \n
 & - \int_{y = 0} \p_x (x \bar{u}_y) U^2 \phi_1^2 \ud x -\int (\bar{u}_{yyy} - \bar{u}^0_{pyyy}) qU_x x \phi_1^2 + \int x \bar{u}_{yy} U^2 \phi_1 \phi_1' \\ \label{million}
 & -2 \int_{y = 0} U^2 x \bar{u}_y \phi_1 \phi_1'  -  \int \bar{u} x U_y^2 \phi_1 \phi_1' - \int 4 \bar{u}_y UU_y x \phi_1 \phi_1' = \sum_{i = 1}^{9} D^{(4)}_i.
\end{align}
For the first term above, $D^{(4)}_1$, we localize in $z$ using the cutoff function $\chi(\cdot)$ (see \eqref{def:chi}), via 
\begin{align}
D_1^{(4)} = - \int 2 x \bar{u}_y U_x U_y\phi_1^2 (1 - \chi(z)) - \int 2 x \bar{u}_y U_x U_y\phi_1^2 \chi(z).
\end{align}
The far-field component is controlled easily via 
\begin{align}
|\int x \bar{u}_y U_x U_y (1- \chi(z)) \phi_1^2| \lesssim \| \bar{u}_y x^{\frac 1 2} \| \| \bar{u} U_x x^{\frac 1 2} \phi_1 \| \| \sqrt{\bar{u}} U_y \| \le \delta \| U, V \|_{X_{\frac 1 2}}^2 + C_\delta \| U, V \|_{X_0}^2,
\end{align}
where we have used $\bar{u} \gtrsim 1$ when $z \gtrsim 1$, according to \eqref{samezies:1}. 

The localized piece requires the use of higher order norms, and we estimate it via  
\begin{align} \n
|\int x \bar{u}_y U_x U_y \chi(z) \phi_1^2| \lesssim & \| \sqrt{x} \bar{u}_y \|_\infty \| U_x x^{\frac 1 2} \chi(z) \phi_1 \| \| U_y \phi_1 \|  \\ \n
\lesssim & ( \| \sqrt{\bar{u}} U_{xy} x\phi_1 \| + \| \bar{u} U_x \sqrt{x} \phi_1\|   ) ( \delta \| U, V \|_{Y_{\frac 1 2}} + C_\delta \| U, V \|_{X_0} ) \\ \label{est:Xhalf:loss:deriv}
\lesssim & \delta \| U, V \|_{X_1}^2 + \delta \| U, V \|_{X_{\frac 1 2}}^2 + \delta \| U, V \|_{Y_{\frac 1 2}}^2 + C_\delta \| U, V \|_{X_0}^2,
\end{align}
where above, we have appealed to \eqref{bob:1}. For $D^{(4)}_2$, we estimate in the same manner as \eqref{fray:1}, whereas $D^{(4)}_3$ can easily be controlled upon using $\| \bar{u} + x \p_x \bar{u} \|_\infty \lesssim 1$, according to \eqref{prof:u:est}. The term $D^{(4)}_4$ is immediately controlled by $\|U, V \|_{X_0}^2$. The term $D^{(4)}_5$ we estimate via 
\begin{align} \n
|\int (\bar{u}_{yyy} - \bar{u}^0_{pyyy}) q U_x x \phi_1^2| \lesssim & \| (\bar{u}_{yyy} - \bar{u}^0_{pyyy}) y x^{1.01} \|_\infty \| U \langle x \rangle^{-1.01} \| \| U_x \langle x \rangle^{\frac 1 2} \phi_1 \| \\
\lesssim & \sqrt{\eps} (\|U, V \|_{X_0}^2 + \| U, V \|_{X_{\frac 1 2}}^2 + \| U, V \|_{X_1}^2),
\end{align}
where we have invoked \eqref{est:ring:1}. Finally, for the remaining four terms from \eqref{million}, $D^{(4)}_k$, $k = 6,7,8,9$, due to the presence of $\phi_1'$, the $x$ weights are all bounded, and these terms can thus be easily controlled by $\|U, V \|_{X_0}^2$. 

We now move to the contribution of the tangential diffusive term, $-\eps u_{xx}$, which produces
\begin{align} \n
- \int  \eps u_{xx} U_x x \phi_1^2 = &\int  \eps u_x U_{xx} x \phi_1^2 + \int \eps u_x U_x \phi_1^2 + 2 \int \eps u_x U_x x \phi_1 \phi_1' \\ \n
= & \int  \eps \p_x (\bar{u} U + \bar{u}_y q) U_{xx} x \phi_1^2+ \int \eps \p_x (\bar{u} U + \bar{u}_y q) U_x \phi_1^2 + 2 \int \eps u_x U_x x \phi_1 \phi_1' \\ \n
= & \int \eps (\bar{u} U_x + \bar{u}_x U + \bar{u}_{xy}q - \bar{u}_y V) U_{xx} x \phi_1^2+ \int \eps (\bar{u} U_x + \bar{u}_x U + \bar{u}_{xy}q - \bar{u}_y V) U_x \phi_1^2 \\ \n
& + 2 \int \eps ( \bar{u} U_x + \bar{u}_x U + \bar{u}_{xy}q - \bar{u}_y V ) U_x x \phi_1 \phi_1' \\ \n 
= & \frac 1 2 \int \eps \bar{u} U_x^2 \phi_1^2 - \frac 3 2 \int \eps x \bar{u}_x U_x^2 \phi_1^2 + \int \eps (\bar{u}_{xx} + \frac{1}{2} \bar{u}_{xxx}x) U^2 \phi_1^2 + \int \eps \bar{u}_{xyy} q V_x \phi_1^2 \\ \n
& + \int \eps \bar{u}_{xy} U V_x x \phi_1^2 - \int \eps \bar{u}_x U V_y \phi_1^2 - \int \eps \bar{u}_{xy} q V_y \phi_1^2 + \frac{\eps}{2} \int \bar{u}_{xyy} x V^2 \phi_1^2 \\ \label{umbrella:1}
& + \int \eps \bar{u}_y V_y V_x x \phi_1^2 + E_{loc}^{(1)} =: \sum_{i = 1}^{9} D^{(5)}_i + E^{(1)}_{loc},
\end{align}
where 
\begin{align}
E_{loc}^{(1)} := - 2\int \eps \bar{u} x \phi_1 \phi_1' U_x^2 + \int \eps (\bar{u}_x + \bar{u}_{xx} x) U^2 \phi_1 \phi_1' - 2 \int \eps x \bar{u}_x UU_x \phi_1 \phi_1' + \int \eps \bar{u}_{yy} V^2 x \phi_1 \phi_1'.
\end{align}

First, it is evident that $|E_{loc}^{(1)}| \lesssim \|U, V \|_{X_0}^2$ as $|x| \lesssim 1$ on the support of $\phi_1'$. We now estimate each of the remaining terms in \eqref{umbrella:1}. $D^{(5)}_1$ and $D^{(5)}_2$ are controlled by the right-hand side of \eqref{Xh:right} upon invoking \eqref{bob:1} and upon using $\| \bar{u} + x \bar{u}_x \|_\infty \lesssim 1$. We estimate $D^{(5)}_3$ by noting that $|\bar{u}_{xx}| + |x \bar{u}_{xxx}| \lesssim \langle x \rangle^{-2}$, after which it is easily controlled by $\| U, V \|_{X_0}$. 

For the fourth term, we estimate via first localizing in $z$ using the cut-off function $\chi$, defined in \eqref{def:chi}, via 
\begin{align*}
D^{(5)}_4 =  \int \eps \bar{u}_{xyy} q V_x \phi_1^2 \chi(z) +  \int \eps \bar{u}_{xyy} q V_x \phi_1^2 (1- \chi(z)). 
\end{align*}
First for the far-field component, we have 
\begin{align} \n
| \int \eps \bar{u}_{xyy} q V_x x \phi_1^2 (1- \chi(z))| \lesssim & \sqrt{\eps} \| \bar{u}_{xyy} y x^{\frac 3 2} \|_\infty \| \sqrt{\eps} \bar{u} V_x x^{\frac 1 2} \phi_1\| \| \frac{q}{y} \langle x \rangle^{-1} \| \\
\lesssim & \sqrt{\eps} \| U, V \|_{X_{\frac 1 2}}  \| U \langle x \rangle^{-1} \|, 
\end{align}
which is an admissible contribution according to Hardy type inequality \eqref{Hardy:three}. For the localized component, we have 
\begin{align} \n
| \int \eps \bar{u}_{xyy} q V_x x \phi_1^2  \chi(z)| = & | \int \eps \bar{u}_{xyy} \frac{q}{y} \frac{y}{\sqrt{x}} V_x x^{\frac 32} \phi_1^2  \chi(z)| \lesssim \int \eps |\bar{u}_{xyy}| |\frac{q}{y}| \bar{u} V_x x^{\frac 32} \phi_1^2| \\
\lesssim & \sqrt{\eps}  \| U \langle x \rangle^{-1} \| \| \sqrt{\eps} \bar{u} V_x \langle x \rangle^{\frac 1 2} \phi_1 \| \lesssim \sqrt{\eps} \| U, V \|_{X_0}^2 + \sqrt{\eps} \| U, V \|_{X_{\frac 12}}^2, 
\end{align}
where we have used the pointwise decay estimate $|\bar{u}_{xyy} \langle x \rangle^2| \lesssim 1$, according to \eqref{prof:u:est}. The fifth term, $D^{(5)}_5$, follows by a nearly identical calculation. 

For the sixth term from \eqref{umbrella:1}, $D^{(5)}_6$, it is convenient to integrate by parts in $x$, which produces 
\begin{align} \n 
|\int \eps \bar{u}_{xy} U V_x x \phi_1^2| = &| \int \eps \bar{u}_{xxy} U V x \phi_1^2 + \int \frac{\eps}{2} \bar{u}_{xyy} V^2 x \phi_1^2 + \int \eps \bar{u}_{xy} UV \phi_1^2 - \int \eps \bar{u}_{xy} UV x 2 \phi_1 \phi_1' | \\ \n
\lesssim & \sqrt{\eps} (\| \bar{u}_{xxy} x^{2+2\sigma} \|_\infty + \| \bar{u}_{xy} x^{1+2\sigma} \|_\infty ) \| U \langle x \rangle^{- \frac 1 2 - \sigma} \| \| \sqrt{\eps} V \langle x \rangle^{-\frac 1 2 - \sigma} \| \\ \n
& + \| \bar{u}_{xyy} y^2 x \|_\infty \Big\| \sqrt{\eps} \frac{V}{y} \phi_1 \Big\|^2 \\ \n
\lesssim & \sqrt{\eps} ( \| U, V \|_{X_0} + \| U, V \|_{X_{\frac 1 2}} ) + C_\delta \| U, V \|_{X_0} + \delta \| U, V\|_{Y_{\frac 1 2}}^2. 
\end{align}

We estimate $D^{(5)}_7$ via 
\begin{align} 
|\int \eps \bar{u}_{xy} q V_y \phi_1^2| \lesssim & \sqrt{\eps} \| y x \bar{u}_{xy} \|_\infty \| \sqrt{\eps} U_x \phi_1 \| \| \frac{q}{y} \langle x \rangle^{-1} \|  \lesssim  \sqrt{\eps} \| \sqrt{\eps}U_x \phi_1\| \| U \langle x \rangle^{-1} \|, 
\end{align}
which is an admissible contribution according to \eqref{Hardy:three}. 

We estimate $D^{(5)}_8$ via 
\begin{align}
|\int \eps x \bar{u}_{xyy} V^2 \phi_1^2| \le \| x \bar{u}_{xyy} y^2 \|_\infty \Big\| \sqrt{\eps} \frac{V}{y} \phi_1 \Big\|^2 \lesssim \| \sqrt{\eps} V_y \phi_1\|^2,  
\end{align}
upon which we invoke \eqref{bob:1}. 

Finally, we estimate $D^{(5)}_9$ via  
\begin{align} \n
|\int \eps \bar{u}_y V_y V_x x \phi_1^2| \lesssim & \| \bar{u}_y x^{\frac 1 2} \|_\infty \| \sqrt{\eps} U_x \phi_1 \| \| \sqrt{\eps} V_x x^{\frac 1 2}\phi_1  \| \\ \n
\le & (C_{\delta_1} \| U, V \|_{X_0} + \delta_1 \| U, V \|_{Y_{\frac 1 2}} )( C_{\delta_2} \| U, V \|_{X_{\frac 1 2}} + \delta_2 \| U, V \|_{X_1} ) \\
\le & \delta \| U, V \|_{Y_{\frac 1 2}}^2 + \delta \| U, V \|_{X_1}^2 + C_\delta \| U, V \|_{X_0}^2 + \delta \| U, V \|_{X_{\frac 1 2}}^2. 
\end{align}

We move to the third diffusive term, by which we mean 
\begin{align}
- \eps \int v_{yy} (V_x x \phi_1^2 + V \phi_1^2 + 2 V x \phi_1 \phi_1') = - \int \eps v_{yy} V_x x \phi_1^2 + \int \eps v_y V_y \phi_1^2 + \int 2 \eps v_y V_y x \phi_1 \phi_1'.  
\end{align}
We easily estimate the final two terms above via 
\begin{align}
|\int \eps v_y V_y \phi_1^2 + \int 2 \eps v_y V_y x \phi_1 \phi_1' | \lesssim \|U, V \|_{X_0}^2. 
\end{align}
We thus deal with the principal contribution, which gives 
\begin{align} \n
- \eps \int v_{yy} V_x x \phi_1^2 = & \int \eps v_y V_{xy} x \phi_1^2 = \int \eps \p_y (\bar{u} V - \bar{u}_x q) V_{xy} x \phi_1^2\\ \n
= & \int \eps (\bar{u} V_y + \bar{u}_y V - \bar{u}_{xy}q - \bar{u}_x U) V_{xy} x \phi_1^2 \\  \n
= & - \int \frac{\eps}{2} \p_x (x \bar{u}) V_y^2 \phi_1^2 - \int \eps \bar{u}_y  V_y V_x x\phi_1^2 + \frac 1 2 \int \eps \p_x (x \bar{u}_{yy}) V^2 \phi_1^2\\ \n
& + \int \eps \bar{u}_{xyy} q V_x x \phi_1^2 + \int \eps \bar{u}_{xy} U V_x x \phi_1^2+ \int \eps \bar{u}_x U_y V_x x \phi_1^2 \\ \label{coffee:2} 
& - \int \eps x \bar{u} V_y^2 \phi_1 \phi_1' + \frac 1 2 \int \eps x \bar{u}_{yy} V^2 \phi_1 \phi_1'. 
\end{align}
These terms are largely identical to those in \eqref{umbrella:1}. The only slightly different term is the sixth term of \eqref{coffee:2}, which is estimated as 
\begin{align}
|\eps \int \bar{u}_x U_y V_x x| \lesssim \| \frac{\bar{u}_x}{\bar{u}} x \|_\infty \| \sqrt{\bar{u}} U_y \| \| \eps \sqrt{\bar{u}} V_x \| \lesssim \| U, V \|_{X_0}^2. 
\end{align}  

We now move to the fourth and final diffusive term, by which we mean 
\begin{align}
- \eps^2 \int v_{xx} V_x x \phi_1^2 - \int \eps^2 v_{xx} V \phi_1^2 -2  \int \eps^2 v_{xx} V x \phi_1 \phi_1',
\end{align}
An integration by parts in $x$ demonstrates that the final two terms above are estimated above by $\|U, V \|_{X_0}^2$. The first term above gives
\begin{align} \n
- \eps^2 \int v_{xx} V_x x \phi_1^2 = & \int \eps^2 v_x V_{xx} x \phi_1^2 + \int \eps^2 v_x V_x \phi_1^2 + 2\int \eps^2 v_x V_x x \phi_1 \phi_1'  \\  \n
= & \int \eps^2 \p_x (\bar{u} V - \bar{u}_x q) V_{xx} x \phi_1^2+ \int \eps^2 \p_x (\bar{u} V - \bar{u}_x q) V_x \phi_1^2 + 2 \int \eps^2 v_x V_x x \phi_1 \phi_1' \\ \label{econ}
= & \tilde{D}^{(6)}_1 + \tilde{D}^{(6)}_2 + \tilde{D}^{(6)}_3. 
\end{align}
The term $\tilde{D}^{(6)}_3$ is easily controlled by a factor of $\|U, V \|_{X_0}^2$. A few integrations by parts produces for the first term, $\tilde{D}^{(6)}_1$, above 
\begin{align} \n
\tilde{D}^{(6)}_1 = &\int \eps^2 \p_x (\bar{u} V - \bar{u}_x q) V_{xx} x \phi_1^2= \int \eps^2 (\bar{u} V_x + 2 \bar{u}_x V - \bar{u}_{xx}q) V_{xx} x  \phi_1^2 \\ \n
= & - \frac 1 2 \int \eps^2 \p_x(\bar{u} x) V_x^2  \phi_1^2 - \int \eps^2 \bar{u} x V_x^2 \phi_1 \phi_1' - 2 \eps^2 \int  (\bar{u}_x x V)_x V_x \phi_1^2  \\ \n
& + \int \eps^2 V_x \p_x (\bar{u}_{xx} qx)  \phi_1^2 - 4 \int \eps^2 \bar{u}_x x V V_x \phi_1 \phi_1' + 2 \int \eps^2 \bar{u}_{xx} q V_x x \phi_1 \phi_1'\\ \n
= & - \frac{\eps^2}{2} \int (\bar{u} x)_x V_x^2 \phi_1^2 - \int 2 \eps^2 x \bar{u}_x V_x^2 \phi_1^2 - 2 \eps^2 \int (x \bar{u}_x)x VV_x \phi_1^2  \\ \n
& + \eps^2 \int q V_x (x \bar{u}_{xx})_x \phi_1^2 - \eps^2 \int x \bar{u}_{xx} VV_x \phi_1^2 \\ \n
= & - \frac{\eps^2}{2} \int (x \bar{u})_x V_x^2 \phi_1^2- 2 \eps^2 \int x \bar{u}_x V_x^2 \phi_1^2 + \eps^2 \int (x \bar{u}_x)_{xx} V^2 \phi_1^2 \\ \n
& + \int \eps^2 V_x q \p_x (x \bar{u}_{xx}) \phi_1^2+ \eps^2 \int \p_x (x \bar{u}_{xx}) V^2 \phi_1^2  - 4 \int \eps^2 \bar{u}_x x V V_x \phi_1 \phi_1' \\ \label{coffee:3}
& + 2 \int \eps^2 \bar{u}_{xx} q V_x x \phi_1 \phi_1' =: \sum_{i = 1}^{7} D^{(6)}_i.
\end{align}

We now proceed to estimate all the terms above. The first term of \eqref{coffee:3}, $D^{(6)}_1$ ,we estimate via 
\begin{align}
|\frac{\eps^2}{2} \int (x \bar{u})_x V_x^2| \lesssim \| \p_x (x \bar{u}) \|_\infty \eps^2 \| V_x \|^2 \lesssim \| U, V \|_{X_0}^2. 
\end{align}
$D^{(6)}_2$ and $D^{(6)}_3$ are estimated in an analogous manner. For $D^{(6)}_4$ and $D^{(6)}_6$, we invoke the Hardy type inequality \eqref{Hardy:four} coupled with the estimate $\| \p_x^2(x \bar{u}_x)  x^2\|_\infty \lesssim 1$. We estimate $D^{(6)}_5$ via 
\begin{align} \n
|\int \eps^2 V_x q \p_x (x \bar{u}_{xx})| \lesssim & \eps \Big\| \frac{\p_x (x \bar{u}_{xx})}{\bar{u}} x y \Big\|_\infty \| \| \frac{q}{y} x^{-1} \| \| \eps \sqrt{\bar{u}} V_x \| \lesssim \eps \| U \langle x \rangle^{-1} \| \| \eps \sqrt{\bar{u}} V_x \| \\
\lesssim & \eps ( \| U, V \|_{X_0} + \| U, V \|_{X_{\frac 1 2}} ) \| U, V \|_{X_0}, 
\end{align}
where we have invoked the Hardy-type inequality \eqref{Hardy:three}. 

For the second term from \eqref{econ}, $\tilde{D}^{(6)}_2$, we expand and integrate by parts to generate
\begin{align} \n
|\int \eps^2 \p_x (\bar{u} V - \bar{u}_x q) V_x | =&| \int \eps^2 \bar{u} V_x^2 - \int 2 \eps^2 \bar{u}_{xx} V^2 + \frac 1 2 \int \eps^2 \bar{u}_{xxxx} q^2 | \\ \n
\lesssim & \| \sqrt{\bar{u}} \eps V_x \|^2 + \eps \| \bar{u}_{xx} x^2 \|_\infty \| \sqrt{\eps} V \langle x \rangle^{-1} \|^2 + \eps^2 \| \bar{u}_{xxxx} y^2 x^2 \| \frac{q}{y} \langle x \rangle^{-1} \|^2 \\
\lesssim & \| U, V \|_{X_0}^2 + \eps (\| U, V \|_{X_0}^2 + \| U, V \|_{X_{\frac 1 2}}^2), 
\end{align}
where we have invoked \eqref{Hardy:three} -  \eqref{Hardy:four}.

\vspace{2 mm}

\noindent \textit{Step 4: Error Terms} We now move to the remaining error terms, the first of which is the $\zeta U$ term from \eqref{sys:sim:1}. For this, we estimate via
\begin{align} \n
\Big| \int \zeta U U_x x \phi_1^2 \Big| \lesssim & \sqrt{\eps} \int \langle x \rangle^{- (1 + \frac{1}{50})} |U| |U_x| x \phi_1^2 \lesssim \sqrt{\eps} \| U \langle x \rangle^{- \frac 1 2 - \frac{1}{50}} \|  \| U_x x^{\frac 1 2} \phi_1 \| \\ \n
\lesssim & \sqrt{\eps} ( \| \bar{u} U \langle x \rangle^{- \frac 1 2 - \frac{1}{200}} \| + \| \sqrt{\bar{u}} U_y \| ) ( \| \bar{u} U_x x^{\frac 1 2} \phi_1 \| + \| \sqrt{\bar{u}} U_{xy} x \phi_1 \| ) \\
\lesssim & \sqrt{\eps} \| U, V \|_{X_0} (\| U, V\|_{X_{\frac 1 2}} + \| U, V \|_{X_1} ),
\end{align}
where we have invoked estimate \eqref{S:0} for pointwise decay of $\zeta$. The $\zeta_y q$ term from \eqref{sys:sim:1} and $\zeta V$ term from \eqref{sys:sim:2} is estimated in an identical manner. 

We now address the remaining error terms in equation \eqref{sys:sim:2}. The first of these is the term 
\begin{align} \n
\Big| \int \eps \alpha U (V_x x + V ) \phi_1^2 \Big| \lesssim \sqrt{\eps} \| U \langle x \rangle^{-\frac 3 4} \| ( \| \sqrt{\eps} \bar{u} V_x x^{\frac 1 2} \phi_1 \| + \| \sqrt{\eps} \bar{u} V \langle x \rangle^{- \frac 3 4} \| ) \lesssim \sqrt{\eps} \| U, V \|_{X_0} \| U, V \|_{X_{\frac 1 2}}, 
\end{align}
where we have invoked the pointwise decay estimate on $\alpha$ from \eqref{S:1}. The estimate on the $\alpha_y q$ term follows in an identical manner. This concludes the proof of Lemma \ref{Lem:3}. 
\end{proof}

\begin{lemma} \label{lemma:y:half} Let $(U, V)$ be a solution to \eqref{sys:sim:1} - \eqref{BC:UVYW}. Then for $0 < \delta << 1$,
\begin{align}  \label{basic:Yhalf:est:st}
\| U, V \|_{Y_{\frac 1 2}}^2 \lesssim C_\delta \| U, V \|_{X_0}^2 + C_{\delta} \|U, V \|_{E}^2+ \delta \| U, V \|_{X_{\frac 1 2}}^2 + \eps \| U, V \|_{X_1}^2 + \mathcal{T}_{Y_{\frac 1 2}} + \mathcal{F}_{Y_{\frac 1 2}},
\end{align}
where 
\begin{align} \label{TY12spec}
\mathcal{T}_{Y_{\frac 1 2}} := & \int ( \p_y \mathcal{N}_1 - \eps \p_x \mathcal{N}_2) U_y x \phi_1^2 \\ \label{FY12spec}
\mathcal{F}_{Y_{\frac 1 2}} := & \int (\p_y F_R - \eps \p_x G_R) U_y x \phi_1^2.   
\end{align}
\end{lemma}
\begin{proof} For this proof, it is convenient to work in the vorticity formulation, \eqref{eq:vort:pre}. We apply the multiplier $U_y x \phi_1(x)^2$ to \eqref{eq:vort:pre}. 

\vspace{2 mm}

\noindent \textit{Step 1: $\mathcal{T}_1$ Terms}  We first note that since $\mathcal{T}_1[U](x, 0) = 0$, we may integrate by parts in $y$ to view the product in the velocity form, and subsequently integrate by parts several times in $y$ and $x$ to produce 
\begin{align} \n
\int \p_y \mathcal{T}_1[U] U_y x \phi_1^2=  &- \int \mathcal{T}_1[U] U_{yy} x \phi_1^2 \\ \n
= & - \int \bar{u}^2 U_x U_{yy} x \phi_1^2 - \int \bar{u} \bar{v} U_y U_{yy} x \phi_1^2 - \int 2 \bar{u}^0_{pyy} U U_{yy} x \phi_1^2 \\ \n
= & \int 2 \bar{u} \bar{u}_y U_x U_y x \phi_1^2 - \frac 1 2 \int \bar{u}^2 U_y^2 \phi_1^2 - \int \bar{u} \bar{u}_x U_y^2 x \phi_1^2 + \frac 1 2 \int (\bar{u} \bar{v})_y U_y^2 x \phi_1^2 \\ \label{form:1}
& + \int 2 \bar{u}^0_{pyy} U_y^2 x - \int \p_y^4 \bar{u}^0_p U^2 x - \int \bar{u}^2 U_y^2 x \phi_1 \phi_1' =: \sum_{i = 1}^7 A^{(1)}_i. 
\end{align}
Note that above, we used the integration by parts identity 
\begin{align} \n
- \int 2 \bar{u}^0_{pyy}  U U_{yy} x \phi_1^2= & \int 2 \bar{u}^0_{pyy} U_y^2 x \phi_1^2 + \int 2 \bar{u}^0_{pyyy} U U_y x \phi_1^2 \\ \label{form:2}
= &  \int 2 \bar{u}^0_{pyy} U_y^2 x \phi_1^2 - \int \bar{u}^0_{pyyyy} U^2 x \phi_1^2, 
\end{align}
which is available due to the condition that $\bar{u}^0_{pyy}|_{y = 0} = 0$ and $\bar{u}^0_{pyyy}|_{y = 0} = 0$. 

We now estimate each of the terms in \eqref{form:1}, starting with $A^{(1)}_1$, which is controlled by
\begin{align} \n
|\int 2\bar{u} \bar{u}_y U_x U_y x \phi_1^2| \lesssim & \| \bar{u}_y x^{\frac 1 2} \|_\infty \| \bar{u} U_x x^{\frac 1 2} \phi_1\|  \| U_y \phi_1 \| \le C_{\delta_1} \| U_y \phi_1 \|^2 + \delta_1 \| U, V \|_{X_{\frac 1 2}}^2 \\ \n
\le & C_{\delta_1} C_{\delta_2} \| \sqrt{\bar{u}} U_y \|^2 + C_{\delta_1}  \delta_2 \| U, V \|_{Y_{\frac 1 2}}^2 + \delta_1 \| U, V \|_{X_{\frac 1 2}}^2 \\
\le & C_\delta \| U, V \|_{X_0}^2 + \delta \| U, V \|_{Y_{\frac 1 2}}^2 + \delta \| U, V \|_{X_{\frac 1 2}}^2, 
\end{align}
where we have invoked \eqref{bob:1}. 
 
For $A^{(1)}_2$, $A^{(1)}_3$, $A^{(1)}_4$, and $A^{(1)}_5$, we appeal to the coefficient estimate 
\begin{align} \label{coeff:1}
\|\frac{1}{2} \bar{u}^2\|_\infty + \| \bar{u} \bar{u}_x x \|_\infty + \frac 1 2 \| x \p_y (\bar{u} \bar{v}) \|_\infty + \| 2 \bar{u}_{yy} x \|_\infty \lesssim 1, 
\end{align} 
to control these terms by $C_\delta \| U, V \|_{X_0}^2 + \delta \| U, V \|_{Y_{\frac 1 2}}^2$. 
 
We estimate $A^{(1)}_6$ via 
\begin{align} \n
|\int \p_y^4 \bar{u}^0_p U^2 x \phi_1^2| \lesssim & |\int \p_y^4 \bar{u}^0_p \mathring{U}^2 x \phi_1^2| + |\int \p_y^4 \bar{u}^0_p U(x, 0)^2 x \phi_1^2| \\ \n
\lesssim & \| \p_y^4 \bar{u}^0_p x y^2 \|_\infty \| U_y \phi_1 \|^2 + \| \p_y^4 \bar{u}^0_p x^{\frac 12} \|_{L^\infty_x L^1_y} \| U(x, 0) x^{\frac 1 4} \|_{x = 0}^2 \\
\le & C_\delta \| U, V \|_{X_0}^2 + \delta \| U, V \|_{Y_{\frac 1 2}}^2.
\end{align}
 The final term in \eqref{form:1}, $A^{(1)}_7$, is localized in $x$, and is clearly bounded above a factor of $\| U, V \|_{X_0}^2$. 
 
\vspace{2 mm}

\noindent \textit{Step 2: $\mathcal{T}_2$ Terms:} We now estimate the contributions from $\mathcal{T}_2$ via first integrating by parts in $x$ to produce 
\begin{align}
- \eps \int \p_x \mathcal{T}_2[V]  U_y x \phi_1^2 = \int \eps \mathcal{T}_2[V] U_{xy} x \phi_1^2 + 2\int \eps \mathcal{T}_2[V] U_y \phi_1 \phi_1'.
\end{align} 
We now appeal to the definition of $\mathcal{T}_2[V]$ in \eqref{def:T2} to produce 
\begin{align} \n
\int \eps \mathcal{T}_2[V] U_{xy} x \phi_1^2 = - & \int \eps ( \bar{u}^2 V_x + \bar{u} \bar{v} V_y + \bar{u}_{yy}V )V_{yy} x \phi_1^2  \\ \n
= & \int 2 \eps \bar{u} \bar{u}_y V_x V_y x \phi_1^2 + \int \eps \bar{u}^2 V_{xy} V_y x \phi_1^2 + \frac 1 2 \int \eps \p_y (\bar{u} \bar{v}) V_y^2 x \phi_1^2 \\ \n
& + \int \eps \bar{u}_{yy} V_y^2 x \phi_1^2 - \frac 1 2 \int \eps \bar{u}_{yyyy} V^2 x \phi_1^2 \\ \n
= & \int 2 \eps \bar{u} \bar{u}_y V_x V_y x \phi_1^2 - \int \frac 1 2 \eps \bar{u}^2 V_y^2 \phi_1^2 - \int \eps \bar{u} \bar{u}_x V_y^2 x \phi_1^2 \\ \n
& + \frac 1 2 \int \eps (\bar{u} \bar{v})_y V_y^2 x \phi_1^2 + \int \eps \bar{u}_{yy} V_y^2 x \phi_1^2 - \int \frac{\eps}{2} \bar{u}_{yyyy} V^2 x \phi_1^2 - \int \eps \bar{u}^2 V_y^2 x \phi_1 \phi_1' \\ \label{wolf:2}
= & A^{(2)}_1 + ... + A^{(2)}_7.
\end{align}

For the first term, $A^{(2)}_1$, we estimate via 
\begin{align} \n 
| \int \eps \bar{u} \bar{u}_y V_x V_y x \phi_1^2 | \lesssim &\| \bar{u}_y x^{\frac 1 2} \|_\infty \| \sqrt{\eps} V_y \phi_1 \| \| \sqrt{\eps} \bar{u} V_x x^{\frac 1 2} \phi_1 \| \\
\le & C_\delta \| U, V \|_{X_0}^2 + \delta \| U, V \|_{Y_{\frac 1 2}}^2 + \delta \| U, V \|_{X_{\frac 1 2}}^2.
\end{align}
For $A^{(2)}_2, A^{(2)}_3, A^{(2)}_4, A^{(2)}_5$, we estimate using the same coefficient estimate as \eqref{coeff:1}. We estimate $A^{(2)}_6$ via 
\begin{align} \n
|\int \eps \bar{u}_{yyyy} V^2 x| \lesssim \| \bar{u}_{yyyy} xy^2 \|_\infty \| \sqrt{\eps} V_y \|^2 \le C_\delta \| U, V \|_{X_0}^2 + \delta \| U, V \|_{Y_{\frac 1 2}}^2. 
\end{align}
The final term, $A^{(2)}_7$, can be controlled by $\|U, V \|_{X_0}^2$ upon invoking the bounded support of $\phi_1'$. 

\vspace{2 mm}

\noindent \textit{Step 3: Diffusive Terms:}  We now compute (in the vorticity form) via a long series of integrations by parts the following identity 
\begin{align}  \n
- \int \p_y^3 u U_y x \phi_1^2= & \int u_{yy} U_{yy} x \phi_1^2+ \int_{y = 0} u_{yy} U_y x \phi_1^2 \ud  x \\ \n
= & \int \p_y^2 (\bar{u} U + \bar{u}_y q) U_{yy} x\phi_1^2 + \int_{y = 0} \p_y^2 (\bar{u} U + \bar{u}_y q) U_{y} x \phi_1^2 \ud x \\ \n
= & \int \bar{u} U_{yy}^2 x \phi_1^2 - \frac 9 2 \int \bar{u}_{yy} U_y^2 x \phi_1^2+ \int 3 \bar{u}_{yyyy} U^2 x \phi_1^2- \int \frac 1 2 \p_y^6 \bar{u} q^2 x \phi_1^2 \\ \label{likeyou}
& + \frac 3 2 \int_{y = 0} \bar{u}_y U_y^2 x \phi_1^2 \ud x + \frac 3 2 \int_{y = 0} \bar{u}_{yy} UU_y x \phi_1^2 = \sum_{i = 1}^6 B^{(1)}_i.
\end{align} 
We first notice that $B^{(1)}_1$ and $B^{(1)}_5$ are positive contributions. $B^{(1)}_2$ is easily estimated by $\| U, V \|_{X_0}^2$ upon using $\|\frac{\bar{u}_{yy}}{\bar{u}}x \|_\infty \lesssim 1$. We estimate $B^{(1)}_3$ by 
\begin{align} \n
|\int 3 \bar{u}_{yyyy} U^2 x \phi_1^2| \lesssim & |\int \bar{u}_{yyyy} (U - U(x, 0))^2 x \phi_1| + |\int \bar{u}_{yyyy} U(x, 0)^2 x| \\ \n
\lesssim & \| \bar{u}_{yyyy} x y^2 \|_\infty \Big\| \frac{U - U(x, 0)}{y} \phi_1 \Big\|^2 + \| \bar{u}_{yyyy} x^{\frac 1 2} \|_{L^\infty_x L^1_y} \|  U(x, 0) x^{\frac 1 4} \|_{y = 0}^2 \\ \label{rockL}
\lesssim & \| U_y \|^2 + \| U(x, 0) x^{\frac 1 4} \|_{y = 0}^2 \le C_\delta \|U, V \|_{X_0}^2 + \delta \|U, V \|_{Y_{\frac 1 2}}^2. 
\end{align}
The term $B^{(1)}_4$ is estimated in an entirely analogous manner. The term $B^{(1)}_6$ is estimated by 
\begin{align}
|\int_{y = 0} \bar{u}_{yy} UU_y x \phi_1^2| \lesssim \eps^{\frac 1 2} \| U \langle x \rangle^{- \frac 1 4} \|_{y = 0} \| U_y \phi_1 \langle x \rangle^{\frac 1 4}  \|_{y = 0} \lesssim \eps^{\frac 1 2} \| U, V\|_{X_0} \|U, V \|_{Y_{\frac 1 2}},
\end{align}
upon invoking the bound $|\bar{u}_{yy}(x, 0)| \lesssim \eps^{\frac 1 2} \langle x \rangle^{-1}$ due to  \eqref{prof:u:est} coupled with the fact that $\bar{u}^0_{pyy}(x, 0) = 0$. 

In addition to this term, we need to estimate the term $\bar{u}^0_{p yyy}q$ from \eqref{sys:sim:1}, we do so via 
\begin{align}  \n
\int \p_y (\bar{u}^0_{pyyy} q) U_y x =& \int \bar{u}^0_{p yyyy} q U_y x + \int \bar{u}^0_{p yyy} UU_y x =   - \int \p_y^5 \bar{u}^0_p q U x - \frac 1 2 \int \p_y^4 \bar{u}^0_p U^2 x \\
= & \frac 1 2 \int \p_y^6 \bar{u}^0_p q^2 x - \frac 1 2 \int \p_y^4 \bar{u}^0_p U^2 x, 
\end{align}
which we estimate in an identical manner to \eqref{rockL}.

The next diffusive term is 
\begin{align}  \label{grey}
- 2 \int \eps u_{xxy} U_y x \phi_1^2= & \int 2 \eps u_{xy} U_{xy} x \phi_1^2 + \int 2 \eps u_{xy} U_y \phi_1^2 + \int 4 \eps u_{xy} U_y x \phi_1 \phi_1' \\ \n
= & \int 2 \eps \p_{xy} (\bar{u} U + \bar{u}_y q) U_{xy} x \phi_1^2 + \int 2 \eps \p_x (\bar{u} U + \bar{u}_y q) U_y \phi_1^2\\ \n
&+ \int 4 \eps u_{xy} U_y x \phi_1 \phi_1' =: \sum_{i = 1}^{5} \tilde{B}^{(2)}_i.  
\end{align}
Due to the localization in $x$ of $\phi_1'$ that the term $\tilde{B}^{(2)}_5$ above is estimated by  
\begin{align} \label{use:E:norm}
|\int \eps u_{xy} U_y x \phi_1 \phi_1'| \lesssim \| \eps u_{xy} \phi_1' \| \| U_y \phi_1 \| \lesssim \|U, V \|_E (C_\delta \|U, V \|_{X_0} + \delta \|U, V \|_{Y_{\frac 1 2}}).
\end{align}
 
Due to the length of the forthcoming expressions, we handle each of the remaining four terms in \eqref{grey}, $\tilde{B}^{(2)}_k$, $k = 1,2,3,4$, individually.  First, integration by parts yields for $\tilde{B}^{(2)}_1$,
\begin{align} \n
\int 2 \eps \p_{xy}(\bar{u} U) U_{xy} x \phi_1^2 = & \int 2 \eps \bar{u} U_{xy}^2 x\phi_1^2 - \int 4 \eps \bar{u}_{xxy} UU_y x \phi_1^2 - \int 4 \eps \bar{u}_{xy} U_x U_y x \phi_1^2- \int 4 \eps \bar{u}_{xy} U U_y \phi_1^2\\ \n
&- \int \eps \p_x (x \bar{u}_x) U_y^2 \phi_1^2 - \int 4 \eps \bar{u}_{yy} U_x^2 x \phi_1^2- \int_{y = 0} 2 \eps \bar{u}_y U_x(x, 0)^2 x \phi_1^2 \ud x  \\ \label{park1}
& + \int \eps \bar{u}_{yyyy} V^2 x \phi_1^2+ \int 2 \eps \bar{u}_{xyy} q U_{xy}x \phi_1^2 + E_{loc}^{(2)} =: \sum_{i = 1}^{9} B^{(2)}_i + E_{loc}^{(2)},  
\end{align}
where $E_{loc}^{(2)}$ are localized contributions that can be controlled by a large factor of $\| U, V \|_{X_0}^2 + \|U, V \|_E^2$. 

The first term, $B^{(2)}_1$, is a positive contribution. The terms $B^{(2)}_2$ and $B^{(2)}_4$ are estimated via 
\begin{align} \n
| \int 4 \eps \bar{u}_{xxy} UU_y x \phi_1^2| + |\int 4 \eps \bar{u}_{xy} UU_y \phi_1^2| \lesssim &  \eps \Big( \| \bar{u}_{xxy} x^2 \|_\infty + \| \bar{u}_{xy} x \|_\infty  \Big) \| U \langle x \rangle^{-1}  \phi_1 \| \| U_y \phi_1  \|  \\
\lesssim & \eps (\| U, V \|_{X_0}^2 + \| U, V \|_{Y_{\frac 1 2}}^2 + \| U, V \|_{X_{\frac 1 2}}^2), 
\end{align}
where we have appealed to \eqref{bob:1} and \eqref{Hardy:three}. 

The terms $B^{(2)}_3$, $B^{(2)}_5$, $B^{(2)}_6$, and $B^{(2)}_8$ are estimated via 
\begin{align} \n
&|\int 4 \eps \bar{u}_{xy} U_x U_y x \phi_1^2| \lesssim  \sqrt{\eps} \| \bar{u}_{xy} x \|_\infty \| \sqrt{\eps} U_x \phi_1 \| \| U_y \phi_1 \| \lesssim \sqrt{\eps} \Big( \| U, V \|_{X_0}^2 + \| U, V \|_{Y_{\frac 1 2}}^2 \Big), \\ \n
&|\int \eps \p_x (x \bar{u}_x) U_y^2 \phi_1^2| \lesssim \| \p_x (x \bar{u}_x) \|_\infty \eps \| U_y \phi_1 \|^2 \lesssim \eps \Big( \| U, V \|_{X_0}^2 + \| U, V \|_{Y_{\frac 1 2}}^2 \Big)  \\ \n
&|\int 4 \eps \bar{u}_{yy} U_x^2 x \phi_1^2 | \lesssim \| \bar{u}_{yy} x \|_\infty \| \sqrt{\eps} U_x \phi_1 \|^2 \le C_\delta \| U, V \|_{X_0}^2 + \delta \| U, V \|_{Y_{\frac 1 2}}^2, \\ \n
&|\int \eps \bar{u}_{yyyy} V^2 x \phi_1^2| \lesssim \eps \| \bar{u}_{yyyy} x y^2 \|_\infty \Big\| \frac{V}{y} \phi_1 \Big\|^2 \le C_\delta \| U, V \|_{X_0}^2 + \delta \| U, V \|_{Y_{\frac 1 2}}^2, 
\end{align}
and $B^{(2)}_9$ is estimated via 
\begin{align} \n
| \int\eps \bar{u}_{xyy} q U_{xy} x \phi_1^2| \lesssim &\sqrt{\eps} \| \bar{u}_{xyy} x^{\frac 3 2} y \|_\infty \Big\| \frac{q}{y} \langle x \rangle^{-1} \phi_1 \Big\| \| \sqrt{\eps} \sqrt{\bar{u}} U_{xy} x^{\frac 1 2} \phi_1 \| \\ \n
\lesssim & \sqrt{\eps} \| U \langle x \rangle^{-1} \| \| \sqrt{\eps} \sqrt{\bar{u}} U_{xy} x^{\frac 1 2} \phi_1 \| \\ \label{whyb}
 \lesssim & \sqrt{\eps} ( \| U, V \|_{X_0} + \| U, V \|_{X_{\frac 1 2}} ) \| U, V \|_{Y_{\frac 1 2}}.
\end{align}
The term $B^{(2)}_7$ requires us to use the $X_1$ norm, albeit with a pre-factor of $\eps$ and with a weaker weight in $x$: 
\begin{align}
|\int_{y = 0} \eps \bar{u}_y U_x(x, 0)^2 x \phi_1^2 \ud x| \lesssim \eps \| \sqrt{\bar{u}_y} U_x(x, 0) x \phi_1 \|_{x = 0}^2 \lesssim \eps \| U, V \|_{X_1}^2,
\end{align}
where we use that the choice of $\phi_1$ is the same as that of $\| \cdot \|_{X_1}$. This concludes treatment of $\tilde{B}^{(2)}_1$. 

The second term from \eqref{grey}, $\tilde{B}^{(2)}_2$, gives 
\begin{align} \n
\int 2 \eps \p_{xy}( \bar{u}_y q) U_{xy} x \phi_1^2 = &\int 2\eps \Big( \bar{u}_{xyy} q + \bar{u}_{xy} U - \bar{u}_{yy} V + \bar{u}_y U_x \Big) U_{xy} x \phi_1^2 \\ \n
= & \int 2 \eps \bar{u}_{xyy} q U_{xy} x \phi_1^2 - \int 2 \eps \bar{u}_{xy} U_x U_y x \phi_1^2 - \int 2 \eps \bar{u}_{xy} UU_y \phi_1^2 \\ \n
& - \int 2 \eps \bar{u}_{xxy} x UU_y \phi_1^2 - \int 2 \eps \bar{u}_{yy} V_y^2 x \phi_1^2 + \int \eps \bar{u}_{yyyy} V^2 x \phi_1^2\\ \label{jup}
&  - \int_{y = 0} \eps \bar{u}_y U_x(x, 0)^2 x \phi_1^2 \ud x - 4 \int \eps \bar{u}_{xy} U U_y x \phi_1 \phi_1' = : \sum_{i = 1}^8 J_i. 
\end{align}
The final term above, $J_8$, is a localized in $x$ contribution, can easily be controlled by $\| U, V \|_{X_0}^2$.  $J_1$ is treated in the same manner as \eqref{whyb}. $J_2, J_3$, and $J_4$ are estimated by 
\begin{align} \n
&|\int \eps \bar{u}_{xy} U_x U_y x \phi_1^2| \lesssim \sqrt{\eps} \| \bar{u}_{xy} x \|_\infty \| \sqrt{\eps} U_x \phi_1 \| \| U_y \phi_1 \| \lesssim \sqrt{\eps} (\| U, V \|_{X_0}^2 + \| U, V \|_{Y_{\frac 1 2}}^2), \\ \n
&|\int \eps (x \bar{u}_{xy})_x UU_y \phi_1| \lesssim \eps \| \bar{u}_{xy} x \|_\infty \| U \langle x \rangle^{-1} \phi_1 \| \| U_y \phi_1 \| \lesssim \eps (\| U, V \|_{X_0}^2 + \| U, V \|_{X_{\frac 1 2}}^2+ \| U, V \|_{Y_{\frac 1 2}}^2).
\end{align}
Terms $J_5$, $J_6$, $J_7$ are identical to $B^{(2)}_6$, $B^{(2)}_8$, and $B^{(2)}_7$. This concludes treatment of $\tilde{B}^{(2)}_2$.

The third and fourth terms from \eqref{grey}, $\tilde{B}^{(2)}_3$ and $\tilde{B}^{(2)}_4$ together give
\begin{align}  \n
\int 2\eps \p_x (\bar{u} U + \bar{u}_y q) U_y \phi_1^2 = &\int 2 \eps \bar{u} U_x U_y \phi_1^2 + \int 2 \eps \bar{u}_x U U_y \phi_1^2+ \int 2 \eps \bar{u}_{xy} q U_y - \int 2 \eps \bar{u}_y V U_y \phi_1^2\\ \n
= & \int 2 \eps \bar{u} U_x U_y \phi_1^2- \int \eps \bar{u}_{xy} U^2 \phi_1^2+ \int 2 \eps \bar{u}_{xy} q U_y \phi_1^2 + \int 2 \eps \bar{u}_y V_y U \phi_1^2\\ \n
&  + \int 2 \eps \bar{u}_{yy} UV\phi_1^2 \\ \n
= & \int 2 \eps \bar{u} U_x U_y \phi_1^2 - \int \eps \bar{u}_{xy} U^2 \phi_1^2+ \int 2 \eps \bar{u}_{xy} q U_y  \phi_1^2+ \int  \eps \bar{u}_{xy} U^2 \phi_1^2\\ \n
&  + \int 2 \eps \bar{u}_{yy} UV \phi_1^2 + \int 2 \eps \bar{u}_y U^2 \phi_1 \phi_1' \\ \label{hein}
= &  \int 2 \eps \bar{u} U_x U_y \phi_1^2+ \int 2 \eps \bar{u}_{xy} q U_y  \phi_1^2+ \int 2 \eps \bar{u}_{yy} UV \phi_1^2 + \int 2 \eps \bar{u}_y U^2 \phi_1 \phi_1'. 
\end{align}

The first and final terms from \eqref{hein} can easily be estimated by $\sqrt{\eps} \| U, V \|_{X_0}^2$, while the second term \eqref{hein}
\begin{align} \n
|\int 2 \eps \bar{u}_{xy} q U_y \phi_1^2| \lesssim \eps \| \bar{u}_{xy} y x \|_\infty \| \frac{q}{y} \langle x \rangle^{-1} \phi_1 \|  \| U_y \phi_1 \| \lesssim \eps \| U \langle x \rangle^{-1} \| \| U_y \phi_1 \| 
\end{align}
and the third term from \eqref{hein} can be estimated via 
\begin{align} 
|\int 2 \eps \bar{u}_{yy} UV \phi_1^2| \lesssim & \eps | \int \bar{u}_{yy} \mathring{U} V \phi_1^2| + \eps |\int \bar{u}_{yy} U(x, 0) V \phi_1^2| \\ \n
\lesssim & \sqrt{\eps} \| \bar{u}_{yy} y^2 \|_\infty \| U_y \phi_1 \| \| \sqrt{\eps} V_y \phi_1 \| + \sqrt{\eps} \| \bar{u}_{yy} y x^{- \frac 1 4} \|_{L^\infty_x L^1_y} \| U(x, 0) x^{\frac 1 4}  \|_{L^2(x = 0)} \| \sqrt{\eps} V_y \phi_1 \|. 
\end{align}

We now arrive at the final diffusive term, which we integrate by parts in $x$ via 
\begin{align} \n
\int \eps^2 v_{xxx} U_y x \phi_1^2 = & - \eps^2 \int v_{xx} U_{xy} x \phi_1^2- \int \eps^2 v_{xx} U_y \phi_1^2 - 2 \eps^2 \int v_{xx} U_y x \phi_1 \phi_1' \\  \n
= & \int \eps^2 v_{xx} V_{yy} x \phi_1^2 - \int \eps^2 v_x V_{yy} \phi_1^2 + 2 \int \eps^2 v_x U_y \phi_1 \phi_1' - 2 \eps^2 \int v_{xx} U_y x \phi_1 \phi_1' \\ \label{misfit}
= & \tilde{P}_1 + \tilde{P}_2 + \tilde{P}_3 + \tilde{P}_4. 
\end{align}
Again, the terms with $\phi_1'$ above, $\tilde{P}_3, \tilde{P}_4$, are easily estimated above by a factor of $\| U, V \|_{X_0}^2 + \|U, V \|_E^2$ due to the localization in $x$, in an analogous manner to \eqref{use:E:norm}. 

For the second term on the right-hand side of \eqref{misfit}, $\tilde{P}_2$, we produce the following identity 
\begin{align} \n
- \int \eps^2 v_x V_{yy}  \phi_1^2= & - \int \eps^2 \p_x (\bar{u} V - \bar{u}_x q) V_{yy} \phi_1^2 \\ \n
= & - \int \eps^2 (\bar{u} V_x + 2 \bar{u}_x V - \bar{u}_{xx}q) V_{yy} \phi_1^2 \\ \n
= & \int \eps^2 \bar{u}_y V_x V_y \phi_1^2 + \int \eps^2 \bar{u} V_{xy} V_y \phi_1^2 + \int 2\eps^2 \bar{u}_x V_y^2 \phi_1^2 + \int 2 \eps^2 \bar{u}_{xy} V V_y \phi_1^2 \\ \n
& - \int \eps^2 \bar{u}_{xx} U V_y \phi_1^2 + \int \eps^2 \bar{u}_{xxy} UV \phi_1^2+ \int \frac{\eps^2}{2} \bar{u}_{xxxyy} q^2 \phi_1^2+ \int \eps^2 \bar{u}_{xxyy} q^2 \phi_1 \phi_1' \\ \n
= & \int \eps^2 \bar{u}_y V_x V_y \phi_1^2- \frac 3 2 \int \eps^2 \bar{u}_x V_y^2\phi_1^2 - \int \eps^2 \bar{u}_{xyy} V^2 \phi_1^2- \frac{\eps^2}{2} \int \bar{u}_{xxx} U^2 \phi_1^2\\ \n
& + \int \eps^2 \bar{u}_{xxy} UV \phi_1^2 + \frac{\eps^2}{2} \int \bar{u}_{xxxyy} q^2 \phi_1^2- \int \eps^2 \bar{u} V_y^2 \phi_1 \phi_1' + \int \eps^2 \bar{u}_{xxyy} q^2 \phi_1 \phi_1' \\
& - \int \eps^2 \bar{u}_{xx} U^2 \phi_1 \phi_1' = \sum_{i = 1}^{9} P^{(1)}_i. 
\end{align}
Again, the terms with a $\phi_1'$, $P^{(1)}_7, P^{(1)}_8, P^{(1)}_9$, are easily controlled by a factor of $\| U, V \|_{X_0}^2$. 

We now proceed to estimate each of the remaining terms above via 
\begin{align*}
&|\int \eps^2 \bar{u}_y V_x V_y \phi_1^2| \lesssim \sqrt{\eps} \| \eps V_x \phi_1\| \| \sqrt{\eps} V_y \phi_1 \| \lesssim \sqrt{\eps} (\| U, V \|_{X_0}^2 + \| U, V \|_{Y_{\frac 1 2}}^2), \\
&|\int \eps^2 \bar{u}_x V_y^2 \phi_1^2| \lesssim \eps \| \sqrt{\eps} V_y \phi_1 \|^2 \lesssim \eps (\| U, V \|_{X_0}^2 + \| U, V \|_{Y_{\frac 1 2}}^2), \\
&|\int \eps^2 \bar{u}_{xyy} V^2 \phi_1^2| \lesssim \eps \| \bar{u}_{xyy} y^2 \|_\infty \Big\| \sqrt{\eps} \frac{V}{y} \phi_1 \Big\|^2 \lesssim \eps \| \sqrt{\eps} V_y \phi_1\|^2 \lesssim \eps (\| U, V \|_{X_0}^2 + \| U, V \|_{Y_{\frac 1 2}}^2), \\
&|\int \eps^2 \bar{u}_{xxx} U^2 \phi_1^2| \lesssim \eps^2 \| \bar{u}_{xxx} x^2 \|_\infty \| U \langle x \rangle^{-1} \|^2 \lesssim \eps^2 (\| U, V \|_{X_0}^2 + \| U, V \|_{X_{\frac 1 2}}^2), \\
&|\int \eps^2 \bar{u}_{xxy} UV \phi_1^2| \lesssim \eps^{\frac 32} \| U \langle x \rangle^{-1} \| \| \sqrt{\eps} V_y \phi_1\| \lesssim \eps^{\frac 3 2} (\| U, V \|_{X_0}^2 + \| U, V \|_{X_{\frac 1 2}}^2 + \| U, V \|_{Y_{\frac 1 2}}^2), \\
&| \frac{\eps^2}{2} \int \bar{u}_{xxxyy} q^2 \phi_1^2| \lesssim \eps^2 \| \bar{u}_{xxxyy} x^2 y^2 \|_\infty \| U \langle x \rangle^{-1} \|^2 \lesssim \eps^2 (\| U, V \|_{X_0}^2 + \| U, V \|_{X_{\frac 1 2}}^2).
\end{align*}
This concludes the treatment of $\tilde{P}_2$.

We now treat $\tilde{P}_1$. We further integrate by parts using that $v = \bar{u} V - \bar{u}_x q$, which produces the following identity 
\begin{align} \label{mis2}
\tilde{P}_1 = \int \eps^2 v_{xx} V_{yy} x \phi_1^2= & \int \eps^2 \p_{xx} (\bar{u} V - \bar{u}_x q) V_{yy} x \phi_1^2\\  \n
= & \int \eps^2 (\bar{u} V_{xx} + 3 \bar{u}_x V_x + 3 \bar{u}_{xx} V - \bar{u}_{xxx} q) V_{yy} x \phi_1^2 =: \tilde{P}_{1,1} + \tilde{P}_{1,2} + \tilde{P}_{1,3} + \tilde{P}_{1,4}.
\end{align}

For $\tilde{P}_{1,1}$, we integrate by parts several times in $x$ and $y$ to produce
\begin{align} \n
\int \eps^2 \bar{u} V_{xx} V_{yy} x \phi_1^2 = & - \int \eps^2 \bar{u} V_{xxy} V_y x \phi_1^2 - \int \eps^2 \bar{u}_y V_{xx} V_y x \phi_1^2 \\ \n
= & \int \eps^2 \bar{u} V_{xy}^2 x \phi_1^2+ \int \eps^2 x \bar{u}_x V_{xy} V_y \phi_1^2+ \int \eps^2 \bar{u} V_{xy} V_y \phi_1^2\\ \n
& + \int \eps^2 \bar{u}_y V_x V_y \phi_1^2+ \int \eps^2 \bar{u}_{xy} V_x V_y x \phi_1^2+ \int \eps^2 \bar{u}_y V_x V_{xy}x \phi_1^2\\ \n
& +  2\int \eps^2 \bar{u} V_{xy} V_y x \phi_1 \phi_1' + 2 \int \eps^2 \bar{u}_y V_x V_y x \phi_1 \phi_1'  \\ \n
= & \int \eps^2 \bar{u} V_{xy}^2 x \phi_1^2 - \eps^2 \frac 1 2 \int \p_x (x \bar{u}_x) V_y^2 \phi_1^2 - \frac 1 2 \int \eps^2 \bar{u}_x V_y^2 \phi_1^2 \\ \n
& + \int \eps^2 \bar{u}_y V_x V_y \phi_1^2+ \int \eps^2 \bar{u}_{xy} V_x V_y x\phi_1^2 - \int \frac{\eps^2}{2} \bar{u}_{yy} V_x^2 x \phi_1^2 \\ \n
& - \int \eps^2 \bar{u} V_y^2 \phi_1 \phi_1' - \int \eps^2 x \bar{u}_x V_y^2 \phi_1 \phi_1' +   2\int \eps^2 \bar{u} V_{xy} V_y x \phi_1 \phi_1' \\ \label{clay}
& + 2 \int \eps^2 \bar{u}_y V_x V_y x \phi_1 \phi_1' =: \sum_{i = 1}^{10} H^{(1)}_i. 
\end{align} 
All of the terms with $\phi_1'$ can again be controlled by a factor of $\| U, V \|_{X_0}^2 + \|U, V \|_E^2$. The first term, $H^{(1)}_1$, is a positive contribution. $H^{(1)}_2$ and $H^{(1)}_3$ are estimated by 
\begin{align*}
|\int \frac 1 2 \p_x (x \bar{u}_x) V_y^2 \phi_1^2| + |\int \frac 1 2 \eps^2 \bar{u}_x V_y^2 \phi_1^2| \lesssim &(\| \p_x (x \bar{u}_x) \|_\infty + \| \bar{u}_x \|_\infty ) \eps \| \sqrt{\eps} V_y \phi_1 \|^2 \\
\lesssim & \eps (\| U, V \|_{X_0}^2 + \| U, V \|_{Y_{\frac 1 2}}^2),
\end{align*}
while the $H^{(1)}_4, H^{(1)}_5$ and $H^{(1)}_6$ are estimated via 
\begin{align*} 
&| \int \eps^2 \p_x (x \bar{u}_y) V_x V_y\phi_1^2 | \lesssim \sqrt{\eps} \| \p_x (x \bar{u}_y) \|_\infty \| \eps V_x \| \| \sqrt{\eps} V_y \| \lesssim \sqrt{\eps} (\| U, V \|_{X_0}^2 + \| U, V \|_{Y_{\frac 1 2}}^2), \\
&|\int \eps^2 \bar{u}_{yy} V_x^2 x\phi_1^2| \lesssim \| \bar{u}_{yy} x \|_\infty \| \eps V_x \|^2 \le \delta \| U, V \|_{Y_{\frac 1 2}}^2 + \| U, V \|_{X_0}^2.  
\end{align*}
This concludes the treatment of $\tilde{P}_{1,1}$. 

The terms $\tilde{P}_{1,k}$, $k = 2, 3, 4$, are equivalent to  
\begin{align} \n
&\eps^2 \int (3  \bar{u}_x V_x + 3 \bar{u}_{xx} V - \bar{u}_{xxx} q) V_{yy} x \phi_1^2\\ \n
= & - \int 3 \eps^2 \bar{u}_{xy} V_x V_y x \phi_1^2+ \frac 3 2 \int \eps^2 \p_x (x\bar{u}_{x}) V_y^2 x \phi_1^2 - \int 3 \eps^2 \bar{u}_{xx} V_y^2 x \phi_1^2 \\ \label{clay2}
& + \frac 3 2 \int \eps^2 \bar{u}_{xxyy} V^2 x \phi_1^2 + \int \eps^2 \bar{u}_{xxxy} q V_y x + \int \eps^2 \bar{u}_{xxx} U V_y x \phi_1^2. 
\end{align}

We estimate each of these contributions in a nearly identical fashion to the terms from \eqref{clay}, and so omit repeating these details. 

\vspace{2 mm}

\noindent \textit{Step 4: Error Terms} We now estimate the error terms on the right-hand side of \eqref{eq:vort:pre}, starting with
\begin{align} \n
\int \p_y(\zeta U) U_y x \phi_1^2 = & \int \zeta U_y^2 x \phi_1^2 + \int \p_y \zeta U U_y x \phi_1^2 \\ \label{riwiu}
= & \int \zeta U_y^2 x \phi_1^2 - \frac 1 2 \int \p_y^2 \zeta U^2 x - \frac 1 2 \int_{y = 0} \p_y \zeta U^2 x \phi_1^2.   
\end{align}
The first term above is estimated via 
\begin{align} \n
|\int \zeta U_y^2 x \phi_1^2| \lesssim \sqrt{\eps} \| U_y \phi_1 \|^2 \lesssim \sqrt{\eps} (\| U, V \|_{X_0}^2 + \| U, V \|_{Y_{\frac 1 2}}^2),  
\end{align}
where we have appealed to the estimate \eqref{S:0} as well as the Hardy type inequality \eqref{bob:1}. 

For the second and third terms from \eqref{riwiu}, we estimate via 
\begin{align}
|\int \p_y^2 \zeta U^2 x| + |\int_{y = 0} \p_y \zeta U^2 x \phi_1^2| \lesssim \sqrt{\eps} \| U \langle x \rangle^{- \frac 1 2 - \frac{1}{100}} \|^2 + \sqrt{\eps} \| U \langle x \rangle^{- \frac 1 2} \|_{y = 0}^2 \lesssim \sqrt{\eps} \| U, V \|_{X_0}^2,
\end{align}
where we have appealed to \eqref{est:zeta:2}.  
  
 The $(\zeta_y q)_y$ and $(\zeta V)_x$ terms on the right-hand side of \eqref{eq:vort:pre} are estimated in a completely analogous manner.  We now estimate the term 
 \begin{align} \n
| \int \eps (\alpha U)_x U_y x \phi_1^2 | \le & |\int \eps \alpha U_x U_y x \phi_1^2| + |\int \eps \alpha_x U U_y x \phi_1^2| \\ \n
\lesssim & \sqrt{\eps} \| \sqrt{\eps} U_x \phi_1 \| \| U_y \phi_1 \| + \eps \| U \langle x \rangle^{-1} \| \| U_y \phi_1 \| \\
\lesssim & \sqrt{\eps} (\| U, V \|_{X_0}^2 + \| U, V \|_{Y_{\frac 1 2}}^2) + \eps \| U, V \|_{X_0}^2, 
 \end{align}
 where we have appealed to estimate \eqref{S:1} to estimate the coefficient $\alpha$. The remaining term with $(\alpha_y q)_x$ is estimated in a completely analogous manner. This concludes the proof of Lemma \ref{lemma:y:half}. 
 \end{proof}

\subsection{$X_n$ Estimates, $1 \le n \le 10$}

It is convenient to estimate the commutators, $\mathcal{C}_1^n, \mathcal{C}_2^n$, defined in \eqref{def:C1n} - \eqref{def:C2n}.
\begin{lemma} The quantities $\mathcal{C}_1^{n}, \mathcal{C}_2^{n}$ satisfy the following estimates 
\begin{align} \label{twins:1}
\| \p_y^j \mathcal{C}_1^{n} \langle x \rangle^{n + \frac 1 2 + \frac j 2} \phi_n \| + \| \sqrt{\eps} \p_y^j \mathcal{C}_2^{n} \langle x \rangle^{n + \frac 1 2 + \frac j 2} \phi_n \| \lesssim \sqrt{\eps} \| U, V \|_{\mathcal{X}_{\le n-1 + \frac j 2}},
\end{align}
for $j = 0, 1$.
\end{lemma}
\begin{proof} We start with the estimation of $\mathcal{C}_1^n$, defined in \eqref{def:C1n}, which we do via 
\begin{align}  \n
\| \mathcal{C}_1^n  \langle x \rangle^{n + \frac 1 2} \phi_n \| \lesssim& \sum_{k = 0}^{n-1} \| \p_x^{n-k} \zeta \langle x \rangle^{(n-k) + 1.01} \|_\infty \| U^{(k)} \langle x \rangle^{k - \frac 1 2 - .01} \| \\  \label{dan:bek}
&+ \| \p_x^{n-k} \p_y \zeta \langle x \rangle^{(n-k) + 1.01} y \|_\infty \| \frac{q^{(k)}}{y} \langle x \rangle^{k - \frac 1 2 - .01} \|   \lesssim \sqrt{\eps} \|U, V \|_{\mathcal{X}_{\le n-1}},
\end{align}
where we have appealed to estimate \eqref{S:0} for the coefficient of $\zeta$. 

We now address the terms in $\mathcal{C}_2^{n}$ via 
\begin{align} \n
\| \sqrt{\eps} \mathcal{C}_2^n \langle x \rangle^{n + \frac 1 2} \phi_n \| \lesssim& \sum_{k = 0}^{n-1} \sqrt{\eps} \| \p_x^{n-k} \alpha \langle x \rangle^{(n-k) + \frac 3 2}  \|_\infty \|  U^{(k)} \langle x \rangle^{k-1} \phi_n \| \\ \n
& + \sqrt{\eps} \| \p_x^{n-k} \alpha_y \langle x \rangle^{(n-k) + \frac 3 2}  y \|_\infty \| \frac{ q^{(k)} }{y} \langle x \rangle^{k-1} \phi_n \| \\
& + \| \p_x^{n-k} \zeta \langle x \rangle^{(n-k) + 1.01} \|_\infty \| \sqrt{\eps} V^{(k)} \langle x \rangle^{k - \frac 1 2 - .01} \| \lesssim \sqrt{\eps} \|U, V \|_{\mathcal{X}_{\le n-1}},
\end{align}
where we have appealed to estimate \eqref{S:1} to estimate the coefficient $\alpha$. The higher order $y$ derivative works in an identical manner. 
\end{proof}

\begin{lemma} For any $n \ge 1$, 
\begin{align} \label{estXnnorm}
\| U, V \|_{X_n}^2 \lesssim \| U, V \|_{\mathcal{X}_{\le n-\frac 1 2}}^2 + \mathcal{T}_{X_{n}} + \mathcal{F}_{X_{n}},
\end{align}
where
\begin{align} \n
\mathcal{T}_{X_n} := &\int \p_x^{n} \mathcal{N}_1(u, v) U^{(n)} \langle x \rangle^{2n} \phi_n^2 + \int \eps \p_x^n \mathcal{N}_2(u, v) \Big( \eps V^{(n)} \langle x \rangle^{2n} \phi_n^2 + 2n \eps V^{(n-1)} \langle x \rangle^{2n-1} \phi_n^2 \\ \label{def:TXn}
&+ 2 \eps V^{(n-1)} \langle x \rangle^{2n} \phi_n \phi_n' \Big), \\ \n
\mathcal{F}_{X_n} := & \int \p_x^n F_R U^{(n)} \langle x \rangle^{2n} \phi_n^2 + \int \eps \p_x^n G_R \Big(  \eps V^{(n)} \langle x \rangle^{2n} \phi_n^2 + 2n \eps V^{(n-1)} \langle x \rangle^{2n-1} \phi_n^2 \\ \label{def:FXn}
&+ 2 \eps V^{(n-1)} \langle x \rangle^{2n} \phi_n \phi_n' \Big).
\end{align}
\end{lemma}
\begin{proof} We apply the multiplier 
\begin{align} \label{mult:Xn}
[U^{(n)} \langle x \rangle^{2n} \phi_{n}^2, \eps V^{(n)} \langle x \rangle^{2n} \phi_n^2 + 2n \eps V^{(n-1)} \langle x \rangle^{2n-1} \phi_n^2+ 2 \eps V^{(n-1)} \langle x \rangle^{2n} \phi_n \phi_n']
\end{align}
to the system \eqref{sys:sim:n1} - \eqref{sys:sim:n3}. The interaction of the multipliers \eqref{mult:Xn} with the the left-hand side of \eqref{sys:sim:n1} - \eqref{sys:sim:n2} is essentially identical to that of Lemma \ref{Lem:2}. As such, we treat the new commutators arising from the $\mathcal{C}_1^n, \mathcal{C}_2^n$ terms, defined in \eqref{def:C1n} - \eqref{def:C2n}. First, we have 
\begin{align} \n
|\int \mathcal{C}_1^{n} U^{(n)} \langle x \rangle^{2n} \phi_n^2| \lesssim & \| \mathcal{C}_1^{n} \langle x \rangle^{n + \frac 1 2} \phi_n \| \| U^{(n)} \langle x \rangle^{n - \frac 1 2} \phi_n \| \\
\lesssim & \sqrt{\eps} \|U, V \|_{\mathcal{X}_{\le n - 1}} ( \|U, V \|_{\mathcal{X}_{\le n - 1}} + \| U, V \|_{X_n} ).
\end{align}
Next, 
\begin{align} \n
|\int \eps \mathcal{C}_2^{n} V^{(n)} \langle x \rangle^{2n} \phi_n^2| \lesssim & \| \sqrt{\eps} \mathcal{C}_2^{n} \langle x \rangle^{n + \frac 1 2} \phi_n \| \| \sqrt{\eps} V^{(n)} \langle x \rangle^{n - \frac 1 2} \phi_n \| \\
\lesssim & \sqrt{\eps} \|U, V \|_{\mathcal{X}_{\le n - 1}} ( \|U, V \|_{\mathcal{X}_{\le n - 1}} + \| U, V \|_{X_n} ).
\end{align}
The identical estimate works as well for the middle term from the multiplier in \eqref{mult:Xn}, whereas the final term with $\phi_n'$ is localized in $x$ and lower order, and therefore trivially bounded by $\|U, V\|_{\mathcal{X}_{\le n - \frac 1 2}}^2$. 

\end{proof}

\subsection{$X_{n + \frac 1 2} \cap Y_{n + \frac 1 2}$ Estimates, $1 \le n \le 10$}

We now provide estimates on the higher order $X_{n + \frac 1 2}$ and $Y_{n + \frac 1 2}$ norms. Notice that these estimates still ``lose a derivative", due to degeneracy at $y = 0$. 
\begin{lemma} For any $0 < \delta << 1$, 
\begin{align} \label{esthalfnX}
\| U, V \|_{X_{n + \frac 1 2 }}^2 \le &C_\delta \|U, V \|_{\mathcal{X}_{\le n}}^2 + \delta \| U, V \|_{X_{n+1}}^2 + \delta \| U, V \|_{Y_{n + \frac 1 2}}^2  +  \mathcal{T}_{X_{n+ \frac 1 2}} + \mathcal{F}_{X_{n+ \frac 1 2}}, 
\end{align}
where we define 
\begin{align} \n
\mathcal{T}_{X_{n + \frac 1 2}} := & \int \p_x^{n} \mathcal{N}_1(u, v)U^{(n)}_x \langle x \rangle^{1+2n} \phi_{n+1}^2 + \int \eps \p_x^n \mathcal{N}_2(u, v) \Big( V^{(n)}_x \langle x \rangle^{1+2n} \phi_{n+1}^2 \\  \label{def:TXnp12}
&+  (1 + 2n) V^{(n)} \langle x \rangle^{2n} \phi_{n+1}^2 + 2 V^{(n)} \langle x\rangle^{1+2n} \phi_{n+1} \phi_{n+1}' \Big)   \\ \n
\mathcal{F}_{X_{n + \frac 1 2}} := & \int \p_x^{n}F_R U^{(n)}_x \langle x \rangle^{1+2n} \phi_{n+1}^2 + \int \eps \p_x^n G_R \Big( V^{(n)}_x \langle x \rangle^{1+2n} \phi_{n+1}^2 \\  \label{def:FXnp12}
&+  (1 + 2n) V^{(n)} \langle x \rangle^{2n} \phi_{n+1}^2 + 2 V^{(n)} \langle x\rangle^{1+2n} \phi_{n+1} \phi_{n+1}' \Big). 
\end{align}
\end{lemma}
\begin{proof} We apply the multiplier 
\begin{align}
[U^{(n)}_x \langle x \rangle^{1+2n} \phi_{n+1}^2, \eps V^{(n)}_x \langle x \rangle^{1+2n} \phi_{n+1}^2 + \eps (1 + 2n) V^{(n)} \langle x \rangle^{2n} \phi_{n+1}^2 + 2\eps V^{(n)} \langle x\rangle^{1+2n} \phi_{n+1} \phi_{n+1}']
\end{align}
to the system \eqref{sys:sim:n1} - \eqref{sys:sim:n3}. Again, the interaction of these multipliers with the left-hand side of  \eqref{sys:sim:n1} - \eqref{sys:sim:n3} is nearly identical to that of Lemma \ref{Lem:3}, and so we proceed to treat the commutators arising from $\mathcal{C}_1^{n}, \mathcal{C}_2^{n}$. We also may clearly estimate the contribution of the $\phi_{n+1}'$ term by a factor of $\|U, V \|_{\mathcal{X}_{\le n}}$. We have 
\begin{align} \n
|\int \mathcal{C}_1^{n} U^{(n)}_x \langle x \rangle^{1+2n} \phi_{n+1}^2| \lesssim &  \| \mathcal{C}_1^{n} \langle x \rangle^{n + \frac 1 2} \phi_{n+1} \| \| U_x^{(n)} \langle x \rangle^{n + \frac 1 2} \phi_{n+1} \| \\
\lesssim & \sqrt{\eps} \|U \|_{\mathcal{X}_{\le n-1}} ( \|U \|_{X_{n+\frac 1 2}} + \|U, V \|_{X_{n+1}} )
\end{align}
and similarly 
\begin{align} \n
&|\int \mathcal{C}_2^{n} (\eps V^{(n)}_x \langle x \rangle^{1+2n} \phi_{n+1}^2 + \eps (1 + 2n) V^{(n)} \langle x \rangle^{2n} \phi_{n+1}^2 )| \\ \n
\lesssim & \| \sqrt{\eps} C_2^{n} \langle x \rangle^{n + \frac 12} \phi_{n+1} \| ( \| \sqrt{\eps} V^{(n)}_x \langle x \rangle^{n + \frac 1 2} \phi_{n+1} \| + \| \sqrt{\eps} V^{(n)} \langle x \rangle^{n - \frac 12} \phi_{n+1} \| ) \\
\lesssim& \sqrt{\eps} \|U \|_{\mathcal{X}_{\le n-1}} ( \|U \|_{X_{n+\frac 1 2}} + \|U, V \|_{X_{n+1}} ),
\end{align}
where above we have invoked estimate \eqref{twins:1}. 
\end{proof}

\begin{lemma}For any $0 < \delta << 1$, 
\begin{align} \label{esthalfnY}
\| U, V \|_{Y_{n + \frac 1 2 }}^2 \le &C_\delta \|U, V \|_{\mathcal{X}_{\le n}}^2 + C_\delta \|U, V \|_E^2+ \delta \| U, V \|_{X_{n+1}}^2 + \delta \| U, V \|_{X_{n + \frac 1 2}}^2  +  \mathcal{T}_{Y_{n+ \frac 1 2}} + \mathcal{F}_{Y_{n+ \frac 1 2}},
\end{align}
where
\begin{align} \label{def:TYn12}
\mathcal{T}_{Y_{n + \frac 1 2}} := & \int \Big( \p_x^n \p_y \mathcal{N}_1(u, v) - \eps \p_x^{n+1} \mathcal{N}_2(u, v) \Big) U^{(n)}_y \langle x \rangle^{1 + 2n} \phi_{n+1}^2, \\  \label{def:TXn12}
\mathcal{F}_{Y_{n + \frac 1 2}} := & \int \Big( \p_x^n \p_y F_R - \eps \p_x^{n+1} G_R \Big) U^{(n)}_y \langle x \rangle^{1 + 2n} \phi_{n+1}^2. 
\end{align}
\end{lemma}
\begin{proof} We again only need to estimate the commutator terms, which are  
\begin{align} \n
|\int ( \p_y \mathcal{C}_1^n - \eps \p_x \mathcal{C}_2^n) U^{(n)}_y\langle x \rangle^{1 + 2n} \phi_{n+1}^2| \lesssim  &\| ( \p_y \mathcal{C}_1^n - \eps \p_x \mathcal{C}_2^n) \langle x \rangle^{n + 1} \phi_{n+1}  \| U^{(n)}_y \langle x \rangle^{n } \phi_{n+1} \| \\
\lesssim & \sqrt{\eps} \|U, V \|_{\mathcal{X}_{\le n - \frac 1 2}} \|U, V \|_{\mathcal{X}_{\le n + \frac 1 2}},
\end{align}
with the help again of estimate \eqref{twins:1}.
\end{proof}

\section{Top Order Estimates} \label{section:top:order}

In this section, we obtain top order control over the solution, more specifically we provide an estimate for $\| U, V \|_{X_{11}}$, defined in \eqref{def:Xn}. To establish this, we need to first perform a \textit{nonlinear change of variables} and to define auxiliary norms which are nonlinear (these will eventually control the $\| U, V \|_{X_{11}}$).    

\subsection{Nonlinear Change of Variables}

We group the linearized and nonlinear terms from \eqref{vel:eqn:1} via 
\begin{align} \label{id:1L1}
\mathcal{L}_1[u, v] + \mathcal{N}_1(u, v) = \mu_s u_x + \mu_{sy} v + \bar{v} u_y + \bar{u}_x u , 
\end{align}
where we have denoted the nonlinear coefficients by 
\begin{align} \label{def:mu:s:nu:s}
\mu_s := \bar{u} + \eps^{\frac{N_2}{2}}u, \qquad \nu_s := \bar{v} + \eps^{\frac{N_2}{2}}v. 
\end{align}

We now apply $\p_x^{11}$ to \eqref{id:1L1}, which produces the identity 
\begin{align} \label{sum:1}
\p_x^{11} ( \mathcal{L}[u, v] + \mathcal{N}_1(u, v) ) = \mu_s u^{(11)}_x + \mu_{sy} v^{(11)} + \sum_{i = 1}^3 \mathcal{R}^{(i)}_1[u, v]. 
\end{align}
where we have isolated those terms with twelve $x$ derivatives, and the remainder terms above have fewer than twelve $x$ derivatives on $u$, and are defined by 
\begin{align} \label{R11} 
\mathcal{R}_1^{(1)}[u, v] :=&  \sum_{j = 1}^{10} \binom{11}{j} ( \p_x^j \mu_s \p_x^{11-j} u_x + \p_x^{11-j} \bar{u}_x \p_x^j u + \p_x^j \nu_s \p_x^{11-j} u_y + \p_x^{11-j} \bar{u}_y \p_x^j v ), \\ \label{R12}
\mathcal{R}_1^{(2)}[u, v] := & \mu_{sx} u^{(11)} + \nu_s u^{(11)}_y, \\  \label{R13}
\mathcal{R}_1^{(3)}[u, v] := & u \p_x^{12} \bar{u} + u_y \p_x^{11} \bar{v} + v \p_y \bar{u}^{(11)} + u_x \p_x^{11} \bar{u}. 
\end{align}

We now introduce the change of variables, which is adapted to the first two terms on the right-hand side of \eqref{sum:1}. The basic objects are 
\begin{align} \label{Chan:var:2}
Q := \frac{\psi^{(11)}}{\mu_s}, \qquad \tilde{U} := \p_y Q, \qquad \tilde{V} := - \p_x Q. 
\end{align}

From here, we derive the identities 
\begin{align} \label{move:2}
&u^{(11)} = \p_x^{11} u = \p_y \psi^{(11)} = \p_y (\mu_s Q) = \mu_s \tilde{U} + \p_y \mu_s Q, \\ \label{move:3}
&v^{(11)} = \p_x^{11} v = - \p_x \psi^{(11)} = - \p_x (\mu_s Q) = \mu_s \tilde{V} - \p_x \mu_s Q. 
\end{align}
We thus rewrite the primary two terms from \eqref{sum:1} as 
\begin{align} \n
\mu_s u^{(11)}_x + \mu_{sy} v^{(11)} = & \mu_s \mu_{sx} \tilde{U} + \mu_s^2 \tilde{U}_x + \mu_s \mu_{sxy} Q - \mu_s \mu_{sy} \tilde{V} + \mu_s \mu_{sy} \tilde{V} - \mu_{sy} \mu_{sx} Q \\
= & \mu_s^2 \tilde{U}_x + \mu_s \mu_{sx} \tilde{U} + (\mu_s \mu_{sxy} - \mu_{sx} \mu_{sy}) Q. 
\end{align}

We may subsequently rewrite \eqref{sum:1} via 
\begin{align}
\p_x^{11} ( \mathcal{L}_1[u, v] + \mathcal{N}_1(u, v) ) = \mu_s^2 \tilde{U}_x + \mu_s \mu_{sx} \tilde{U} + (\mu_s \mu_{sxy} - \mu_{sx} \mu_{sy}) Q + \sum_{i = 1}^3 \mathcal{R}_1^{(i)}[u, v]. 
\end{align}

We now address the second equation, for which we similarly record the identity 
\begin{align} \label{move:1}
\p_x^{11} (\mathcal{L}_{2}[u, v] + \mathcal{N}_2(u, v)) = & \mu_s v^{(11)}_x + \nu_s v^{(11)}_y +  \mathcal{R}_2^{(1)}[u, v] + \mathcal{R}_2^{(2)}[u, v] + \mathcal{R}_2^{(3)}[u, v],  
\end{align}
where we again define the lower order terms appearing above via  
\begin{align}  \label{def:R21:def}
\mathcal{R}^{(1)}_2[u, v] := & \sum_{j = 1}^{10} \binom{11}{j} ( \p_x^j \mu_s \p_x^{11-j} v_x + \p_x^{11-j} \bar{v}_x \p_x^j u + \p_x^j \nu_s \p_x^{11-j} v_y + \p_x^{11-j} \bar{v}_y \p_x^j v  ), \\ \label{def:R22}
\mathcal{R}_2^{(2)}[u, v] := &\nu_{sx} u^{(11)} + \nu_{sy} v^{(11)}, \\ \label{def:R23}
\mathcal{R}_2^{(3)}[u, v] := & v \p_x^{11} \bar{v}_y  + u \p_x^{12} \bar{v} + v_x \p_x^{11} \bar{u} + v_y  \p_x^{11}\bar{v}.   
\end{align}

We will now rewrite the first two terms from \eqref{move:1} by using \eqref{move:2} - \eqref{move:3} so as to produce 
\begin{align} \n
\mu_s v^{(11)}_x + \nu_s v^{(11)}_y = & \mu_s^2 \tilde{V}_x + \mu_s \nu_s \tilde{V}_y + (2 \mu_s \mu_{sx} + \nu_s \mu_{sy}) \tilde{V} - \mu_{sx} \nu_s \tilde{U} - (\mu_s \mu_{sxx} + \nu_s \mu_{sxy})Q.
\end{align}
Continuing then from \eqref{move:1}, we obtain 
\begin{align} \n
\p_x^{11} (\bar{\mathcal{L}}_{2}[u, v] + \mathcal{N}_2(u, v)) = & \mu_s^2 \tilde{V}_x + \mu_s \nu_s \tilde{V}_y + (2 \mu_s \mu_{sx} + \nu_s \mu_{sy}) \tilde{V} - \mu_{sx} \nu_s \tilde{U} \\
& -(\mu_s \mu_{sxx} + \nu_s \mu_{sxy})Q + \sum_{i = 1}^3 \mathcal{R}_2^{(i)}[u, v]. 
\end{align}

We now summarize the full nonlinear equation upon introducing these new quantities:
\begin{align} \label{eq:11:first}
\mu_s^2 \tilde{U}_x + & \mu_s \mu_{sx} \tilde{U} - \Delta_\eps (\p_x^{11} u) + (\mu_s \mu_{sxy} - \mu_{sx} \mu_{sy}) Q  + \sum_{i = 1}^3 \mathcal{R}_1^{(i)}[u, v] + \p_x^{11} P_x = \p_x^{11} F_R, 
\end{align}
and the second equation which reads
\begin{align} \n
\mu_s^2 \tilde{V}_x& + \mu_s \nu_s \tilde{V}_y -  \Delta_\eps (\p_x^{11} v) + (2 \mu_s \mu_{sx} + \nu_s \mu_{sy}) \tilde{V} - \mu_{sx} \nu_s \tilde{U} - (\mu_s \mu_{sxx} + \nu_s \mu_{sxy})Q \\ \label{eq:11:second}
&+ \sum_{i = 1}^3 \mathcal{R}_2^{(i)}[u, v] +  \p_x^{11}\frac{P_y}{\eps} = \p_x^{11} G_R. 
\end{align}

\subsection{Nonlinearly Modified Norms} \label{subsection:NLMN}

While our objective is to control $\| U, V \|_{X_{11}}$, we will need to change the weights appearing in this norm from $\bar{u}$ to $\mu_s$. Define thus
\begin{align} \label{def:Theta:11}
\| \tilde{U}, \tilde{V} \|_{\Theta_{11}} := \| \sqrt{\mu_s} \tilde{U}_y x^{11} \phi_{11} \| + \sqrt{\eps} \| \sqrt{\mu_s} \tilde{U}_x x^{11} \phi_{11} \| + \eps \| \sqrt{\mu_s} \tilde{V}_x x^{11} \phi_{11} \| + \| \mu_{sy} \tilde{U} x^{11} \phi_{11} \|_{y = 0}. 
\end{align}

We now prove
\begin{lemma} The following estimates are valid, for $j = 0, 1$,
\begin{align}  \label{same:Q}
\| \sqrt{\eps}Q x^{9.5} \phi_{10} \| \lesssim & \|U, V \|_{\mathcal{X}_{\le 10}},  \\ \label{same:Q:2}
\| \sqrt{\eps} Q x^{10} \phi_{11} \|_{L^2_x L^\infty_y} \lesssim & \|U, V \|_{\mathcal{X}}, \\\label{same:same}
\| \mu_s \tilde{U} x^{10.5} \phi_{11} \| + \sqrt{\eps} \| \mu_s \tilde{V} x^{10.5} \phi_{11} \| \lesssim &\eps^{\frac{N_2}{2}- M_1 - 5}\|U, V \|_{\mathcal{X}} + \|U, V \|_{\mathcal{X}_{\le 10.5}}, 
\end{align}
and, for any $0 < \delta <<  1$,
\begin{align} \label{same:same:high} 
\|  \tilde{U} x^{10.5} \phi_{11} \| + \sqrt{\eps} \|  \tilde{V} x^{10.5} \phi_{11} \| \le & \delta \| \tilde{U}, \tilde{V} \|_{\Theta_{11}} + C_\delta \|U, V \|_{\mathcal{X}_{\le 10.5}} + \eps^{\frac{N_2}{2}- M_1 - 5} \|U, V \|_{\mathcal{X}}. 
\end{align}
\end{lemma}
\begin{proof} We use the formulas \eqref{Chan:var:2} to write 
\begin{align} \label{trans:Q}
\mu_s Q = \p_x^{11} \psi = \p_x^{11} (\bar{u} q) = \bar{u} \p_x^{11} q + \sum_{k = 1}^{11} \binom{11}{k} \p_x^k \bar{u} \p_x^{11-k} q. 
\end{align}
We divide through both sides by $\mu_s$, multiply by $\sqrt{\eps} x^{9.5} \phi_{10}$, and compute the $L^2$ norm, which gives
\begin{align} \n
\| \sqrt{\eps} Q x^{9.5} \phi_{10} \| \lesssim & \| \sqrt{\eps} V^{(9)}_x x^{9.5} \phi_{10} \| + \sum_{k = 1}^9 \| \frac{\p_x^k \bar{u}}{\bar{u}} x^k \|_\infty \| \sqrt{\eps} V^{(9-k)}_x x^{9-k + \frac 1 2} \phi_{11} \| \\ \n
& + \| \p_x^{10} \bar{u}_P x^9 y \|_\infty \|  \sqrt{\eps} \frac{V}{y} x^{\frac 1 2} \phi_{11} \|  + \| \p_x^{10} \bar{u}_E x^{10.5} \|_\infty \| \sqrt{\eps} V \langle x \rangle^{-1} \| \\  \n
& + \| \p_x^{11} \bar{u}_P x^{10.5}y \|_\infty \|  \sqrt{\eps} \frac{q}{y} \langle x \rangle^{-1} \| + \| \p_x^{11} \bar{u}_E x^{11.5} \|_\infty \| \sqrt{\eps} q \langle x \rangle^{-2} \| \\
 \lesssim &  \|U, V \|_{\mathcal{X}_{\le 10}}, 
\end{align}
where we need to treat the lower order terms corresponding to $k = 9, 10$ in the sum \eqref{trans:Q} differently, in order to avoid using the critical Hardy inequality. We have invoked estimates \eqref{Hardy:four}, \eqref{est:Eul:piece}, \eqref{est:Pr:piece}. 

We now address the second inequality, \eqref{same:Q:2}. We divide \eqref{trans:Q} by $\mu_s$, multiply by $\sqrt{\eps} x^{10} \phi_{11}$, and compute the $L^2_x L^\infty_y$ norm, which gives 
\begin{align} \n
\| \sqrt{\eps} Q x^{10} \phi_{11} \|_{L^2_x L^\infty_y} \lesssim & \| V^{(10)} x^{10} \phi_{11} \|_{L^2_x L^\infty_y} + \sum_{k = 1}^9 \| \frac{\p_x^{k} \bar{u}}{\bar{u}} x^{k} \|_{\infty} \| V^{(10-k)} x^{10-k} \phi_k \|_{L^2_x L^\infty_y} \\ \n
& + \| \p_x^{10} \bar{u}_P  x^{9.5} y \|_{\infty} \| \frac{V}{y} x^{\frac 1 2} \|_{L^2_x L^\infty_y} + \| \p_x^{10} \bar{u}_E x^{10.5} \|_\infty \| V \langle x \rangle^{- \frac 1 2} \|_{L^2_x L^\infty_y} \\ \label{bieb:1}
& +  \| \p_x^{11} \bar{u}_P  x^{10.5} y \|_{\infty} \| \frac{q}{y} x^{-\frac 1 2} \|_{L^2_x L^\infty_y} + \| \p_x^{11} \bar{u}_E x^{11.5} \|_\infty \| q \langle x \rangle^{- \frac 3 2} \|_{L^2_x L^\infty_y} \\ \n
\lesssim & \|U, V \|_{\mathcal{X}}, 
\end{align}
where we have again invoked estimates \eqref{est:Eul:piece}, \eqref{est:Pr:piece} for the $\bar{u}$ terms, the mixed norm estimate \eqref{Lpq:emb:V:1},  as well as the following Sobolev interpolation estimates 
\begin{align}
&\| \frac{V}{y} x^{\frac 1 2} \phi_{11} \|_{L^2_x L^\infty_y} \lesssim \| V_y x^{\frac 1 2} \phi_{11} \|_{L^2_x L^\infty_y} \lesssim \| V_y x^{\frac 1 2} \phi_{11} \|^{\frac 1 2} \| U^{(1)}_y x^{\frac 1 2} \phi_{11} \|^{\frac 1 2} \lesssim \|U, V \|_{\mathcal{X}_{\le 1.5}}, \\
&\| \sqrt{\eps} V \langle x \rangle^{- \frac 1 2} \|_{L^2_x L^\infty_y} \lesssim \| \sqrt{\eps} V \langle x \rangle^{-1} \phi_{11} \|^{\frac 1 2} \| \sqrt{\eps} U_x \phi_{11} \| \lesssim \| U, V \|_{\mathcal{X}_{\le 1}},
\end{align}
and the analogous estimates for $q$ instead of $V$ for the final two terms from \eqref{bieb:1}.

Dividing through by $\mu_s$ and differentiating in $y$ yields 
\begin{align}  \label{trans:U}
\tilde{U} = \frac{\bar{u}}{\mu_s} U^{(11)}  - \p_y ( \frac{\bar{u}}{\mu_s} ) V^{(10)} + \sum_{k = 1}^{11} \binom{11}{k} \p_y ( \frac{\p_x^k \bar{u}}{\mu_s} ) \p_x^{11-k} q + \sum_{k = 1}^{11} \binom{11}{k}  \frac{\p_x^k \bar{u}}{\mu_s}  U^{(11-k)},  
\end{align}
and similarly, dividing through by $\mu_s$ and differentiating in $x$ yields 
\begin{align}  \label{trans:V}
\tilde{V} = \frac{\bar{u}}{\mu_s} V^{(11)} + \p_x (\frac{\bar{u}}{\mu_s}) V^{(10)} + \sum_{k = 1}^{11} \binom{11}{k} \frac{\p_x^k \bar{u}}{\mu_s} V^{(11-k)} - \sum_{k = 1}^{11} \binom{11}{k} \p_x (\frac{\p_x^k \bar{u}}{\mu_s}) \p_x^{11-k} q. 
\end{align}
We first establish the following auxiliary estimate, which will be needed in forthcoming calculations due to the second term from \eqref{trans:U}.
\begin{align} \n
\p_y (\frac{\bar{u}}{\mu_s}) = \p_y (\frac{\mu_s - \eps^{\frac{N_2}{2}} u }{\mu_s}) = - \eps^{\frac{N_2}{2}} \p_y (\frac{u}{\mu_s}) = - \eps^{\frac{N_2}{2}} \p_y ( \frac{u}{\bar{u}} \frac{\bar{u}}{\mu_s} ) = - \eps^{\frac{N_2}{2}} \frac{\bar{u}}{\mu_s} \p_y (\frac{u}{\bar{u}}) - \eps^{\frac{N_2}{2}} \frac{u}{\bar{u}} \p_y (\frac{\bar{u}}{\mu_s}),
\end{align}
which, rearranging for the quantity on the left-hand side, yields the identity 
\begin{align}
\p_y (\frac{\bar{u}}{\mu_s}) = - \frac{\eps^{\frac{N_2}{2}}}{1 + \eps^{\frac{N_2}{2}} \frac{u}{\bar{u}} } \frac{\bar{u}}{\mu_s} \p_y (\frac{u}{\bar{u}}),
\end{align}
from which we estimate 
\begin{align}  \label{depen:1}
\| \p_y (\frac{\bar{u}}{\mu_s}) x^{\frac 12} \psi_{12} \|_{L^\infty_x L^2_y} \lesssim & \eps^{\frac{N_2}{2}} \| \frac{\bar{u}}{\mu_s} \|_{L^\infty} \frac{1}{1 - \eps^{\frac{N_2}{2}} \| \frac{u}{\bar{u}} \|_{L^\infty} } \| \p_y (\frac{u}{\bar{u}}) x^{\frac 1 2} \psi_{12} \|_{L^\infty_x L^2_y} \lesssim  \eps^{\frac{N_2}{2}-M_1} \|U,V \|_{\mathcal{X}},
\end{align}
where we have invoked estimates \eqref{mixed:emb} and \eqref{Linfty:wo}. 

From these formulas, we provide the estimate \eqref{same:same:high} via 
\begin{align} \n
\| \mu_s \tilde{U} x^{10.5} \phi_{11} \| & \lesssim  \| \bar{u} U^{(10)}_x x^{10.5} \phi_{11} \| + \| \p_y ( \frac{\bar{u}}{\mu_s}) x^{\frac 1 2} \psi_{12} \|_{L^\infty_x L^2_y} \| V^{(10)} x^{10} \phi_{11} \|_{L^2_x L^\infty_y} \\ \n
 &+ \sum_{k = 1}^{10}  \| \frac{\bar{u}}{\mu_s} \|_{\infty} \| \p_y (\frac{\p_x^{k} \bar{u}}{\bar{u}})  y \langle x \rangle^k \|_{\infty} \| \frac{\p_x^{11-k} q}{y} \langle x \rangle^{11-k - \frac 1 2} \| \\ \n
& + \| \frac{\bar{u}}{\mu_s} \|_\infty \| \p_y (\frac{\p_x^{11} \bar{u}_P}{\bar{u}}) y^2 \langle x \rangle^{10.5} \|_\infty \| \frac{q - y U(x, 0)}{\langle y \rangle^2} \| + \| \frac{\bar{u}}{\mu_s} \|_\infty \| \p_y (\frac{\p_x^{11} \bar{u}_E}{\bar{u}}) y \langle x \rangle^{11.5} \|_\infty \| U \langle x \rangle^{-1} \| \\ \n
&+ \sum_{k = 1}^{10} \| \frac{\bar{u}}{\mu_s} \|_\infty \| \frac{\p_x^k \bar{u}}{\bar{u}} x^k \|_\infty \| U^{(11-k)} \langle x \rangle^{(11-k- \frac 1 2)} \| + \| \frac{\bar{u}}{\mu_s} \|_\infty \| \frac{\p_x^{11} \bar{u}_P}{\bar{u}} y \langle x \rangle^{10.5} \|_\infty \| \frac{U - U(x, 0)}{\langle y \rangle} \| \\ \n
& + \| \frac{\bar{u}}{\mu_s} \|_\infty \| \frac{\p_x^{11} \bar{u}_P}{\bar{u}} \langle x \rangle^{11-\frac 1 4} \|_{L^\infty_x L^2_y} \| U(x, 0) \langle x \rangle^{- \frac 1 4} \|_{L^2_x} + \| \frac{\p_x^{11} u_E}{\bar{u}} \langle x \rangle^{11.5} \|_\infty \| U \langle x \rangle^{-1} \| \\
&  \lesssim  \| U, V \|_{\mathcal{X}_{\le 10.5}} + \eps^{\frac{N_2}{2} - M_1} \| U, V \|_{\mathcal{X}},
\end{align}
where we have invoked estimates \eqref{prof:u:est}, \eqref{est:Eul:piece}, \eqref{est:Pr:piece}, \eqref{depen:1}, and  \eqref{Lpq:emb:V:1}.
An essentially identical proof applies also to the $\tilde{V}$ quantity from \eqref{same:same}, so we omit repeating these details. This establishes estimate \eqref{same:same}, with \eqref{same:same:high} following similarly, upon using the Hardy-type inequality \eqref{Hardy:three}. 
\end{proof}

As long as we have sufficiently strong control on lower-order quantities, it will turn our that the $\Theta_{11}$ norm will control the $X_{11}$ norm. This is the content of the following lemma. 
\begin{lemma} Assume $\| U, V \|_{\mathcal{X}} \le 1$. Then, 
\begin{align}
\| U, V \|_{\mathcal{X}} \lesssim \| \tilde{U}, \tilde{V} \|_{\Theta_{11}}+ \| U, V \|_{\mathcal{X}_{\le 10.5}} \lesssim \| U, V \|_{\mathcal{X}}.
\end{align}
\end{lemma}
\begin{proof} Dividing through equation \eqref{trans:Q} by $\bar{u}$ and computing $\p_y^2$ gives 
\begin{align} \n
U^{(11)}_y = & \frac{\mu_s}{\bar{u}} \tilde{U}_y + 2\p_y (\frac{\mu_s}{\bar{u}}) \tilde{U} + \p_y^2 (\frac{\mu_s}{\bar{u}}) Q   - \sum_{k = 1}^{11} \binom{11}{k} \p_y^2 (\frac{\p_x^k \bar{u}}{\bar{u}}) \p_x^{11-k} q \\
& - 2 \sum_{k = 1}^{11} \binom{11}{k}  \p_y^2 (\frac{\p_x^k \bar{u}}{\bar{u}} ) \p_x^{11-k} U - \sum_{k = 1}^{11} \binom{11}{k} \frac{\p_x^k \bar{u}}{\bar{u}} \p_x^{11-k} U_y 
\end{align}
From here, we obtain the estimate 
\begin{align} \n
\| \sqrt{\bar{u}} U^{(11)}_y x^{11} \phi_{11} \| \lesssim & \| \sqrt{ \frac{\mu_s}{\bar{u}} } \|_\infty \| \sqrt{\mu_s} \tilde{U}_y x^{11} \phi_{11} \| + \| \sqrt{\bar{u}} \p_y (\frac{\mu_s}{\bar{u}}) x^{\frac 1 2} \|_\infty \| \tilde{U} x^{10.5} \phi_{11} \|  \\ \n
& + \eps^{\frac{N_2}{2}} \| u_{yy} x \psi_{12} \|_{L^\infty_x L^2_y} \| Q x^{10} \phi_{11} \|_{L^2_x L^\infty_y} \\  \n
& + \sum_{k = 1}^{11} \| \p_y^2(\frac{\p_x^k \bar{u}}{\bar{u}}) y x^{k + \frac 1 2} \|_\infty \| \frac{\p_x^{11-k} q}{y} x^{11-k - \frac 1 2} \phi_{11} \| \\ \n
& + \sum_{k = 1}^{11} \| \frac{\p_x^k \bar{u}}{\bar{u}} x^k \|_\infty \| U^{(11-k)}_y x^{11-k} \phi_{11} \| \\
\lesssim & \| \tilde{U}, \tilde{V} \|_{\Theta_{11}} + \| U, V \|_{\mathcal{X}_{\le 10.5}} + \eps^{\frac{N_2}{2} - M_1} \| U, V \|_{\mathcal{X}} \| U, V \|_{\mathcal{X}_{\le 10.5}},   
\end{align}
where we have invoked \eqref{Linfty:wo}, \eqref{mixed:emb}, \eqref{same:Q:2}, and \eqref{bob:1}. 

An essentially identical calculation applies to the remaining terms from the $\| U, V \|_{X_{11}}$ norm, and also an essentially identical computation enables us to go backwards. We note, however, that to compare the quantities $ \| \mu_{sy} \tilde{U} x^{11} \phi_{11} \|_{y = 0}$ and $\| \bar{u}_y U^{(11)} x^{11} \phi_{11} \|_{y = 0}$, we also need to demonstrate boundedness of the coefficients $|\frac{\bar{u}_y}{\mu_{sy}}| \phi_{11}$ and $|\frac{\mu_{sy}}{\bar{u}_y}| \phi_{11}$. For this purpose, we estimate 
\begin{align} \n
|\mu_{sy}(x, 0) - \bar{u}^0_{py}(x, 0)| \phi_{11} \le  &\sum_{i = 1}^{N_1} \eps^{\frac i 2} (\sqrt{\eps} \| u^i_{EY} \|_{L^\infty_y} + \| u^i_{py} \|_{L^\infty_y} ) + \eps^{\frac{N_2}{2}} \| u_y \psi_{12} \|_{L^\infty_y} \\
\lesssim & \sum_{i = 1}^{N_1} \eps^{\frac i 2} (\sqrt{\eps} \langle x \rangle^{- \frac 3 2} + \langle x \rangle^{- \frac 3 4 + \sigma_\ast}) + \eps^{\frac{N_2}{2}- M_1} \langle x \rangle^{- \frac 3 4} \| U, V \|_{\mathcal{X}} \lesssim \eps^{\frac 1 2} \langle x \rangle^{- \frac 3 4 + \sigma_\ast}, 
\end{align}
where we have invoked estimates \eqref{water:88}, \eqref{water:65}, \eqref{pw:dec:u}, the identity that $\phi_{11} = \psi_{12} \phi_{11}$, the fact that $\frac{N_2}{2} - M_1 >> 0$, and finally the assumption that $\|U, V \|_{\mathcal{X}} \le 1$. 
\end{proof}

We will need the following interpolation estimates to close the nonlinear estimates below. 
\begin{lemma} \label{lemma:Mixed:here} Let $\tilde{W} \in \{ \tilde{U}, \tilde{V}\}$. The following estimates are valid:
\begin{align} \label{weds:weds:1}
\| \bar{u}^{\frac 1 2} \tilde{W} \langle x \rangle^{10.75} \phi_{11} \|_{L^2_x L^\infty_y}^2 \lesssim & \| \bar{u} \tilde{W} \langle x \rangle^{10.5} \phi_{11} \|^2 + \| \sqrt{\bar{u}} \tilde{W}_y \langle x \rangle^{11} \phi_{11} \|^2, \\ \label{weds:weds:2}
\| \tilde{W} \langle x \rangle^{10.5} \phi_{11} \chi(z) \|_{L^2_x L^4_y}^2 \lesssim & \| \bar{u} \tilde{W} \langle x \rangle^{10.5} \phi_{11} \|^2 + \| \sqrt{\bar{u}} \tilde{W}_y \langle x \rangle^{11} \phi_{11} \|^2.
\end{align}
\end{lemma}
\begin{proof} We begin with the first estimate, \eqref{weds:weds:1}. For this, we consider 
\begin{align} \n
\bar{u} \tilde{W}^2 \langle x \rangle^{21.5} \le & \langle x \rangle^{21.5} |\int_y^\infty 2 \bar{u}  \tilde{W} \tilde{W}_y(y') \ud y' | + \langle x \rangle^{21.5} |\int_y^\infty 2 \tilde{W}^2 \bar{u}_y \ud y' | \\
\lesssim&  \| \tilde{W} \langle x \rangle^{10.5} \|_{L^2_y} \| \sqrt{\bar{u}} \tilde{W}_y \langle x \rangle^{11} \|_{L^2_y} + \| \bar{u}_y \langle x \rangle^{\frac 1 2} \|_\infty \|  \tilde{W} \langle x \rangle^{10.5} \|_{L^2_y}^2. 
\end{align}
Multiplying both sides by $\phi_{11}^2$ and placing both sides in $L^2_x$ gives estimate \eqref{weds:weds:1}. 

We turn now to \eqref{weds:weds:2}. To compute the $L^4_y$, we raise to the fourth power and integrate by parts via  
\begin{align} \n
\| \tilde{W} \langle x \rangle^{10.5} \chi(z) \|_{ L^4_y}^4 =  &\int \tilde{W}^4 \chi(z)  \langle x \rangle^{42} \ud y = \int \p_y (y) \tilde{W}^4 \chi(z) \langle x \rangle^{42} \ud y \\ \label{lauvlauv}
=& - \int 4 y \tilde{W}^3 \tilde{W}_y \chi(z) \langle x \rangle^{42} - \int y \tilde{W}^4 \frac{1}{\sqrt{x}} \chi'(z) \langle x \rangle^{42} \ud y.
\end{align}
We will first handle the first integral above. Using that $z \le 1$ on the support of $\chi(z)$, we can estimate via 
\begin{align} \n
|\int 4y \tilde{W}^3 \tilde{W}_y \chi(z) \langle x \rangle^{42}| \lesssim & \int |\tilde{W}|^3 \bar{u} |\tilde{W}_y| \langle x \rangle^{42.5} \ud y \\
\lesssim & \| \sqrt{\bar{u}} \tilde{W} \langle x \rangle^{10.5} \|_{L^\infty_y} \| \tilde{W} \langle x \rangle^{10.5} \|_{L^4_y}^2 \| \sqrt{\bar{u}} \tilde{W}_y \langle x \rangle^{11} \|_{L^2_y},
\end{align}

For the far-field integral from \eqref{lauvlauv}, we estimate it by first noting that $|y/\sqrt{x}| \lesssim 1$ on the support of $\chi'$. Thus,  
\begin{align}
|\int \tilde{W}^4 \langle x \rangle^{42} z \chi'(z) \ud y| \lesssim \| \bar{u} \tilde{W} \langle x \rangle^{10.5}  \|_{L^\infty_y}^2 \| \bar{u} \tilde{W} \langle x \rangle^{10.5} \|_{L^2_y}^2 \lesssim \| \bar{u} \tilde{W} \langle x \rangle^{10.5} \|_{L^2_y}^3 \| \sqrt{\bar{u}} \tilde{W}_y \langle x \rangle^{10.5} \|_{L^2_y}. 
\end{align}
Multiplying by $\phi_{11}^2$, integrating over $x$, and appealing to \eqref{weds:weds:1} gives the desired estimate.
\end{proof}

\subsection{Complete $\|\tilde{U}, \tilde{V} \|_{\Theta_{11}}$ Estimate} \label{subsection:NL:PR}

Before performing our main top order energy estimate in Lemma \ref{lemma:top}, we first record an estimate on the lower-order error terms. 
\begin{lemma} Assume $\| U, V \|_{\mathcal{X}} \le 1$. Let $\mathcal{R}_1^{(1)}$ and $\mathcal{R}_2^{(1)}$ be defined as in \eqref{R11} and \eqref{def:R21:def}. Then the following estimate is valid, for any $0 < \delta << 1$, 
\begin{align} \label{R11:estimate}
\| \mathcal{R}_{1}^{(1)} \langle x \rangle^{11.5} \phi_{11} \| + \| \sqrt{\eps} \mathcal{R}_2^{(1)} \langle x \rangle^{11.5} \phi_{11} \| \le \delta \| U, V \|_{\mathcal{X}} + C_\delta \| U, V \|_{\mathcal{X}_{\le 10.5}}.
\end{align}
\end{lemma} 
\begin{proof} We begin first with $\mathcal{R}_1^{(1)}$, defined in \eqref{R11}. For the first term in \eqref{R11}, we assume $1 \le j \le 6$, in which case we bound 
\begin{align} \n
\| \p_x^j \mu_s \p_x^{11-j} u_x \langle x \rangle^{11.5} \phi_{11} \| \lesssim & \| \p_x^j \mu_s \langle x \rangle^j \psi_{12} \|_\infty \| \p_x^{11-j} u_x \langle x \rangle^{11-j+\frac 1 2} \phi_{11} \| \\
\lesssim & (1 + \eps^{\frac{N_2}{2} - M_1} \| U, V \|_{\mathcal{X}})(C_\delta \| U, V \|_{\mathcal{X}_{\le 10.5}} + \delta \| U, V \|_{\mathcal{X}}),
\end{align}
where we have invoked estimates \eqref{prof:u:est}, \eqref{pw:dec:u}, and \eqref{L2:uv:eq}. We have also invoked the identity $\psi_{12} \phi_{11} = \phi_{11}$ to insert the cut-off function $\psi_{12}$ freely above. The remaining case $j > 6$ can be treated symmetrically, as can the second term from \eqref{R11}. For the third term from \eqref{R11}, we first treat the case when $1 \le j \le 6$, which we estimate via 
\begin{align} \n
\| \p_x^j \nu_s \p_x^{11-j} u_y \langle x \rangle^{11.5} \phi_{11} \| \lesssim &\| \p_x^j \nu_s \langle x \rangle^{j + \frac 12} \psi_{12} \|_\infty \| \p_x^{11-j} u_y \langle  x\rangle^{11-j} \phi_{11} \| \\
\lesssim & (1 + \eps^{\frac{N_2}{2}- M_1} \| U, V \|_{\mathcal{X}} ) \| U, V \|_{\mathcal{X}_{\le 10}}, 
\end{align}
where above we have invoked \eqref{prof:v:est}, \eqref{pw:v:1}, and \eqref{L2:uv:eq:2}, and again the ability to insert freely the cut-off $\psi_{12}$ in the presence of $\phi_{11}$. In the case $6 < j \le 10$, we estimate the nonlinear component via 
\begin{align} \n
\eps^{\frac{N_2}{2}} \|  \p_x^j v  \p_x^{11-j} u_y \langle x \rangle^{11.5} \phi_{11} \| \lesssim & \eps^{\frac{N_2}{2}} \| \p_x^{11-j} u_y \langle x \rangle^{11-j + \frac 1 2} \psi_{12} \|_{L^\infty_x  L^2_y} \| \p_x^j v \langle x \rangle^j \phi_{11}\|_{L^2_x L^\infty_y} \\
\lesssim & \eps^{\frac{N_2}{2} - M_1} \|U, V \|_{\mathcal{X}} \| U, V \|_{\mathcal{X}_{\le 10.5}}
\end{align}
where we have used the mixed-norm estimates in \eqref{mL2again} and \eqref{mixed:emb}. 

We now move to $\mathcal{R}_2^{(1)}$. The first term here is estimated, in the case when $j \le 6$, via 
\begin{align} \n
\sqrt{\eps} \| \p_x^j \mu_s \p_x^{11-j} v_x \langle x \rangle^{11.5} \phi_{11} \| \lesssim &\| \p_x^j \mu_s \langle x \rangle^j \psi_{12} \|_\infty \| \sqrt{\eps} \p_x^{11-j} v_x \langle x \rangle^{11-j+ \frac 1 2} \phi_{11}\| \\
\lesssim & (1 + \eps^{\frac{N_2}{2}-M_1} \| U, V \|_{\mathcal{X}})  (C_\delta \|U, V \|_{\mathcal{X}_{\le 10.5}} + \delta \| U, V \|_{\mathcal{X}}),
\end{align}
where we have invoked \eqref{prof:u:est}, \eqref{L2:uv:eq}, \eqref{pw:dec:u}. 

In the case when $6 < j \le 10$, we estimate the nonlinear term via 
\begin{align} \n
\sqrt{\eps} \eps^{\frac{N_2}{2}} \| \p_x^j u \p_x^{11-j} v_x \langle x \rangle^{11.5} \phi_{11} \| \lesssim &\eps^{\frac{N_2}{2}} \| \sqrt{\eps} \p_x^{12-j} v \langle x \rangle^{12-j + \frac 1 2} \psi_{12} \|_\infty \| \p_x^{j-1} u_x \langle x \rangle^{j-1 + \frac 1 2} \phi_{11} \|, \\
\lesssim &  \eps^{\frac{N_2}{2}-M_1} \| U, V \|_{\mathcal{X}} \| U, V \|_{\mathcal{X}_{\le 10}}
\end{align}
where we have invoked \eqref{pw:v:1} and \eqref{L2:uv:eq}.
 
The same estimates work for the second term in $\mathcal{R}_2^{(1)}$, and so we move to the third term. In the case when $j \le 6$, we can estimate the third term via 
\begin{align} \n
\sqrt{\eps} \| \p_x^j \nu_s \p_x^{11-j} v_y \langle x \rangle^{11.5} \phi_{11} \| \lesssim & \sqrt{\eps} \| \p_x^j \nu_s \langle x \rangle^{j + \frac 1 2} \psi_{12}\|_\infty \| \p_x^{12-j} u \langle x \rangle^{11-j + \frac 1 2} \phi_{11} \| \\
\lesssim & (1 + \eps^{\frac{N_2}{2}-M_1} \| U, V \|_{\mathcal{X}}) (C_\delta \| U, V \|_{\mathcal{X}_{\le 10.5}} + \delta \| U, V \|_{\mathcal{X}}),
\end{align}
where we have invoked \eqref{pw:v:1}, \eqref{L2:uv:eq}. 

The same estimate will apply even when $j \ge 7$, for the $\bar{v}$ contribution from $\nu_s$ for this term. It remains to treat the nonlinear contribution when $7 \le j \le 10$, for which we estimate via 
\begin{align}\n
 \eps^{\frac{N_2}{2}}\| \sqrt{\eps} \p_x^j v \p_x^{12-j} u \langle x \rangle^{11.5} \phi_{11} \| \lesssim & \eps^{\frac{N_2}{2}}  \| \p_x^{12-j} u \langle x \rangle^{12-j }  \psi_{12} \|_\infty \| \sqrt{\eps} \p_x^{j-1}v_x \langle  x\rangle^{j-1 + \frac 1 2} \phi_{11}\| \\
 \lesssim & \eps^{\frac{N_2}{2}-M_1} \| U, V \|_{\mathcal{X}} \| U, V \|_{\mathcal{X}_{\le 10}}, 
\end{align}
where we have invoked \eqref{pw:dec:u} and \eqref{L2:uv:eq}. The identical estimate applies also to the fourth term from \eqref{def:R21:def}. This concludes the proof.  
\end{proof}

\begin{lemma} \label{lemma:top} Let $[\tilde{U}, \tilde{V}]$ satisfy \eqref{eq:11:first} - \eqref{eq:11:second}, and suppose that $\| U, V \|_{\mathcal{X}} \le 1$.
\begin{align} \label{energy:theta11}
\| \tilde{U}, \tilde{V} \|_{\Theta_{11}}^2 \lesssim \sum_{k = 0}^{10} \| U, V \|_{X_k}^2 + \| U, V \|_{X_{k + \frac 1 2} \cap Y_{k + \frac 1 2}}^2+ \mathcal{F}_{X_{11}} + \eps^{\frac{N_2}{2}-M_1-5} \|U, V \|_{\mathcal{X}}^2, 
\end{align}
where we define $\mathcal{F}_{X_{11}}$ to contain the forcing terms from this estimate, 
\begin{align} \label{def:FX11}
\mathcal{F}_{X_{11}} := \int \p_x^{11} F_R \tilde{U} x^{22} \phi_{11}^2 +  \int  \p_x^{11}  G_R \Big( \eps \tilde{V} x^{22} \phi_{11}^2- 22 \eps Q x^{21} \phi_{11}^2 - 2 \eps Q x^{22} \phi_{11} \phi'_{11} \Big)x^{22} \phi_{11}^2. 
\end{align}
\end{lemma}
\begin{proof} We apply the multiplier 
\begin{align}
\int  \eqref{eq:11:first} \times \tilde{U} x^{22} \phi_{11}^2  +  \int  \eqref{eq:11:second} \times (\eps \tilde{V} x^{22} \phi_{11}^2- 22 \eps Q x^{21} \phi_{11}^2 - 2 \eps Q x^{22} \phi_{11} \phi'_{11}).
\end{align}
We note that the multiplier above is divergence free and moreover that $\tilde{V}|_{y = 0} = Q|_{y = 0} = 0$, and hence the pressure contribution will vanish. 

We compute the first two terms from \eqref{eq:11:first}, which yields
\begin{align} \n
|\int (\mu_s^2 \tilde{U}_x + \mu_s \mu_{sx} \tilde{U}) \tilde{U} x^{22} \phi_{11}^2| =& |- 11 \int \mu_s^2 \tilde{U}^2 x^{21}\phi_{11}^2 - \int \mu_s^2 \tilde{U}^2 x^{22} \phi_{11} \phi_{11}'| \\
 \lesssim & \| \mu_s \tilde{U} x^{10.5} \|^2 \lesssim \| U, V \|_{\mathcal{X}_{\le 10.5}}^2 + \eps^{\frac{N_2}{2}-M_1 - 5} \|U, V \|_{\mathcal{X}}^2, 
\end{align}
upon invoking \eqref{same:same}. 

We now compute the $Q$ terms from \eqref{eq:11:first}. The first we split based on the definition of $\mu_s$ 
\begin{align} \n
&\Big| \int \mu_s \mu_{sxy}  Q \tilde{U} x^{22} \phi_{11}^2 \Big| =  \Big| \int \mu_s (\bar{u}_{xy} + \eps^{\frac{N_2}{2}} u_{xy}) Q \tilde{U} x^{22} \phi_{11}^2 \Big| \\ \n
\lesssim& \| \bar{u}_{xy} xy \|_\infty \| \frac{Q}{y} x^{10.5} \phi_{11} \| \| \mu_s \tilde{U} x^{10.5} \phi_{11} \| + \eps^{\frac{N_2}{2}} \| u_{xy} x^{\frac 3 2} \psi_{12} \|_{L^\infty_x L^2_y} \| Q x^{10} \phi_{10} \|_{L^2_x L^\infty_y} \| \mu_s \tilde{U} x^{10.5}  \phi_{11} \| \\ \n
\lesssim& \| \bar{u}_{xy} xy \|_\infty \| \tilde{U} x^{10.5} \phi_{11} \| \| \mu_s \tilde{U} x^{10.5} \phi_{11} \| + \eps^{\frac{N_2}{2}} \| u_{xy} x^{\frac 3 2} \psi_{12} \|_{L^\infty_x L^2_y} \| Q x^{10} \phi_{10} \|_{L^2_x L^\infty_y} \| \mu_s \tilde{U} x^{10.5} \phi_{11} \| \\
\lesssim & (C_\delta\| U, V \|_{\mathcal{X}_{\le 10.5}} + \delta \| \tilde{U}, \tilde{V} \|_{\Theta_{11}}  ) \| U, V \|_{\mathcal{X}_{\le 10.5}} + \eps^{\frac{N_2}{2} - M_1} \| U, V \|_{\mathcal{X}}^3, 
\end{align}
where we have invoked estimate \eqref{prof:u:est} for $\bar{u}$, \eqref{same:same:high}, \eqref{same:same}, \eqref{same:Q:2}, as well as the embedding \eqref{mixed:emb}. Note that we have used that $\phi_{11}^2 = \psi_{12} \phi_{11}^2$, according to the definition \eqref{psi:twelve:def}. The remaining $Q$ term works in an identical manner. 

We now address the terms in $\mathcal{R}_1^{(1)}$, which are defined in \eqref{R11}. For this, we invoke \eqref{R11:estimate} as well as \eqref{same:same:high} to estimate 
\begin{align} \n
\Big| \int \mathcal{R}_1^{(1)} \tilde{U} \langle x \rangle^{22} \phi_{11}^2 \Big| \lesssim & \| \mathcal{R}_1^{(1)} \langle x \rangle^{11.5} \phi_{11} \| \| \tilde{U} \langle x \rangle^{10.5} \phi_{11} \| \lesssim \| U, V \|_{\mathcal{X}} (\delta \| \tilde{U}, \tilde{V} \|_{\Theta_{11}} + C_\delta \| U, V \|_{\mathcal{X}_{\le 10.5}}),
\end{align}
where we have invoked estimates \eqref{same:same:high} and \eqref{R11:estimate}. 

We now address the terms in $\mathcal{R}_1^{(2)}$, which are defined in \eqref{R12}. We estimate these terms easily via 
\begin{align} \n
|\int \mathcal{R}_1^{(2)} \tilde{U} x^{22} \phi_{11}^2| \le & |\int \mu_{sx} u^{(11)} \tilde{U} x^{22} \phi_{11}^2| +   |\int \nu_{s} u^{(11)}_y \tilde{U} x^{22} \phi_{11}^2| \\ \n
\lesssim & \| \mu_{sx} x \psi_{12} \|_\infty \| u^{(11)} x^{10.5} \phi_{11} \| \| \tilde{U} x^{10.5} \phi_{11} \| + \| \frac{\nu_s}{\bar{u}} x^{\frac 1 2} \psi_{12} \|_\infty \| u^{(11)}_y x^{11} \phi_{11} \| \| \bar{u} \tilde{U} x^{10.5} \phi_{11} \| \\ \n 
\lesssim & (1 + \eps^{\frac{N_2}{2} - M_1} \| U, V \|_{\mathcal{X}_{\le 10}} ) \| U, V \|_{\mathcal{X}_{\le 10}} (C_\delta \| U, V \|_{\mathcal{X}_{\le 10}} + \delta \| \tilde{U}, \tilde{V} \|_{\Theta_{11}} ) \\
& + (1 + \eps^{\frac{N_2}{2} - M_1} \| U, V \|_{\mathcal{X}_{\le 10.5}} ) \| U, V \|_{X_{11}} \| U, V \|_{\mathcal{X}_{\le 10}}, 
\end{align}
where we have invoked the estimate \eqref{prof:u:est}, \eqref{prof:v:est}, \eqref{L2:uv:eq},  \eqref{same:same:high}, \eqref{pw:v:1}, and \eqref{same:same}, and again the identity $\phi_{11}^2 = \psi_{12} \phi_{11}^2$.

We now move to $\mathcal{R}_1^{(3)}$, defined in \eqref{R13}, which we estimate the first two terms via 
\begin{align} \n
\Big| \int ( u \p_x^{12} \bar{u} + u_y \p_x^{11} \bar{v} ) \tilde{U} x^{22} \phi_{11}^2 \Big| \lesssim & \| \p_x^{12} \bar{u} x^{11.5} y \|_\infty \| \frac{u}{y} \| \| \tilde{U} x^{10.5} \| + \| \p_x^{11} \bar{v} x^{11.5} \|_\infty \| u_y \| \| \tilde{U} x^{10.5} \| \\
\lesssim & \| u_y \| \| \tilde{U} x^{10.5}\| \le C_\delta \| U, V \|_{\mathcal{X}_{\le 10.5}} + \delta \| U, V \|_{\Theta_{11}}^2, 
\end{align}
where we have invoked \eqref{prof:u:est}, \eqref{L2:uv:eq:2} and \eqref{same:same:high}. The final two terms of $\mathcal{R}_1^{(3)}$ are estimated via 
\begin{align} \n
\Big| \int ( v \p_y  \p_x^{11} \bar{u} + u_x \p_x^{11} \bar{u} ) \tilde{U} x^{22} \phi_{11}^2 \Big| \lesssim &( \| \p_y \p_x^{11} \bar{u} y x^{11} \|_\infty \| \frac{v}{y} x^{\frac 1 2} \| + \| \p_x^{11} \bar{u} x^{11} \|_\infty \| u_x x^{\frac 1 2} \| ) \| \tilde{U} x^{10.5} \phi_{11} \| \\
\lesssim & \| v_y x^{\frac 1 2} \phi_{11} \| \| \tilde{U} x^{10.5} \phi_{11} \| \le C_\delta \| U, V \|_{\mathcal{X}_{\le 10.5}} + \delta \| \tilde{U}, \tilde{V}\|_{\Theta_{11}}^2, 
\end{align}
where we have invoked \eqref{prof:u:est}, \eqref{L2:uv:eq:2}, and \eqref{same:same:high}. 

We now move to the diffusive terms, starting with the $- u^{(11)}_{yy}$ term, for which one integration by parts yields 
\begin{align}  \label{diffuse:1}
- \int u^{(11)}_{yy} \tilde{U} x^{22} \phi_{11}^2 = & \int u^{(11)}_y \tilde{U}_y x^{22} \phi_{11}^2+ \int_{y = 0} u^{(11)}_y \tilde{U} x^{22} \phi_{11}^2 \ud x. 
\end{align}
We now use \eqref{move:2} to expand the first term on the right-hand side of \eqref{diffuse:1}, via 
\begin{align} \n
 \int u^{(11)}_y \tilde{U}_y x^{22} \phi_{11}^2 = & \int ( \mu_s \tilde{U}_y + 2 \mu_{sy} \tilde{U} + \mu_{syy} Q  ) \tilde{U}_y x^{22} \phi_{11}^2 \\
 = & \int \mu_s \tilde{U}_y^2 x^{22} \phi_{11}^2 - \int \mu_{syy} \tilde{U}^2 x^{22}\phi_{11}^2 - \int_{y =0} \mu_{sy} \tilde{U}^2 x^{22} \phi_{11}^2+ \int \mu_{syy} Q \tilde{U}_y x^{22} \phi_{11}^2. 
\end{align}
We also expand the second term on the right-hand side of \eqref{diffuse:1}, again by using \eqref{move:2}, which gives 
\begin{align} 
 \int_{y = 0} u^{(11)}_y \tilde{U} x^{22} \phi_{11}^2 \ud x  = \int_{y = 0} ( \mu_s \tilde{U}_y + 2 \mu_{sy} \tilde{U} + \mu_{syy} Q  ) \tilde{U} x^{22} \phi_{11}^2 \ud x = 2 \int_{y = 0} \mu_{sy} \tilde{U}^2 x^{22}\phi_{11}^2 \ud x. 
\end{align} 
Hence, we obtain 
\begin{align} \n
- \int u_{yy}^{(11)} \tilde{U} x^{22} \phi_{11}^2= &\int \mu_s \tilde{U}_y^2 x^{22} \phi_{11}^2+ \int_{y = 0} \mu_{sy} \tilde{U}^2 x^{22} \phi_{11}^2 \ud x - \int \mu_{syy} \tilde{U}^2 x^{22} \phi_{11}^2\\ \label{ay1}
&+ \int \mu_{syy} Q \tilde{U}_y x^{22} \phi_{11}^2. 
\end{align}
The first two terms from \eqref{ay1} are positive contributions towards the $\Theta_{11}$ norm, whereas the third and fourth terms need to be estimated. We first estimate the third term from \eqref{ay1} via 
\begin{align} \n
\Big| \int \mu_{syy} \tilde{U}^2 x^{22} \phi_{11}^2\Big| \lesssim  &\| \bar{u}_{yy} x \|_\infty \| \tilde{U} x^{10.5}\phi_{11}\|^2 + \eps^{\frac{N_2}{2}} \| u_{yy} x \psi_{12} \|_{L^\infty_x L^2_y} \| \tilde{U} x^{10.5} \phi_{11}\|_{L^2_x L^4_y}^2   \\ \n
\lesssim & \delta \| \tilde{U}, \tilde{V} \|_{\Theta_{11}}^2 + C_\delta \| U, V \|_{\mathcal{X}_{\le 10.5}}^2 + \eps^{\frac{N_2}{2} - M_1 } \| U, V \|_{\mathcal{X}} ( \| \bar{u} \tilde{U} \langle x \rangle^{10.5} \phi_{11} \|^2 \\ \label{keep:damp:1}
&+ \| \sqrt{\bar{u}} \tilde{U}_y \langle x \rangle^{11} \phi_{11} \|^2 ),
\end{align}
where we have invoked \eqref{same:same:high}, as well as \eqref{mixed:emb} and \eqref{weds:weds:2}. 

We now address the fourth term from \eqref{ay1}, for which we split the coefficient $\mu_s$. In the case of $\bar{u}_{yy}$, we may integrate by parts in $y$ to obtain 
\begin{align} \n
\int \bar{u}_{yy} Q \tilde{U}_y x^{22} \phi_{11}^2 = & - \int \bar{u}_{yy} \tilde{U}^2 x^{22} \phi_{11}^2 + \frac 1 2 \int \bar{u}_{yyyy} Q^2 x^{22} \phi_{11}^2,
\end{align}
both of which are estimated in an identical manner to \eqref{keep:damp:1}. In the case of $u_{yy}$, we split into the regions where $z \le 1$ and $z \ge 1$. First, the localized contribution is estimated via 
\begin{align} \n
\eps^{\frac{N_2}{2}} | \int u_{yy} Q \tilde{U}_y x^{22} \phi_{11}^2 \chi(z) | \lesssim & \eps^{\frac{N_2}{2}} \| u_{yy} x \psi_{12}  \|_{L^\infty_x L^2_y} \| \frac{Q}{\sqrt{y}} x^{10.25} \chi(z) \phi_{11}\|_{L^2_x L^\infty_y} \| \sqrt{\bar{u}} \tilde{U}_y x^{11} \| \\
\lesssim & \eps^{\frac{N_2}{2}-M_1} \|U, V \|_{\mathcal{X}}^3,
\end{align}
where we have invoked \eqref{mixed:emb}, as well as the inequality 
\begin{align} \n
|Q| \chi(z) = \chi(z) |\int_0^y \tilde{U} \ud y'| = \chi(z)| \int_0^y \tilde{U} (y')^{\frac 1 2} (y')^{-\frac 1 2} \ud y'| \lesssim \chi(z) y^{\frac 1 2} x^{\frac 1 4} \| \sqrt{\bar{u}} \tilde{U} \|_{L^\infty_y},
\end{align}
from which we obtain $\|\frac{Q}{\sqrt{y}} \chi(z) \langle x \rangle^{10.5}  \phi_{11}\|_{L^2_x L^\infty_y} \lesssim \|U, V \|_{\mathcal{X}}$, after using the interpolation inequality 
\begin{align}
\| \bar{u}^{\frac{1}{2}} \tilde{U} x^{10.75} \phi_{11} \|_{L^2_x L^\infty_y} \lesssim \|\tilde{U} x^{10.5} \phi_{11} \| ^{\frac 1 2} \| \bar{u} \tilde{U}_y x^{11} \phi_{11} \|_{L^2_y}^{\frac 1 2} \lesssim \|U, V \|_{\mathcal{X}}. 
\end{align}
For the far-field contribution, we estimate via 
\begin{align}
\eps^{\frac{N_2}{2}} | \int u_{yy} Q \tilde{U}_y x^{22} \phi_{11}^2 (1- \chi(z)) | \lesssim & \eps^{\frac{N_2}{2}} \| u_{yy} x \psi_{12} \|_{L^\infty_x L^2_y} \| Q x^{10} \phi_{11} \|_{L^2_x L^\infty_y} \| \mu_s \tilde{U}_y x^{11} \phi_{11} \| \\
\lesssim & \eps^{\frac{N_2}{2}- M_1- \frac 1 2} \|U, V \|_{\mathcal{X}}^3, 
\end{align}
where we have invoked the mixed-norm estimates \eqref{mixed:emb}, \eqref{same:Q:2}.

We now address the $- \eps u^{(11)}_{xx}$ terms from \eqref{eq:11:first}. This produces 
\begin{align} \n
&- \int \eps u^{(11)}_{xx}  \tilde{U} x^{22} \phi_{11}^2 - \int \eps v^{(11)}_{yy} \eps \tilde{V} x^{22} \phi_{11}^2  + 22 \int \eps v^{(11)}_{yy} Q x^{21} + 2 \int \eps v^{(11)}_{yy} Q x^{22} \phi_{11} \phi_{11}' \\ \label{Mon:1}
= &2 \int \eps u^{(11)}_x \tilde{U}_x x^{22} \phi_{11}^2 + 44 \int \eps u^{(11)}_x \tilde{U} x^{21} \phi_{11}^2 + 4 \int \eps u^{(11)}_x \tilde{U} x^{22} \phi_{11} \phi_{11}'.
\end{align} 
We first estimate easily the second and third terms from \eqref{Mon:1}. First, 
\begin{align}
|\int \eps u^{(11)}_x \tilde{U} x^{21} \phi_{11}^2| \lesssim & \sqrt{\eps} \| \sqrt{\eps} u^{(11)}_x x^{11} \phi_{11} \| \| \tilde{U} x^{10} \phi_{11} \| \lesssim \sqrt{\eps} \|U, V \|_{\mathcal{X}_{11}}^2, \\ \label{fir}
|\int \eps u^{(11)}_x \tilde{U} x^{22} \phi_{11} \phi_{11}'| \lesssim & \| \sqrt{\eps} u_x^{(11)} x^{11} \phi_{11} \| \| \tilde{U} x^{10} \phi_{10} \| \le \delta \|U, V \|_{\mathcal{X}_{11}} + C_\delta \|U, V \|_{\mathcal{X}_{\le 10}}, 
\end{align}
where we have invoked \eqref{L2:uv:eq:2} and \eqref{same:same:high}, and for estimate \eqref{fir}, we use that $\phi_{10} = 1$ on the support of $\phi_{11}$. 

We now treat the primary term, which is the first term from \eqref{Mon:1}, using the formula \eqref{move:2}, which gives 
\begin{align} \n
2 \int \eps u^{(11)}_x \tilde{U}_x x^{22} \phi_{11}^2 = & 2 \int \eps \p_x (\mu_s \tilde{U} + \p_y \mu_s Q) \tilde{U}_x x^{22} \phi_{11}^2 \\ \n
= & \int 2 \eps (\mu_s \tilde{U}_x + \p_x \mu_s \tilde{U} + \p_{xy} \mu_s Q - \p_y \mu_s \tilde{V}) \tilde{U}_x x^{22} \phi_{11}^2 \\ \n 
= & \int 2 \eps \mu_s \tilde{U}_x^2 x^{22} \phi_{11}^2 - \int \eps \p_{xx} \mu_s \tilde{U}^2 x^{22} \phi_{11}^2 - \int 22 \eps \p_x \mu_s \tilde{U}^2 x^{21} \phi_{11}^2 \\ \n
&- \int 2 \eps \p_x \mu_s \tilde{U}^2 x^{22} \phi_{11} \phi_{11}'  + \int 2 \eps \p_y^2 \p_x \mu_s Q \tilde{V} x^{22} \phi_{11}^2 + \int 2 \eps \p_{xy} \mu_s \tilde{U} \tilde{V} x^{22} \phi_{11} \\ \label{lauv:1}
&  - \int \eps \p_y^2 \mu_s \tilde{V}^2 x^{22} \phi_{11}^2. 
\end{align} 
The first term in \eqref{lauv:1} is a positive contribution. For the second and third terms, we estimate via 
\begin{align}
|\int \eps \p_{xx} \mu_s \tilde{U}^2 x^{22} \phi_{11}^2| + |\int 22 \eps \p_x \mu_s \tilde{U}^2 x^{21} \phi_{11}^2| \lesssim \eps ( \| \frac{ \p_{xx} \mu_s}{\mu_s} x  \psi_{12}\|_\infty + \| \p_x \mu_s \psi_{12}\|_\infty) \| \sqrt{\mu_s} \tilde{U} x^{10.5} \|^2,  
\end{align}
whereas for the fifth term, it is advantageous for us to split up the coefficient via 
\begin{align} \label{best:1}
\int 2 \eps \p_y^2 \p_x \mu_s Q \tilde{V} x^{22} \phi_{11}^2 =  \int 2 \eps \p_y^2 \p_x \bar{u} Q \tilde{V} x^{22} \phi_{11}^2 + \eps^{\frac{N_2}{2} + 1} \int u_{xyy} Q \tilde{V} x^{22} \phi_{11}^2,
\end{align}
after which we estimate the first term from \eqref{best:1} via 
\begin{align}
|\int 2 \eps \bar{u}_{xyy} Q \tilde{V} x^{22} \phi_{11}^2| \lesssim \sqrt{\eps} \| \bar{u}_{xyy} y x^{\frac 3 2} \|_\infty \| \frac{Q}{y} x^{10.5} \| \| \sqrt{\eps} \tilde{V} x^{10.5} \|,
\end{align}
and for the second term from \eqref{best:1}, we obtain 
\begin{align}
\eps^{\frac{N_2}{2}} |\int u_{xyy} Q \tilde{V} x^{22} \phi_{11}^2| \lesssim \eps^{\frac{N_2}{2}} \| u_{xyy} x^2 \psi_{12} \|_{L^\infty_x L^2_y} \| Q x^{10} \phi_{11} \|_{L^2_x L^\infty_y} \| \tilde{V} x^{10.5} \phi_{11} \| 
\end{align}

We now estimate the sixth term from \eqref{lauv:1} via 
\begin{align} \n
| \int 2 \eps \p_{xy} \mu_s \tilde{U} \tilde{V} x^{22} \phi_{11}| \lesssim \| \p_{xy} \mu_s \langle x \rangle^{\frac 3 2} \psi_{12}\|_\infty \| \tilde{U} \langle x \rangle^{10.5} \phi_{11} \| \| \sqrt{\eps} \tilde{V} \langle x \rangle^{10.5} \phi_{11} \|
\end{align}

The seventh term from \eqref{lauv:1} is fairly tricky. First, the contribution arising from the $\bar{u}_{yy}$ component of $\p_y^2 \mu_s$ is straightforward, and we estimate it via 
\begin{align}
|\int \eps \p_y^2 \bar{u} \tilde{V}^2 x^{22} \phi_{11}^2 | \lesssim \| \bar{u}_{yy} \langle x \rangle \|_\infty \| \sqrt{\eps} \tilde{V} \langle x \rangle^{10.5} \phi_{11} \|^2, 
\end{align}
which is an admissible contribution according to \eqref{same:same:high}. 

To handle the $\eps^{\frac{N_2}{2}}u_{yy}$ contribution from $\p_y^2 \mu_s$, we first localize in $z$. The far-field contribution is handled via
\begin{align} \n
|\eps^{\frac{N_2}{2}} \int \eps u_{yy} \tilde{V}^2 x^{22} \phi_{11}^2 (1- \chi(z))| \lesssim & \eps^{\frac{N_2}{2}+1} \| u_{yy} \langle x \rangle \psi_{12} \|_{L^\infty_x L^2_y} \| \tilde{V} \langle x \rangle^{10.5} \phi_{11} (1- \chi(z)) \|_{L^2_x L^\infty_y} \\ \n
& \times \| \tilde{V} \langle x \rangle^{10.5} \phi_{11} (1- \chi(z)) \|\\ \n
\lesssim & \eps^{\frac{N_2}{2}+1 - M_1} \| U, V \|_{\mathcal{X}} \| \bar{u} \tilde{V} \langle x \rangle^{10.5} \phi_{11}\|_{L^2_x L^\infty_y} \| \bar{u} \tilde{V} \langle x \rangle^{10.5} \phi_{11}\| \\ \n
\lesssim & \eps^{\frac{N_2}{2}+1 - M_1} \| U, V \|_{\mathcal{X}} \| \bar{u} \tilde{V}_y \langle x \rangle^{10.5} \phi_{11}\|^{\frac 1 2} \| \bar{u} \tilde{V} \langle x \rangle^{10.5} \phi_{11}\|^{\frac 3 2},
\end{align}
where we have used the presence of $(1 - \chi(z))$ to insert factors of $\bar{u}$ above, as well as estimate \eqref{mixed:emb}.

To handle this same contribution for $z \le 1$, we use Holder's inequality in the following manner 
\begin{align*}
|\eps^{\frac{N_2}{2}} \int \eps u_{yy}  \tilde{V}^2 x^{22} \phi_{11}^2 \chi(z)| \lesssim &\eps^{\frac{N_2}{2}+1} \| u_{yy} \langle x \rangle \psi_{12} \|_{L^\infty_x L^2_y} \| \tilde{V} \langle x \rangle^{10.5} \chi(z) \phi_{11}\|_{L^2_x L^4_y}^2
\end{align*} 
from which the result follows from an application of \eqref{mixed:emb} and \eqref{weds:weds:2}.

We now move to the final diffusive term, which contributes the following 
\begin{align} \n
&- \int \eps^2 v^{(11)}_{xx} (\tilde{V} x^{22} \phi_{11}^2 - 22 Q x^{21} \phi_{11}^2 - 2 Q x^{22} \phi_{11} \phi_{11}' ) \\ \n
= & \int \eps^2 v^{(11)}_x (\tilde{V} x^{22} \phi_{11}^2)_x - 22 \int \eps^2 v^{(11)}_x (Q x^{21} \phi_{11}^2)_x - 2 \int \eps^2 v^{(11)}_x ( Q x^{22} \phi_{11} \phi_{11}')_x \\ \n
= &  \int \eps^2 v^{(11)}_x \tilde{V}_x x^{22} \phi_{11}^2 + 44 \int \eps^2 v^{(11)}_x \tilde{V} x^{21} \phi_{11}^2 - 462 \int \eps^2 v^{(11)}_x Q x^{20} \phi_{11}^2 \\  \label{bkl}
& + 2 \int \eps^2 v^{(11)}_x \tilde{V} x^{22} \phi_{11} \phi_{11}' - 44 \int \eps^2 v^{(11)}_x Q x^{21} \phi_{11} \phi_{11}'- 2 \int \eps^2 v^{(11)}_x ( Q x^{22} \phi_{11} \phi_{11}')_x
\end{align}
We will now analyze the first term from \eqref{bkl}, which gives upon appealing to \eqref{move:2}, 
\begin{align} \n
\int \eps^2 v^{(11)}_x \tilde{V}_x x^{22} \phi_{11}^2  = & \int \eps^2 \p_x (\mu_s \tilde{V} - \p_x \mu_s Q) \tilde{V}_x x^{22} \phi_{11}^2 \\ \label{load:1}
= & \int \eps^2 \mu_s \tilde{V}_x^2 x^{22} \phi_{11}^2 + 2 \int \eps^2 \p_x \mu_s \tilde{V} \tilde{V}_x x^{22} \phi_{11}^2 - \int \eps^2 \p_{xx} \mu_s Q \tilde{V}_x x^{22} \phi_{11}^2 
\end{align}
The first term above in \eqref{load:1} is a positive contribution, whereas the second two can easily be estimated via  
\begin{align} \n
&|2 \int \eps^2 \p_x \mu_s \tilde{V} \tilde{V}_x x^{22} \phi_{11}^2 - \int \eps^2 \p_{xx} \mu_s Q \tilde{V}_x x^{22} \phi_{11}^2 | \\ \n
\lesssim & \sqrt{\eps} \| \p_x \mu_s x \psi_{12} \|_\infty \| \sqrt{\eps} \tilde{V} x^{10.5} \phi_{11} \| \| \eps \sqrt{\mu_s} \tilde{V}_x x^{11} \phi_{11} \| \\
&+ \sqrt{\eps} \| \p_{xx} \mu_s x^2 \psi_{12} \|_\infty \| \sqrt{\eps} Q x^{9.5} \phi_{11}\| \| \eps \sqrt{\mu_s} \tilde{V}_x x^{11} \phi_{11} \|. 
\end{align}
We now estimate the remaining terms in \eqref{bkl}. First, the contributions from $\phi_{11}'$ are supported for finite $x$, and can thus be estimated by lower order norms. The second term from \eqref{bkl} is bounded by 
\begin{align} \n
|\int \eps^2 v_x^{(11)} \tilde{V} x^{21} \phi_{11}^2 | \lesssim \sqrt{\eps} \| \eps v_x^{(11)} \phi_{11} x^{11} \| \| \sqrt{\eps} \tilde{V} \phi_{11} x^{10.5} \|, 
\end{align}
and the third term from \eqref{bkl} by
\begin{align} \n
| \int \eps^2 v^{(11)}_x Q x^{20} \phi_{11}^2 | \lesssim \sqrt{\eps} \| \eps v^{(11)}_x \phi_{11} \langle x \rangle^{11} \| \| \sqrt{\eps}Q \phi_{10} x^{9} \|,
\end{align}
both of which are acceptable contributions according to estimates \eqref{L2:uv:eq:2}, \eqref{same:Q}, and \eqref{same:same:high}. 

We now arrive at the remaining terms from \eqref{eq:11:second}, which will all be treated as error terms. We first record the identity 
\begin{align} \n
&\int \mu_s^2 \tilde{V}_x (\eps \tilde{V} x^{22} \phi_{11}^2 - 2 \eps Q x^{21} \phi_{11}^2 - 2 \eps Q x^{22} \phi_{11} \phi_{11}' ) \\ \n
= & - \eps \int \mu_s \p_x \mu_s \tilde{V}^2 x^{22} \phi_{11}^2 - 33 \int \eps \mu_s^2 \tilde{V}^2 x^{21} \phi_{11}^2 - \int \eps \mu_s^2 \tilde{V}^2 x^{22} \phi_{11} \phi_{11}' \\ \n
& - 44 \int \eps \tilde{V} Q \mu_s \p_x \mu_s x^{21} \phi_{11}^2 + 462 \int \eps \mu_s^2 \tilde{V} Q x^{20} \phi_{11}^2 + 44 \int \eps \mu_s^2 \tilde{V} Q x^{21} \phi_{11} \phi_{11}' \\ \label{bss}
& + 2 \int \eps \tilde{V} \p_x (\mu_s^2 Q x^{22} \phi_{11} \phi_{11}').  
\end{align}
To estimate these, we simply note that due to the pointwise estimates $|\p_x \mu_s \langle x \rangle \psi_{12} | \lesssim \bar{u}$, we have 
\begin{align}
\eps |\int \mu_s \p_x \mu_s \tilde{V}^2 x^{22} \phi_{11}^2| \lesssim \| \sqrt{\eps} \bar{u} \tilde{V} x^{10.5}  \phi_{11} \|^2, 
\end{align} 
and similarly for the second term from \eqref{bss}. An analogous estimate applies to the fourth and fifth terms from \eqref{bss}. 

We next move to 
\begin{align} \n
&\int \mu_s \nu_s \tilde{V}_y (\eps \tilde{V} x^{22} \phi_{11}^2 - 2 \eps Q x^{21} \phi_{11}^2 - 2 \eps Q x^{22} \phi_{11} \phi_{11}' ) \\ \n
= & - \frac 1 2 \int \eps (\mu_s \nu_s)_y \tilde{V}^2 x^{22} \phi_{11}^2 + 2 \int \eps \mu_s \nu_s \tilde{V} \tilde{U} x^{21} \phi_{11}^2 + \int \eps (\mu_s \nu_s)_y \tilde{V} Q x^{21} \phi_{11}^2 \\
& + 2 \int \eps \mu_s \nu_s \tilde{V} \tilde{U} x^{22} \phi_{11} \phi_{11}' + 2 \int \eps (\mu_s \nu_s)_y \tilde{V} Q x^{22} \phi_{11} \phi_{11}'.
\end{align}
To estimate these terms, we proceed via 
\begin{align} \n
&|\int \frac 1 2 \eps (\mu_s \nu_s)_y \tilde{V}^2 x^{22} \phi_{11}^2| + 2| \int \eps \mu_s \nu_s \tilde{V} \tilde{U} x^{21} \phi_{11}^2| + | \int \eps (\mu_s \nu_s)_y \tilde{V} Q x^{21} \phi_{11}^2| \\ \n
 \lesssim & \| \p_y (\mu_s \nu_s) x \psi_{12}\|_\infty \| \sqrt{\eps} \tilde{V} x^{10.5} \phi_{11} \|^2 + \sqrt{\eps}\| \nu_s x^{\frac 1 2} \psi_{12} \|_\infty \| \tilde{U} x^{10.5}  \phi_{11}\| \| \tilde{V} x^{10.5} \phi_{11} \|  \\ \label{analogue:1}
 & + \| (\mu_s \nu_s)_y x \psi_{12} \|_\infty \| \sqrt{\eps} \tilde{V} x^{10.5} \phi_{11} \| \sqrt{\eps} Q x^{9.5} \phi_{11} \|, 
\end{align}
all of which are acceptable contributions according to the pointwise decay estimates \eqref{pw:dec:u} - \eqref{Linfty:wo}, and according to \eqref{same:Q} - \eqref{same:same:high}. 

We now arrive at the three error terms from \eqref{eq:11:second} which are of the form $(2 \mu_s \mu_{sx} + \nu_s \mu_{sy}) \tilde{V} - \mu_{sx} \nu_s \tilde{U} - (\mu_s \mu_{sxx} + \nu_s \mu_{sxy})Q$. To estimate these contributions it suffices to note that the coefficient in front of $\tilde{V}$ satisfies the estimate $|2 \mu_s \mu_{sx} + \nu_s \mu_{sy}| \lesssim \langle x \rangle^{-1}$, and similarly the coefficient in front of $\tilde{U}$ satisfies $|\mu_{sx}\nu_s| \lesssim \langle x \rangle^{-1}$. Third, the coefficient in front of $Q$ satisfies $|\mu_s \mu_{sxx} + \nu_s \mu_{sxy}| \lesssim \langle x \rangle^{-2}$. Thus, we may apply an analogous estimate to \eqref{analogue:1}. 

We now estimate the error terms in $\mathcal{R}_2^{(i)}[u, v]$, for $i = 1,2, 3$, beginning first with those of $\mathcal{R}_2^{(1)}$. For this, we invoke \eqref{R11:estimate} as well as \eqref{same:same:high} to estimate 
\begin{align} \n
\Big| \int \eps \mathcal{R}_2^{(1)} \tilde{V} \langle x \rangle^{22} \phi_{11}^2 \Big| \lesssim & \|\sqrt{\eps} \mathcal{R}_2^{(1)} \langle x \rangle^{11.5} \phi_{11} \| \| \sqrt{\eps} \tilde{V} \langle x \rangle^{10.5} \phi_{11} \| \\
\lesssim & (\delta \|U, V \|_{\mathcal{X}} + C_\delta \|U, V \|_{X_{\le 10.5}})(\delta \|U, V \|_{\mathcal{X}} + C_\delta \|U, V \|_{X_{\le 10.5}}). 
\end{align}

We now estimate the contributions from $\mathcal{R}_2^{(2)}[u, v]$, defined in \eqref{def:R22}, for which we first have 
\begin{align} \n
&|\int \nu_{sx} u^{(11)} (\eps \tilde{V} x^{22} \phi_{11}^2 - 2 \eps Q x^{21} \phi_{11}^2 - 2 \eps Q x^{22} \phi_{11} \phi_{11}' )|  \\ \n
\lesssim & \sqrt{\eps}  \| \nu_{sx} x \psi_{12} \|_\infty  \| u^{(11)}  x^{10.5} \phi_{11}\| ( \| \sqrt{\eps} \tilde{V} x^{10.5} \phi_{11} +  \| \sqrt{\eps} Q x^{9.5} \phi_{10} \|) + \sqrt{\eps} \|U, V \|_{\mathcal{X}_{\le 10.5}}  \\
\lesssim & \sqrt{\eps} (1 + \eps^{\frac{N_2}{2}-M_1} \|U, V \|_{\mathcal{X}})  (C_\delta \| U, V \|_{\mathcal{X}_{\le 10.5}} + \delta \| U, V\|_{\mathcal{X}}) (C_\delta \| U, V \|_{\mathcal{X}_{\le 10.5}} + \delta \| U, V\|_{\mathcal{X}}),
\end{align}
where we have invoked \eqref{pw:v:1}, \eqref{L2:uv:eq}, \eqref{same:same:high}. 

and similarly for the second term from $\mathcal{R}_2^{(2)}$, we have 
\begin{align} \n
&|\int \nu_{sy} v^{(11)} (\eps \tilde{V} x^{22} \phi_{11}^2 - 2 \eps Q x^{21} \phi_{11}^2 - 2 \eps Q x^{22} \phi_{11} \phi_{11}' )|  \\ \n
\lesssim &   \| \nu_{sy} x \|_\infty  \|  \sqrt{\eps} v^{(11)}  x^{10.5} \phi_{11}\| ( \| \sqrt{\eps} \tilde{V} x^{10.5} \phi_{11} +  \| \sqrt{\eps} Q x^{9.5} \phi_{10} \|) +  \sqrt{\eps} \|U, V \|_{\mathcal{X}_{\le 10.5}}  \\
\lesssim & (1 + \eps^{\frac{N_2}{2}-M_1} \| U, V \|_{\mathcal{X}} )( C_\delta \|U, V \|_{\mathcal{X}_{\le 10.5}} + \delta \|U, V \|_{\mathcal{X}_{11}}    ) ( C_\delta \|U, V \|_{\mathcal{X}_{\le 10.5}} + \delta \|U, V \|_{\mathcal{X}_{11}}    ), 
\end{align}
where we have used \eqref{pw:dec:u}, \eqref{L2:uv:eq}, and \eqref{same:same:high}. 

We now move to the error terms from $\mathcal{R}_2^{(3)}$, First, we estimate using the definition \eqref{def:R23}, 
\begin{align} \n
\| \sqrt{\eps} \mathcal{R}_2^{(3)} x^{11.5} \phi_{11}\| \lesssim &\sqrt{\eps} \| \p_x^{11} \bar{v}_y y x^{11.5} \|_\infty \| v_y \phi_{11} \| + \| \p_x^{12} \bar{v} x^{12.5} \|_\infty \| u \langle x \rangle^{-1}\phi_{11} \| \\ \n
& + \| \p_x^{11} \bar{u} x^{11} \|_\infty \| \sqrt{\eps} v_x x^{\frac 1 2}\phi_{11} \| + \| \p_x^{11} \bar{v} x^{11.5} \|_\infty \| v_y \phi_{11}\| \\
\lesssim & \| U, V \|_{\mathcal{X}_{\le 4}}. 
\end{align}
From this, we estimate simply 
\begin{align} \n
&|\int \mathcal{R}_2^{(3)} (\eps \tilde{V} x^{22} \phi_{11}^2 - 2 \eps Q x^{21} \phi_{11}^2 - 2 \eps Q x^{22} \phi_{11} \phi_{11}' )| \\ \n
\lesssim & \| \sqrt{\eps} \mathcal{R}_2^{(3)} x^{11.5} \phi_{11} \| ( \| \sqrt{\eps} \tilde{V} x^{10.5} \phi_{11} \| + \| \sqrt{\eps} Q x^{9.5} \phi_{10} \|+   \sqrt{\eps} \|U, V \|_{\mathcal{X}_{\le 10.5}}) \\
\lesssim & \| U, V \|_{\mathcal{X}_{\le 4}} ( \delta \| \tilde{U}, \tilde{V} \|_{\Theta_{11}} + C_\delta \|U, V \|_{\mathcal{X}_{\le 10.5}} ),
\end{align}
where we have invoked estimate \eqref{same:same:high}. This concludes the proof. 
\end{proof}

\section{Nonlinear Analysis} \label{section:NL}

We first obtain estimates on the ``elliptic" component of the $\mathcal{X}$-norm, defined in \eqref{def:E:norm}. For this component of the norm, the mechanism is entirely driven by elliptic regularity. 
\begin{lemma} \label{lemma:elliptic} Let $(u, v)$ solve \eqref{vel:eqn:1} - \eqref{vel:eqn:2}. Then the following estimate is valid 
\begin{align} \label{elliptic:1}
\|U, V \|_E^2 \lesssim &\| U, V \|_{X_0}^2 + \mathcal{F}_{Ell}, 
\end{align}
where we define 
\begin{align}
\mathcal{F}_{Ell} := &  \sum_{k = 1}^{11} \| \p_x^{k-1} F_R \|^2 + \| \sqrt{\eps} \p_x^{k-1} G_R  \|^2,
\end{align}
\end{lemma}
\begin{proof} This is a consequence of standard elliptic regularity. Indeed, rewriting \eqref{vel:eqn:1} - \eqref{vel:eqn:2} as a perturbation of the scaled Stokes operator, we obtain 
\begin{align}
&- \Delta_\eps u + P_x = F_R + \mathcal{N}_1 - (\bar{u} u_x + \bar{u}_y v + \bar{u}_x u + \bar{v} u_y), \\
&- \Delta_\eps v + \frac{P_y}{\eps} = G_R + \mathcal{N}_2 - (\bar{u} v_x + \bar{v}_y v + \bar{v}_x u + \bar{v} v_y), \\
&u_x + v_y = 0, 
\end{align}
with boundary conditions \eqref{vel:eqn:2}. From here, we apply standard $H^2$ estimates for the Stokes operator on the quadrant, \cite{Blum}, and subsequently bootstrap elliptic regularity for the Stokes operator away from $\{x = 0\}$ in the standard manner (see, for instance, \cite{Iyer2a} - \cite{Iyer2c}), which immediately results in \eqref{elliptic:1}. 
\end{proof}

We now analyze the nonlinear terms. Define the total trilinear contribution via 
\begin{align} \label{def:full:T}
\mathcal{T} := \sum_{k = 0}^{10} \Big( \mathcal{T}_{X_k} + \mathcal{T}_{X_{k + \frac 1 2}} + \mathcal{T}_{Y_{k + \frac 1 2}} \Big), 
\end{align}
where the quantities appearing on the right-hand side of \eqref{def:full:T} are defined in \eqref{def:TX0}, \eqref{TX12:spec}, \eqref{TY12spec}, \eqref{def:TXn}, and \eqref{def:TXnp12}. Our main proposition regarding the trilinear terms will be 
\begin{proposition} \label{prop:trilinear} The trilinear quantity $\mathcal{T}$ obeys the following estimate 
\begin{align} \label{est:T} 
| \mathcal{T} | \lesssim \eps^{\frac{N_2}{2} -M_1 -  5} \| U, V \|_{\mathcal{X}}^3.  
\end{align}
\end{proposition}
\begin{proof}[Proof of Proposition] The proposition follows from combining estimate \eqref{TX0:est} and \eqref{Trilin:rest}. 
\end{proof}

\begin{lemma} The quantity $\mathcal{T}_{X_0}$, defined in \eqref{def:TX0}, obeys the following estimate 
\begin{align} \label{TX0:est}
|\mathcal{T}_{X_0}| \lesssim \eps^{\frac{N_2}{2} - M_1 - 5} \| U, V \|_{\mathcal{X}}^3. 
\end{align}
\end{lemma}
\begin{proof} We recall the definition of $\mathcal{T}_{X_0}$ from \eqref{def:TX0}. We first address the terms from $\mathcal{N}_1$, which give 
\begin{align} \n
\int \mathcal{N}_1 U g^2 = & \eps^{\frac{N_2}{2}} \int u u_x U g^2+ \eps^{\frac{N_2}{2}} \int v u_y U g^2\\ \n
= &  \eps^{\frac{N_2}{2}} \int u \p_x (\bar{u} U g^2+ \bar{u}_y q) Ug^2 + \eps^{\frac{N_2}{2}}  \int v \p_y (\bar{u} U + \bar{u}_y q) Ug^2 \\ \n
= & \eps^{\frac{N_2}{2}} \int \bar{u} u U U_x g^2+  \eps^{\frac{N_2}{2}} \int \bar{u}_x u U^2 g^2+  \eps^{\frac{N_2}{2}} \int \bar{u}_{xy} u q Ug^2-  \eps^{\frac{N_2}{2}} \int \bar{u}_y u VUg^2 \\ \label{N1}
& +  \eps^{\frac{N_2}{2}} \int \bar{u} v U_y U g^2+ 2  \eps^{\frac{N_2}{2}} \int \bar{u}_y v U^2 g^2 +  \eps^{\frac{N_2}{2}} \int \bar{u}_{yy} v q U g^2. 
\end{align}
We now proceed to estimate 
\begin{align} \label{walked:1}
\eps^{\frac{N_2}{2}} |\int \bar{u} u U U_x g^2 | \lesssim \eps^{\frac{N_2}{2}} \| u x^{\frac 14} \|_\infty \| U \langle x \rangle^{- \frac 3 4} \| \| U_x \langle x \rangle^{\frac 1 2} \| \lesssim \eps^{\frac{N_2}{2} - M_1} \|U, V \|_{\mathcal{X}}^2 \|\bar{u}U_x \langle x \rangle^{\frac 1 2} \|,  
\end{align}
where we have invoked estimate \eqref{Linfty:wo}. To conclude, we estimate the final term appearing in \eqref{walked:1} by splitting $\| \bar{u}U_x \langle x \rangle^{\frac 1 2} \| \lesssim \| \bar{u}U_x (1 - \phi_{12}) \| + \| \bar{u}U_x \langle x \rangle^{\frac 1 2} \phi_{12}\|$, where we have used that the support of $(1 - \phi_{12})$ is bounded in $x$, and so we can get rid of the weight in $x$ for this term. For the $x$ large piece, we use that $\phi_1 = 1$ in the support of $\phi_{12}$, and so $\| U_x \langle x \rangle^{\frac 1 2} \phi_{12} \| \le \| U_x \langle x \rangle^{\frac 1 2} \phi_{1} \| \lesssim \|U, V \|_{\mathcal{X}}$. For the ``near $x = 0$" case, we simply estimate by using $\| \sqrt{\bar{u}} U_x \| \le \eps^{- \frac 1 2} \|U, V \|_{X_0}$. 

The second and third terms from \eqref{N1} follows in the same manner, via 
\begin{align}
\eps^{\frac{N_2}{2}} | \int \bar{u}_x u U^2 g^2| + |  | \lesssim \eps^{\frac{N_2}{2}} \| \bar{u}_x x \|_\infty \| u x^{\frac 1 4} \|_\infty \| U \langle x \rangle^{- \frac 5 8} \| 
\end{align} 
For the sixth term from \eqref{N1}, we first decompose $\bar{u}$ into its Euler and Prandtl components via 
\begin{align} \label{dec:hard:1}
\eps^{\frac{N_2}{2}} \int \bar{u}_y v U^2 g^2= \eps^{\frac{N_2}{2}} \int \p_y \bar{u}_P v U^2 g^2+ \eps^{\frac{N_2+1 }{2}} \int \p_Y \bar{u}_E v U^2 g^2. 
\end{align}
For the Euler component, we can use the enhanced $x$-decay available from \eqref{est:Eul:piece} to estimate 
\begin{align}
 \eps^{\frac{N_2+1 }{2}} |\int \p_Y \bar{u}_E v U^2 g^2| \lesssim  \eps^{\frac{N_2+1 }{2}} \int \langle x \rangle^{- \frac 3 2} U^2 \lesssim \eps^{\frac{N_2 + 1}{2}} \| U \langle x \rangle^{- \frac 3 4} \|^2. 
\end{align}
For the Prandtl component of \eqref{dec:hard:1}, we do not get strong enough $x$-decay, but rather must rely on self-similarity coupled with the sharp decay of $v$. More specifically, we need to first decompose $U = U(x, 0) + (U - U(x, 0))$, after which we obtain 
\begin{align} \n
&\eps^{\frac{N_2}{2}} | \int \bar{u}_{Py} v U^2 g^2| \le  \eps^{\frac{N_2}{2}} | \int \bar{u}_{Py} v U(x, 0)^2| + \eps^{\frac{N_2}{2}} | \int \bar{u}_{Py} v (U - U(x, 0))^2| \\
\lesssim & \eps^{\frac{N_2}{2}} \sup_x \| \bar{u}_{Py} \|_{L^1_y} \| v \langle x \rangle^{\frac 1 2} \|_\infty \| U(x, 0) \langle x \rangle^{- \frac 1 2} \|_{y = 0}^2 + \eps^{\frac{N_2}{2}} \| \bar{u}_{Py} y^2 x^{- \frac 1 2} \|_\infty \| v \langle x \rangle^{\frac 1 2} \|_\infty \| \frac{U - U(x, 0)}{y} \|^2.
\end{align}

We now address the terms from $\mathcal{N}_2$, which gives 
\begin{align} \n
\int \eps \mathcal{N}_2 ( V g^2 + \frac{1}{100} q \langle x \rangle^{- 1 - \frac{1}{100}} )= & \eps^{\frac{N_2}{2} + 1} \int u v_x V g^2 + \eps^{\frac{N_2}{2} + 1} \int v v_y V g^2 \\
&+ \frac{1}{100} \eps^{\frac{N_2}{2}+ 1} \int u v_x q \langle x \rangle^{- 1 - \frac{1}{100}} +  \frac{1}{100} \eps^{\frac{N_2}{2}+ 1} \int v v_y q \langle x \rangle^{- 1 - \frac{1}{100}}.
\end{align}
We estimate these terms directly via,  
\begin{align}
\eps^{\frac{N_2}{2}+1} |\int u v_x V g^2| \lesssim & \eps^{\frac{N_2}{2}} \| u \langle x \rangle^{\frac 1 4}\|_\infty \| \sqrt{\eps} v_x \langle x \rangle^{\frac 1 2} \| \| \sqrt{\eps} V \langle x \rangle^{- \frac 3 4} \|, \\
\eps^{\frac{N_2}{2}+1} |\int v v_y V g^2| \lesssim & \eps^{\frac{N_2}{2}} \| v \langle x \rangle^{\frac 1 2} \|_\infty \| v_y \langle x \rangle^{\frac 1 2} \| \| \sqrt{\eps} V \langle x \rangle^{-1} \| \\
\eps^{\frac{N_2}{2}+1} |\int u v_x q \langle x \rangle^{-1 - \frac{1}{100}} | \lesssim & \eps^{\frac{N_2}{2}} \| u \langle x \rangle^{\frac 1 4} \|_\infty \| \sqrt{\eps} v_x \langle x \rangle^{\frac 1 2} \| \| \sqrt{\eps} q \langle x \rangle^{- \frac 7 4} \|, \\
 \eps^{\frac{N_2}{2}+ 1} | \int v v_y q \langle x \rangle^{- 1 - \frac{1}{100}}| \lesssim & \eps^{\frac{N_2}{2}} \| v \langle x \rangle^{\frac 1 2} \|_\infty \| v_y \langle x \rangle^{\frac 1 2} \| \| \sqrt{\eps} q \langle x \rangle^{ -2 } \|.  
\end{align}

\end{proof}

We note that in the estimation of the trilinear terms, $\mathcal{T}_{X_{\frac 1 2}}$ and $\mathcal{T}_{Y_{\frac 1 2}}$, we do not need to integrate by parts to find extra structure. In fact, it is a bit more convenient to state a general lemma first, which simplifies the forthcoming estimates.
\begin{lemma} For $0 \le k \le 10$, 
\begin{align} \label{N2:quant}
\| \frac{1}{\bar{u}} \p_x^k \mathcal{N}_1 \langle x \rangle^{k + \frac 1 2} \phi_1 \| + \sqrt{\eps} \| \frac{1}{\bar{u}} \p_x^k \mathcal{N}_2 \langle x \rangle^{k + \frac 1 2} \phi_1 \| \lesssim \eps^{\frac{N_2}{2} - 2M_1} \| U, V \|_{\mathcal{X}}^2. 
\end{align}
\end{lemma}
\begin{proof} First, regarding the cutoff function $\phi_1$ present in \eqref{N2:quant}, we will rewrite it as $\phi_1 = \phi_1  \psi_{12} = \phi_1 (\psi_{12} - \phi_{12} ) + \phi_1 \phi_{12}$, according to the definitions \eqref{def:phi:j} and \eqref{psi:twelve:def}. As a result, we separate the estimation of \eqref{N2:quant} into 
\begin{align} \n
&\| \frac{1}{\bar{u}} \p_x^k \mathcal{N}_1 \langle x \rangle^{k + \frac 1 2} \phi_1 \| + \sqrt{\eps} \| \frac{1}{\bar{u}} \p_x^k \mathcal{N}_2 \langle x \rangle^{k + \frac 1 2} \phi_1 \| \\ \n
\le & \| \frac{1}{\bar{u}} \p_x^k \mathcal{N}_1 \langle x \rangle^{k + \frac 1 2} (\psi_{12} - \phi_{12}) \| + \sqrt{\eps} \| \frac{1}{\bar{u}} \p_x^k \mathcal{N}_2 \langle x \rangle^{k + \frac 1 2} (\psi_{12} - \phi_{12}) \| \\
& +  \| \frac{1}{\bar{u}} \p_x^k \mathcal{N}_1 \langle x \rangle^{k + \frac 1 2}  \phi_{12} \| + \sqrt{\eps} \| \frac{1}{\bar{u}} \p_x^k \mathcal{N}_2 \langle x \rangle^{k + \frac 1 2}  \phi_{12} \|.
\end{align}
The quantities with $\psi_{12} - \phi_{12}$ are supported in a finite region of $x$, and are thus estimated by $\eps^{\frac{N_2}{2} - 2M_1} \|U, V \|_{\mathcal{X}}^2$. We must thus consider the more difficult case of large $x$, in the support of $\phi_{12}$.

We first treat the two terms arising from $\mathcal{N}_1$. Applying the product rule yields 
\begin{align} \label{x:N1:1}
\p_x^k \mathcal{N}_1 = \sum_{j = 0}^k \binom{k}{j} (\p_x^j u \p_x^{k-j+1} u +  \p_x^j v \p_x^{k-j} \p_y u )
\end{align}
Let us first treat the first quantity in the sum. As $0 \le k \le 10$, either $j$ or $k - j + 1$ must be less than $6$. By symmetry of this term, we assume that $j \le 6$ and then $k \le 10$. In this case, note that $k-j+1 \le 11$, and so we estimate via 
\begin{align} 
\eps^{\frac{N_2}{2}}\| \frac{1}{ \bar{u} } \p_x^j u \p_x^{k-j+1} \langle x \rangle^{k+ \frac 1 2} \phi_{12} \| \lesssim & \eps^{\frac{N_2}{2}}\| \frac{1}{\bar{u}} \p_x^j u \langle x \rangle^{j + \frac 1 4} \psi_{12} \|_\infty \| \p_x^{k-j+1} u \langle x \rangle^{k-j+\frac 12} \phi_{12} \| \lesssim  \eps^{\frac{N_2}{2}- M_1} \|U, V \|_{\mathcal{X}}^2,
\end{align} 
where we have invoked \eqref{pw:dec:u} and \eqref{L2:uv:eq}.

We now move to the second term from \eqref{x:N1:1}, which is not symmetric and thus we consider two different cases. First, we assume that $j \le 6$ and $k \le 10$. In this case, we estimate 
\begin{align} \n
\eps^{\frac{N_2}{2}} \| \frac{1}{\bar{u}} \p_x^j v \p_x^{k-j} \p_y u \langle x \rangle^{k + \frac 1 2} \phi_{12}\| \lesssim & \eps^{\frac{N_2}{2}}\| \frac{1}{\bar{u}} \p_x^j v \langle x \rangle^{j + \frac 1 2} \psi_{12}\|_\infty \| \p_x^{k-j} \p_y u \langle x \rangle^{k-j} \phi_{12}\| \\
\lesssim & \eps^{\frac{N_2}{2}-M_1} \|U, V \|_{\mathcal{X}}^2, 
\end{align} 
where we have invoked \eqref{pw:v:1} and \eqref{L2:uv:eq:2}. 

We next consider the case that $0 \le k-j \le 6$ and $6 \le j \le 10$. In this case, we estimate via 
\begin{align} \n
\eps^{\frac{N_2}{2}} \| \frac{1}{\bar{u}}\p_x^j v \p_x^{k-j} \p_y u \langle x \rangle^{k + \frac 1 2} \phi_{12} \| \lesssim &\eps^{\frac{N_2}{2}} \| \p_x^{k-j} \p_y u \langle x \rangle^{k-j + \frac 1 2} \psi_{12}\|_{L^\infty_x L^2_y} \| \frac{1}{\bar{u}} \p_x^j v \langle x\rangle^{j }  \phi_{12} \|_{L^2_x L^\infty_y}  \\
\lesssim & \eps^{\frac{N_2}{2}- M_1} \| U, V \|_{\mathcal{X}}^2
\end{align}
where we have used the mixed-norm estimates in \eqref{pw:v:1} and \eqref{mixed:L2:orig:1} (and crucially that $j \le 10$ to be in the range of admissible exponents for \eqref{mixed:L2:orig:1}). 

We now consider the second quantity in \eqref{N2:quant}, for which we again apply the product rule to obtain 
\begin{align} \label{second:sum}
\p_x^k \mathcal{N}_2 = \sum_{j = 0}^{k} \binom{k}{j} ( \p_x^j u \p_x^{k-j+1} v + \p_x^j v \p_x^{k-j} \p_y v  )
\end{align}

We again treat two cases. First, assume that $j \le 6$, so $1 \le k-j+1 \le 11$. In this case, estimate by 
\begin{align} \n
\eps^{\frac{N_2}{2}} \| \frac{1}{\bar{u}} \p_x^j u \sqrt{\eps} \p_x^{k+\frac 1 2} v \langle x \rangle^{k + \frac 1 2} \phi_{12} \| \lesssim & \eps^{\frac{N_2}{2}}  \|  \frac{1}{\bar{u}} \p_x^j u \langle x \rangle^{j + \frac 1 4} \psi_{12} \|_\infty \| \sqrt{\eps} \p_x^{k-j+1} v \langle x \rangle^{k-j+ \frac 1 2} \phi_{12}\| \\ 
\lesssim & \eps^{\frac{N_2}{2}- M_1} \| U, V \|_{\mathcal{X}}^2, 
\end{align}
where we have invoked \eqref{L2:uv:eq} and \eqref{pw:dec:u}. 

Second, we assume that $1 \le k-j+1 \le 4$ and $j \ge 7$, in which case we estimate by 
\begin{align}
\sqrt{\eps} \| \frac{1}{\bar{u}} \p_x^j u \p_x^{k-j+1} v \langle x \rangle^{k + \frac 1 2} \phi_{12} \| \lesssim \sqrt{\eps} \| \frac{1}{\bar{u}} \p_x^{k-j+1} v \langle  x\rangle^{k-j+1 + \frac 1 2} \psi_{12} \|_\infty \| \p_x^j u \langle x \rangle^{j - \frac 1 2} \phi_{12}\|.
\end{align}

For the second contribution from \eqref{second:sum}, we again split into two cases. For the first case, we assume that $j \le 6$, in which case 
\begin{align}
\sqrt{\eps} \| \frac{1}{\bar{u}} \p_x^j v \p_x^{k-j} \p_y v \langle x \rangle^{k + \frac 1 2} \phi_{12}\| \lesssim \| \frac{1}{\bar{u}} \p_x^j v \langle x \rangle^{j + \frac 1 2} \psi_{12} \|_\infty \| \p_x^{k-j} u_x \langle x \rangle^{k-j+\frac 1 2} \phi_{12}\|. 
\end{align}
In the second case, we assume that $7 \le j \le 10$, in which case $k-j+1 \le 4$, and so we put 
\begin{align}
\sqrt{\eps} \| \frac{1}{\bar{u}} \p_x^{j} v \p_x^{k-j} \p_y v \langle x \rangle^{k + \frac 1 2} \phi_{12}  \| \lesssim \sqrt{\eps} \| \p_x^j v \langle x \rangle^{j - \frac 1 2}\phi_{12} \| \|\frac{1}{\bar{u}}  \p_x^{k-j+1} u \langle x \rangle^{k-j+1 + \frac 1 4} \psi_{12} \|_\infty. 
\end{align} 
This concludes the proof of the lemma. 
\end{proof}

We now establish the following corollary
\begin{corollary} The following estimate is valid: 
\begin{align} \label{Trilin:rest}
\Big| \sum_{k = 1}^{10}  \mathcal{T}_{X_k} + \sum_{k = 0}^{10} \mathcal{T}_{X_{k + \frac 1 2}} + \mathcal{T}_{Y_{k + \frac 1 2}} \Big| \lesssim \eps^{\frac{N_2}{2} - 2M_1-5} \| U, V \|_{\mathcal{X}}^3.  
\end{align}
\end{corollary}
\begin{proof} This is an immediate corollary of \eqref{N2:quant} and the definitions of $\mathcal{T}_{X_k},  \mathcal{T}_{X_{k + \frac 1 2}}$, and  $\mathcal{T}_{Y_{k + \frac 1 2}}$. 
\end{proof}

From the above analysis, the proof of the main theorem, Theorem \ref{thm:2}, is essentially immediate. Indeed, we have 
\begin{proof}[Proof of Theorem \ref{thm:2}]  We now add together estimates \eqref{basic:X0:est:st}, \eqref{Xh:right}, \eqref{basic:Yhalf:est:st}, \eqref{estXnnorm}, \eqref{esthalfnX}, \eqref{esthalfnY}, \eqref{energy:theta11}, \eqref{elliptic:1}, from which we obtain 
\begin{align}
\| U, V \|_{\mathcal{X}}^2 \le \sum_{k =0}^{11} \mathcal{F}_{X_k} + \sum_{k = 0}^{10} \mathcal{F}_{X_{k + \frac 1 2}} + \mathcal{F}_{Y_{k + \frac 1 2}} + \mathcal{F}_{Ell} +  \mathcal{T}.  
\end{align}
Appealing to estimate \eqref{est:T} and the established estimates on the forcing quantities, \eqref{est:forcings:part1} gives the main \textit{a-priori} estimate, which reads  
\begin{align}
\| U, V \|_{\mathcal{X}}^2 \lesssim \eps^5 +  \eps^{\frac{N_2}{2} - 2M_1-5} \| U, V \|_{\mathcal{X}}^3. 
\end{align}
From here, the existence and uniqueness follows from a standard contraction mapping argument. 
\end{proof}

\noindent \textbf{Acknowledgements:} S.I is grateful for the hospitality and inspiring work atmosphere at NYU Abu Dhabi, where this work was initiated. The work of S.I is partially supported by NSF grant DMS-1802940. The work of N.M. is supported by NSF grant DMS-1716466 and by Tamkeen under the NYU Abu Dhabi Research Institute grant
of the center SITE.

\def\bibindent{3.5em}

\end{document}